\documentclass[11pt,final,3p]{elsarticle}
\usepackage{amssymb}
\usepackage{amsthm}
\usepackage{amsmath}
\usepackage{latexsym}
\usepackage{amsmath}
\usepackage{xcolor}
\usepackage{graphicx}
\usepackage{indentfirst}
\usepackage[OT2,OT1]{fontenc}
\usepackage{mathrsfs}
\biboptions{numbers,sort&compress}
\usepackage{pdfsync}
\usepackage{hyperref}
\hypersetup{hypertex=true,
	colorlinks=true,
	linkcolor=blue,
	anchorcolor=blue,
	citecolor=blue}
\allowdisplaybreaks
\usepackage[utf8]{inputenc}

\renewcommand{\thefootnote}{\arabic{footnote}}
\newtheorem{theorem}{\color{black}\indent Theorem}[section]
\newtheorem{lemma}{\color{black}\indent Lemma}[section]

\newtheorem{definition}{\color{black}\indent Definition}[section]
\newtheorem{remark}{\color{black}\indent Remark}[section]

\numberwithin{equation}{section}

\newcommand\blfootnote[1]{%
	\begingroup
	\renewcommand\thefootnote{}\footnote{#1}%
	\addtocounter{footnote}{-1}%
	\endgroup
}
\journal{
}
\begin{document}
\begin{frontmatter}
\title{ Stochastic tamed 3D Navier–Stokes equations with locally weak monotonicity coefficients: existence, uniqueness and  averaging principle }

\author{{ \blfootnote{$^{*}$Corresponding author } Shuaishuai Lu$^{a}$ \footnote{ E-mail address : luss23@mails.jlu.edu.cn}
		,~ Xue Yang$^{a}$}  \footnote{E-mail address : xueyang@jlu.edu.cn},~ Yong Li$^{a,b,*}$  \footnote{E-mail address : liyong@jlu.edu.cn}\\
	{$^{a}$College of Mathematics, Jilin University,} {Changchun 130012, P. R. China.}\\
	{$^{b}$Center for Mathematics and Interdisciplinary Sciences, Northeast Normal University,}
	{Changchun 130024, P. R. China.}
}

\begin{abstract}
This paper investigates the stochastic tamed 3D Navier–Stokes equations with locally weak monotonicity coefficients in the whole space as well as in the three-dimensional torus, which play a crucial role in turbulent flows analysis. A significant issue is addressed in this work, specifically, the reduced  regularity of the coefficients and the inapplicability of Gronwall's lemma complicates the establishment of  pathwise uniqueness for weak solutions. Initially, the existence of a martingale solution for the system is established via Galerkin approximation; thereafter,  the pathwise uniqueness of this martingale solution is confirmed by constructing a specialized control function. Ultimately, the Yamada-Watanabe theorem is employed to establish the existence and uniqueness of the strong solution to the system. Furthermore, an averaging principle, referred to as the first Bogolyubov theorem, is established for stochastic tamed 3D Navier–Stokes equations with highly oscillating components, where the coefficients  satisfy the assumptions of linear growth and locally weak monotonicity. This result is achieved using classical Khasminskii time discretization, which illustrates the convergence of the solution from the original Cauchy problem to the averaged equation over a finite interval  [0, T].
~\\
~\\
 \textbf{keywords}: Averaging principle; Locally weak monotonicity; Stochastic tamed 3D Navier–Stokes equations; Yamada-Watanabe
theorem
 ~\\
 ~\\
  \textbf{MSC codes}: 35Q30, 60H15, 70K65
\end{abstract}
\end{frontmatter}
\section{\textup{Introduction}}
The Navier-Stokes systems represent fundamental equations in fluid mechanics, describing the motion of incompressible fluids. When randomness is introduced, as observed in atmospheric turbulence and uncertainties in financial markets, these  equations are reformulated as stochastic Navier-Stokes equations to describe turbulent incompressible flows, typically represented as follows.
\begin{align}\label{o111}
\begin{cases}
\text{d}u =[\nu \Delta  u-( u,\nabla  )  u+\nabla p_{1}+f]\text{d}t+[(\mathcal{K} \cdot \nabla )u+\nabla p_{2}+g]\text{d}W(t),\\
 \text{div} u(t)=0,\quad u(0)=u_{0},
 \end{cases}
	\end{align}
where $u$ is the unknow velocity field, $p_{1}$ and $p_{2}$ are the pressure, $\nu>0$ is the kinematic viscosity, $f$ is the external driving force, and $W(t)$ is  a cylindrical Wiener process, which reflects the influence of random factors in the environment on fluid motion. Furthermore, $\mathcal{K}$ is used to model turbulence. This equation illustrates that the fluid flow velocity varies over time, influenced by external forces and random perturbations. The stochastic 2D Navier-Stokes equations have been extensively studied, with  details available \cite{ref10,ref11,ref12,ref13,ref15,ref59,ref60,ref61,ref62,ref999} and among others. In three-dimensional scenarios, the classical Navier-Stokes equations are characterized by their nonlinearity and high complexity, which pose significant challenges, particularly in the study of turbulence. Specifically, the complexities related to  the nonlinear convection term, the failure of the Sobolev embedding theorem in higher dimensions, and the uniqueness of solution paths represent notable challenges. Collectively, these factors  present considerable challenges in addressing the stochastic three-dimensional Navier-Stokes equations. In view of this, Flandoli, Gatarek \cite{ref16} and Mikulevicius, Rozovskii \cite{ref17} established the existence of martingales solutions and steady-state solutions. Concurrently, in \cite{ref18,ref19}, Hofmanov$\acute{\text{a}} $ et al. demonstrated the non-uniqueness of solutions of the classical stochastic 3D Navier-Stokes system. Further results pertaining to the stochastic three-dimensional Navier-Stokes equations can be found in references \cite{ref52,ref53,ref55,ref56,ref57,ref63,ref65,ref66,ref69,ref70} and the associated literature.

Consequently,  the tame term is introduced, which holds considerable significance for the stochastic 3D Navier-Stokes equations in both theoretical and practical contexts. The taming technique is a modification of the Navier-Stokes equations designed to prevent solution blow-up by mitigating nonlinear effects in high-energy regions. The stochastic tamed  3D Navier-Stokes equation was first proposed by R$\ddot{\text{o}} $ckner and Zhang in \cite{ref1} and is defined as follows:
\begin{align}\label{o222}
\begin{cases}
\text{d}u =[\nu \Delta  u-( u,\nabla  )  u+\nabla p_{1}-\Psi _{N}(\left | u \right |^{2} ) u+f]\text{d}t+[(\mathcal{K} \cdot \nabla )u+\nabla p_{2}+g]\text{d}W(t),\\
 \text{div} u(t)=0,\quad u(0)=u_{0},
 \end{cases}
	\end{align}
where the tamed function $\Psi _{N}:\mathbb{R} ^{+}\to \mathbb{R} ^{+}$ is smooth and satisfies, for some $N\in \mathbb{N} $:
\begin{align*}
		\left\{\begin{matrix}
  \Psi _{N}(z)=0,& z\le N,\\
  \Psi _{N}=\frac{z-N}{\nu}, & z\ge N+1,\\
  0\le  \Psi' _{N}(z)\le \frac{2}{\nu \wedge 1},  &z\ge 0,
\end{matrix}\right.
	\end{align*}
and the notation $m\vee n$ denotes $\max\{m,n\} $ and  $m\wedge n$  signifies $\min\{m,n\} $. It is known that if a bounded smooth solution exists for the classical stochastic 3D Navier-Stokes equations, one can identify a sufficiently large $N$ to construct the tame term $\Psi _{N}$ such that the solution complies with equation \eqref{o222}. Conversely, by choosing $N$ sufficiently large, one can investigate the existence of weak solutions for system \eqref{o111}. Thus, another significant motivation for studying stochastic tame systems is to elucidate certain properties of system  \eqref{o111} through the analysis of system \eqref{o222}. The stochastic tame 3D Navier-Stokes  equations have been extensively investigated in the existing literature. For instance, in \cite{ref1}, R$\ddot{\text{o}} $ckner and Zhang proved the uniqueness of the strong solution of system \eqref{o222} under both full space and periodic boundary conditions, and  also demonstrated the uniqueness of the invariant measure in the context of periodic boundary conditions and degenerate additive noise. Subsequently, R$\ddot{\text{o}} $ckner et al. \cite{ref21} employed the weak convergence method to prove a Freidlin-Wentzell-type large deviation principle for \eqref{o222} driven by multiplicative noise. For the small time large deviation principle, R$\ddot{\text{o}} $ckner and Zhang  \cite{ref5} conducted further research based on Galerkin approximation method. Recently,  for stochastic tame 3D Navier-Stokes equations with fast oscillations, Hong et al. \cite{ref27} established a strong averaging principle  and demonstrated a Freidlin-Wentzell-type large deviation principle. Additional relevant properties regarding the 3D tame Navier-Stokes equations can also be found in references \cite{ref23,ref25,ref28,ref58}.
	
It is noteworthy that the aforementioned literature is established on the assumption that the diffusion and drift coefficients of the system satisfy the smoothness conditions such as Lipschitz continuity. However,  in practical fluid dynamics models, particularly within heterogeneous media or those characterized by high complexity (e.g., multiphase flow, flows with obstacles), the properties of the fluid (such as density and viscosity) frequently exhibit irregularities and lack smoothness. This implies that the formation and evolution of turbulence often rely on these irregular boundary conditions and the properties of the medium. The introduction of  locally weak monotonicity coefficients facilitates a more accurate characterization of these irregularities and non-uniformities, rendering the turbulence model more appropriate for actual complex environments.  In recent years, there has been an increasing interest in the development of stochastic models that exhibit non-Lipschitz continuity. For example, Fang and Zhang \cite{ref29} studied  stochastic  differential equations characterized by non-Lipschitz coefficients, specifically investigating the existence and uniqueness of strong solutions. In \cite{ref30}, Kulik and Scheutzow established the weak uniqueness of weak solutions for a class of stochastic functional differential equations with H$\ddot{\text{o}}$lder continuous coefficients through the method of generalized coupling. Han \cite{ref67} established the well-posedness and small mass limit for the stochastic wave equation with H$\ddot{\text{o}}$lder noise coefficient. Furthermore, R$\ddot{\text{o}} $ckner et al. \cite{ref31}  investigated  the averaging principle of semilinear slow-fast stochastic partial differential equations with additive noise, employing the Poisson equation in Hilbert space under the assumption of H$\ddot{\text{o}}$lder continuity for the fast variables. In \cite{ref999}, Agresti and Veraar considered H$\ddot{\text{o}}$lder coefficients for the transport noise, and obtained local well-posedness  for
the true Navier-Stokes equations. For further relevant properties, please refer to  \cite{ref32,ref33,ref35,ref38,ref39,ref99911,ref99922} and references therein.

To the best of our knowledge, the problem of stochastic tamed 3D Navier–Stokes equations with locally weak
monotonicity coefficients has not been studied in earlier works. Therefore, one of the primary objectives of this study is to establish the existence of a unique strong solution for system \eqref{o222}, under the condition that  the coefficients $f$ and  $g$ satisfy locally weak
monotonicity and linear growth conditions. The existing literature suggests that path uniqueness plays a crucial role in determining the well-posedness of (strong) solutions in stochastic differential systems. A classical result demonstrates that under Lipschitz conditions, the Gronwall lemma can be employed to establish the path uniqueness of weak solutions.  However, when the regularity of the coefficients falls below Lipschitz conditions, particularly under irregular conditions such as locally weak
monotonicity, this presents a significant challenge, as Gronwall lemma are not applicable. Consequently,  in the first step of this study, the Galerkin projection technique is employed to transform system \eqref{o222}  into a finite-dimensional system, thereby preliminarily establishing the existence of a weak solution for the finite-dimensional system. Subsequently, by constructing a specialized control function, the path uniqueness of the weak solution is established via a proof by contradiction. Furthermore, the Yamada-Watanabe theorem is employed to ascertain the existence and uniqueness of strong solutions. In the second step, referring to the conclusion (Theorem 3.8) in \cite{ref1}, the martingale solution of system \eqref{o222}  is derived by approximating the finite-dimensional strong solution, and similarly, the pathwise uniqueness  is established. Ultimately, the existence of a unique strong solution to the stochastic tamed  3D Navier-Stokes equations with locally weak monotonicity coefficients is  established through the Yamada-Watanabe theorem.

In traditional fluid dynamics, especially  when dealing with complex turbulent phenomena, the instantaneous behavior of the flow field is frequently intricate, exhibiting significant oscillatory components. These high-frequency oscillations pose considerable challenges for the direct analysis and simulation of system properties. Thus, it is very important to find a simplified system that can effectively control  the evolution of the  original system over long time scales. In this context, the averaging principle is an effective method. A fundamental aspect of the averaging principle is to ``average out" the highly oscillatory components under suitable conditions, thereby yielding an averaged system that, in a certain sense, replaces the original system. The concept of the averaging principle can be traced back to Bogoliubov's research on nonlinear  oscillation in deterministic systems, where Bogoliubov and Mitropolsky \cite{ref50} provided the first rigorous proof of the averaging principle within deterministic finite-dimensional systems. Subsequently Khasminskii \cite{ref51} extended the notion of the averaging principle to stochastic differential equations. The averaging principle for both finite and infinite-dimensional stochastic systems  has been extensively investigated in recent decades, building upon the foundational work  of Khasminskii. The interested readers might see \cite{ref68,ref71,ref72,ref73,ref75,ref9,ref76,ref78,ref79} and references therein for this direction.

Therefore, based on the above consideration, by systematically applying the averaging principle to the stochastic tamed  3D Navier-Stokes system, we can derive the statistical properties of the fluid and reveal its stationary state under random perturbations. The averaging principle provides a novel mathematical framework for elucidating and characterizing the global behavior of turbulence, while also facilitating a deeper understanding of some properties of stochastic 3D Navier-Stokes equations. Concurrently, it is recognized  that in numerous practical fluid dynamics problems, the physical characteristics of the fluid are frequently non-uniformly distributed. The locally weak
monotonicity conditions enable  the equation to adapt to the characteristics of the inhomogeneous medium. This adaptability allows  the analysis of the average behavior of the fluid within the inhomogeneous medium, thereby revealing the influence of irregularities on the statistical characteristics of the fluid. The second objective of this study is to establish an averaging principle for the following systems:
\begin{align}\label{o333}
		\begin{cases}
 \text{d}u^{\varepsilon}(t)=[\eta _{1}(\frac{t }{\varepsilon } ) \Delta  u-\eta _{2}(\frac{t }{\varepsilon })[( u,\nabla  )  u+\Psi _{N}(\left | u \right |^{2} ) u]+\nabla p_{1}+f(\frac{t }{\varepsilon },u)]\text{d}t\\~~~~~~~~~~+[(\mathcal{K} (\frac{t }{\varepsilon },x)\cdot \nabla )u+\nabla p_{2}+g(\frac{t }{\varepsilon },u)]\text{d}W(t),\\
 \text{div} u^{\varepsilon}(t)=0,\quad u(0)=u^{\varepsilon}_{0},
\end{cases}
	\end{align}
where  $\varepsilon$ is a small parameter and coefficients $f$, $g$   satisfy the conditions of  linear growth and locally weak
monotonicity with respect to $u$. Consequently, a fundamental question arises: what is the asymptotic behavior of the solution to the averaged system \eqref{o333} as the time scale $\varepsilon$ approaches zero? This is a challenging task because the locally weak
monotonicity coefficients can be quite irregular and the Gronwall lemma are not applicable. Therefore, we will continue to utilize a specialized control function and apply a proof by contradiction. In comparison to the existing conclusions, besides relaxing the coefficients to locally weak
monotonicity, we also consider the addition of these terms $\eta _{1}(\frac{t }{\varepsilon } )$, $\eta _{2}(\frac{t }{\varepsilon }) $, $\mathcal{K} (\frac{t }{\varepsilon },x)$, thereby enhancing the robustness of the turbulence model. However, this also introduces some difficulties. To address these problems, we will employ the classical Khasminskii time discretization method to establish the averaging principle for systems \eqref{o333}.  Under appropriate  conditions, we   derive the following conclusions:
\begin{align*}
		\lim_{\varepsilon  \to 0} \mathbb{E}\underset{t\in [0,T]}{\sup}   \left \|u^{\varepsilon }(t;u_{0} ^{\varepsilon }) -u^{* }(t;u_{0} ^{*}) \right \|^{2}_{\mathcal{H} ^{0}}=0,
	\end{align*}
for all  $T > 0$ provided that $\lim_{\varepsilon  \to 0} \mathbb{E} \left \|u_{0} ^{\varepsilon } -u_{0}  ^{*} \right \|^{2}_{\mathcal{H} ^{1}}=0$. Here $u^{* }$ denotes  the solution of the corresponding averaged equation (for
a detailed discussion, please refer to Theorem \ref{th2} below). If $\eta _{1}(\frac{t }{\varepsilon } )=1$,  $\eta _{2}(\frac{t }{\varepsilon }) =1$ and $\mathcal{K} (t ,x)=0$, we can derive a further assertion in the first-order Sobolev space, i.e, for all  $T > 0$, when $\lim_{\varepsilon  \to 0} \mathbb{E} \left \|u_{0} ^{\varepsilon } -u_{0}  ^{*} \right \|^{2}_{\mathcal{H} ^{1}}=0$,
\begin{align*}
		\lim_{\varepsilon  \to 0} \mathbb{E}\underset{t\in [0,T]}{\sup}   \left \|u^{\varepsilon }(t;u_{0} ^{\varepsilon }) -u^{* }(t;u_{0} ^{*}) \right \|^{2}_{\mathcal{H}^{1}}=0.
	\end{align*}

The remainder of this paper is structured as follows. In Section 2, we introduce  some standard notations and results and  necessary estimates. In Section 3, employing Galerkin projection techniques,  we rigorously prove the existence and uniquess of strong solutions for stochastic tamed 3D Navier-Stokes equations pertaining to turbulent incompressible flows characterized by locally weak
monotonicity coefficients. In Section 4, within the framework of the locally weak
monotonicity conditions, we establish the strong averaging principle for the stochastic tamed 3D Navier-Stokes equations with highly oscillating components in  spaces $\mathcal{H}^{0}$ and $\mathcal{H}^{1}$, which can be viewed as the functional
law of large numbers.

	\section{\textup{Preliminaries}}
	Let $\mathbb{D}=\mathbb{R} ^{3}$ or $\mathbb{T}^{3}$(representing the standard torus) and $C^{\infty }_{0}(\mathbb{D};\mathbb{R} ^{3}  )$ denote the set of all smooth functions from $\mathbb{D}$ to $\mathbb{R} ^{3}$ with compact
supports.  We write $L^{p}(\mathbb{D}; \mathbb{R} ^{3} )(p\ge 1)$ to be the vector valued $L^{p}$-space in $\mathbb{D}$, whose norm is denoted by $\left \| \cdot  \right \| _{L^{p}}$. Specifically, when $p=2$, the space $L^{2}(\mathbb{D};\mathbb{R} ^{3}  )$ forms a real Hilbert space equipped with an inner product: $$\left \langle u,v \right \rangle _{L^{2}}=\int _{\mathbb{D} }u(x)v^{T}(x)\text{d}x.$$
The corresponding norm is given by  $\left \| u \right \| _{L^{2}}=\left \langle u,u \right \rangle _{L^{2}}$, where $v^{T}$ denotes the transposition of the row vector $v$. Let $W^{m,2}(\mathbb{D};\mathbb{R} ^{3}  )$ represent the Sobolev space on $\mathbb{D}$, equipped with the norm $$\left \| u \right \| _{W^{m,2}}^{2}=\int _{\mathbb{D} }\left | (I-\Delta )^{\frac{m}{2} }u(x) \right |\text{d}x,$$
where $m=0,1,2,...$, and $(I-\Delta )^{\frac{m}{2} }$ is defined by Fourier transformation. Define $$\mathcal{H}^{m} =\left \{ u\in W^{m,2}(\mathbb{D};\mathbb{R}^{2}  ):\text{div}u=0 \right \}, $$
where ``div" denotes the divergence operator. Then the $(\mathcal{H}^{m},\left \| \cdot  \right \| _{\mathcal{H}^{m}})$ represents a separable Hilbert space.  Let $\left ( \Omega ,\mathscr{F},\mathbb{P} \right )$ be a complete probability space, equipped with a filtration  $\lbrace\mathscr{F}_{t}\rbrace_{t\ge0}$ that satisfies the usual conditions. We consider the following the stochastic tamed 3D Navier–Stokes equations:
\begin{align}\label{r1}
		\begin{cases}
 \text{d}u(t)=[\nu \Delta  u-( u,\nabla  )  u+\nabla p_{1}(t,x)-\Psi _{N}(\left | u \right |^{2} ) u+f(t,u)]\text{d}t\\~~~~~~~~~~+[(\mathcal{K} (t,x)\cdot \nabla )u+\nabla p_{2}(t,x)+g(t,u)]\text{d}W(t),\\
 \text{div} u(t)=0,\quad u(0)=u_{0},
\end{cases}
	\end{align}
 where $\nu>0$ represents the kinematic viscosity and  $p_{1}$, $p_{2}$ are unknown scalar functions. Additionally, let $W(t)$ be a  cylindrical Wiener processes  on a separable Hilbert space $K$ with respect to a complete
filtered probability space $\left ( \Omega ,\mathscr{F},\mathbb{P} \right )$. For a fixed time horizon $T$, $$f:[0,T]\times L^{2}(\mathbb{D} )\times \Omega\to L^{2}(\mathbb{D} ),\quad \mathcal{K} :[0,T]\times \mathbb{D}    \to \mathscr{L}(K,L^{2}(\mathbb{D} )) \quad g:[0,T]\times L^{2}(\mathbb{D} )\times \Omega  \to \mathscr{L}(K,L^{2}(\mathbb{D} )),$$
are measurable mappings, where $\mathscr{L}(K,L^{2}(\mathbb{D} ))$ denotes   the space of all Hilbert-Schmidt operators from $K$ into $L^{2}(\mathbb{D} )$.

Let $\Pi $ denote the orthogonal projection from $L^{2}(\mathbb{D} )$ to $\mathcal{H}^{0}$. It is well known that $\Pi $ commutes with the differential  operators.  We define the Stokes operator $S_{1}$  as follows:
$$S_{1}:W^{2,2}(\mathbb{D};\mathbb{R}^{2}  )\cap \mathcal{H}^{1}\to \mathcal{H}^{0}, \quad S_{1}(u)=\Pi\Delta  u$$
and define
$$S^{1}_{2}:\mathcal{D}(S_{2})\subset \mathcal{H}^{0}\times \mathcal{H}^{1}\to \mathcal{H}^{0}, \quad S_{2}(u)=S^{1}_{2}(u,u)+\Pi\Psi _{N}(\left | u \right |^{2} ) u,\quad S^{1}_{2}(u,v)=\Pi( u,\nabla  )  v. $$
Applying the operator $\Pi$ to both sides of  \eqref{r1}, the Navier–Stokes equations can be reformulated in the following abstract form:
\begin{align}\label{t1}
		\begin{cases}
 \text{d}u(t)=[\nu S_{1}(u)-S_{2}(u)+F(t,u)]\text{d}t+G(t,u)\text{d}W(t),\\
  u(0)=u_{0},
\end{cases}
	\end{align}
where $F(t,u):=\Pi f(t,u)$ and $G(t,u):=\Pi (\mathcal{K} (t,x)\cdot \nabla )u+\Pi g(t,u)$. Hence. for any $u\in \mathcal{H}^{0}$ and $v\in L^{2}(\mathbb{D} )$,  $$\left \langle u,v \right \rangle _{\mathcal{H}^{0} }:=\left \langle u,\Pi v \right \rangle_{\mathcal{H}^{0} } =\left \langle u,v  \right \rangle_{L^{2}} .$$
\section{\textup{Existence and uniqueness }}
Let $\mathcal{B}(X)$ denote the $\sigma $-algebra generated by space $X$, $\mathcal{P}(X) $ be the family of all probability measures  defined on $ \mathcal{B}(X)$. We begin by introducing the following concept of weak solutions of \eqref{t1}.
\begin{definition}\label{de1} \cite{ref1} We describe \eqref{t1} as possessing a weak solution  if there exist a stochastic basis $\left ( \Omega ,\mathscr{F},\mathbb{P} \right )$, a process $u(t)$ adapted to $\lbrace\mathscr{F}_{t}\rbrace_{t\ge0}$ taking values in $\mathcal{H} ^{1}$ such that
\begin{enumerate}[(\textbf{I})]
        \item $u\in C([0,T];\mathcal{H} ^{1})\cap L^{2}([0,T];\mathcal{H} ^{2})$, $\mathbb{P}$-a.s., for any $T>0$;
	\end{enumerate}
\begin{enumerate}[(\textbf{II})]
        \item It holds that in $\mathcal{H} ^{0}$
        \begin{align*}
        u(t)=u_{0}+\int_{0}^{t} [\nu S_{1}(u(s))-S_{2}(u(s))+F(s,u(s))]\text{d}s+\int_{0}^{t}G(s,u(s))\text{d}W(s) ,
        \end{align*}
        for all $t\in [0,T]$, $\mathbb{P}$-a.s.
	\end{enumerate}
\end{definition}
The investigation of the existence and uniqueness of solutions to equation \eqref{t1},  characterized by locally weak
monotonicity coefficients, assumes that the initial value  $u_{0}$ to be independent of  \(W(t)\). Initially, we posit that the coefficients in \eqref{t1} satisfy the following hypotheses:
 ~\\
\\\textbf{(H1)} For continuous functions $f$, $g$, there exist constants $C$ and  $M$  such that for all $t\in[0,T]$ and $u\in L^{2}(\mathbb{D} ;\mathbb{R}^{3})$,
\begin{align*}
		\left \langle u,f(t,u) \right \rangle _{L^{2}} \vee \left \| g(t,u ) \right \| ^{2}_{\mathscr{L}(K,L^{2})}  \le C \left \| u\right \|^{2}_{L^{2} }+M;
	\end{align*}
and for any $t\in[0,T]$ and $u\in \mathcal{H}^{1}  (\mathbb{D} ;\mathbb{R}^{3})$,
\begin{align*}
		\left \langle u,f(t,u) \right \rangle _{\mathcal{H}^{1}} \vee \left \| g(t,u ) \right \| ^{2}_{\mathscr{L}(K,\mathcal{H}^{1})}  \le C \left \| u\right \|^{2}_{\mathcal{H}^{1} }+M.
	\end{align*}
\textbf{(H2)}  Let $\mathcal{A}:[0,1)\to \mathbb{R} _{+} $ is an increasing, concave and continuous function satisfying $$\mathcal{A}(0)=0, \int_{0^{+}}\frac{\text{d}r}{\mathcal{A}(r)}=+\infty.$$    The functions $f$, $g$ satisfy,  for all $t\in[0,T]$ and $u,v\in L^{2}(\mathbb{D} ;\mathbb{R}^{3})  $ with $\left \| u-v \right \| _{L^{2}}\le\zeta (\zeta\in(0,1))$,
\begin{align*}
		\langle u-v, f(t,u)- f(t,v )\rangle_{L^{2}}  \le c\mathcal{A}(\left \| u -v  \right \|^{2} _{L^{2}}),\quad  \left \| g(t,u )- g(t,v ) \right \|^{2} _{\mathscr{L}(K,L^{2})} \le c\mathcal{A}(\left \| u -v  \right \|^{2} _{L^{2}}),
	\end{align*}
and for all $t\in[0,T]$ and $u,v\in \mathcal{H}^{1}(\mathbb{D} ;\mathbb{R}^{3})  $ with $\left \| u-v \right \| _{\mathcal{H}^{1}}\le\zeta $,
\begin{align*}
\langle u-v, f(t,u)- f(t,v )\rangle_{\mathcal{H}^{1}}  \le c\mathcal{A}(\left \| u -v  \right \|^{2} _{\mathcal{H}^{1}}),\quad  \left \| g(t,u )- g(t,v ) \right \|^{2} _{\mathscr{L}(K,\mathcal{H}^{1})} \le c\mathcal{A}(\left \| u -v  \right \|^{2} _{\mathcal{H}^{1}}).
	\end{align*}
\textbf{(H3)} For any $T > 0$, there exists a constant $C$ such that for any $x,y\in \mathbb{D} $
\begin{align*}
		\sup_{t\in[0,T],x\in \mathbb{D} }\left | \partial x_{i}\mathcal{K} (t,x) \right | \le C,\quad \sup_{t\in[0,T],x\in \mathbb{D} } \left | \mathcal{K} (t,x) \right |^{2}=a^{*} \le
\left\{\begin{matrix}
  \frac{\nu  }{73},& p\in [1,\frac{75}{2}  ],\\
  \frac{p\nu -\epsilon }{2p(p-1)}, &p>\frac{75}{2} ,
\end{matrix}\right.
\end{align*}
where $i=1,2$ and $\epsilon >0 $ is arbitrarily small.

In this paper, $C_{T}$ represents certain positive constants dependent on $T$, which may vary from line to line. In addition to the aforementioned assumptions, the following two inequalities will be extensively employed. These inequalities play a pivotal role in the analysis of the Navier-Stokes equations (see  \cite{ref2,ref3,ref5}):
\begin{enumerate}[(1)]
		\item Let $q\in[1,\infty ]$ and $m\in\mathbb{N} $. If
\begin{align*}
		\frac{1}{q} =\frac{1}{2} -\frac{m\theta }{3}, \quad \theta\in[0,1],
	\end{align*}
then, for any $u\in W^{m,2}$,
\begin{align}\label{k1}
		\left \| u \right \| _{L^{q}}\le C_{q,m}\left \| u \right \|^{\theta }_{W^{m,2}} \left \| u \right \|^{1-\theta }_{L^{2}}.
	\end{align}
\end{enumerate}
\begin{enumerate}[(2)]
        \item \begin{align}\label{k2}
		\sup_{x\in \mathbb{D} } \left | u(x) \right | ^{2}\le C\left \| \Delta u  \right \|_{L^{2}} \left \| \nabla u \right \|_{L^{2}}.
	\end{align}
	\end{enumerate}

A considerable body of research has been devoted to investigating the existence and uniqueness of solutions under the condition that the system coefficients satisfy Lipschitz continuity. Conversely, research concerning coefficients with lower regularity, such as those meeting only locally weak
monotonicity conditions, remains relatively scarce.  We investigate the existence and uniqueness of solutions, as well as additional asymptotic properties of equation \eqref{t1}, employing the Galerkin-type approximation technique.
\begin{theorem}\label{th1} Consider \eqref{t1}. Under assumptions \textbf{(H1)}$-$\textbf{(H3)}, for any initial value $u_{0} \in \mathcal{H} ^{1}$, there exists a unique solution $u(t,x)\in L^{2}(\Omega ; C([0,T],\mathcal{H} ^{1}))\cap L^{2}(\Omega ; L^{2}([0,T],\mathcal{H} ^{2}))$ for \eqref{t1} in the sense of Definition \ref{de1} such that for any $p\ge 1$ and $T>0$,
\begin{align}\label{t111}
		& \mathbb{E} (\sup_{t\in[0,T]}\left \| u(t) \right \|_{\mathcal{H}^{1} }^{2p})+\mathbb{E}\int_{0}^{T}\left \| u(t) \right \|_{\mathcal{H}^{1} }^{2p-2}[ \left \| u(s)\right \| ^{2}_{\mathcal{H}^{2}}+  \left \| \left | u \right |\left | \nabla u \right | \right \|^{2}_{\mathcal{H}^{0}}\text{d}s\nonumber
\\&\le C_{\nu,T,M,\epsilon,p,N}(\left \| u_{0} \right \|_{\mathcal{H}^{1} }^{2p}+1).
	\end{align}
\end{theorem}
\begin{proof} To establish the existence and uniqueness of a strong solution for system \eqref{t1} characterized by locally weak
monotonicity coefficients, we delineate the proof into two distinct steps:
~\\
\\\textbf{Step1: Galerkin’s approximation}

We commence by employing the Galerkin projection technique to convert system \eqref{t1} into a finite-dimensional framework. Let  $$B=\left \{ u\in C^{\infty }_{0}(\mathbb{D};\mathbb{R} ^{3}  ):\text{div}u=0 \right \},$$ and assume the existence of orthonormal bases  $\left \{ e _{1},e _{2},e  _{3} ,...\right \} \subset B$ for $\mathcal{H}^{1} $ and $\left \{ \varrho_{1}, \varrho_{2}, \varrho_{3}, ..., \right \}$ for $K$. By selecting the first $k$ orthonormal bases from each set, we define the following operators:
\begin{align*}
\begin{split}
		\Lambda _{1}^{k}:B^{*}\to \mathcal{H}^{1}_{k}:=\text{span}\{e _{1},e  _{2},...,e  _{k} \},\quad \Lambda _{2}^{k}:K\to K^{k}:=\text{span}\{\varrho   _{1},\varrho   _{2},...,\varrho   _{k} \}.
\end{split}
	\end{align*}
For any $u\in B^{*}$ and $k\ge1$, we derive $u^{k}=\Lambda  _{1}^{k}(u)=\sum_{i=1}^{k} \langle u, e  _{i}\rangle _{\mathcal{H}^{1} }e  _{i}$ and  $W^{k}(t)=\Lambda  _{2}^{k}[W(t)]=\sum_{i=1}^{k}\left \langle W(t),\varrho   _{i} \right \rangle _{K}\varrho   _{i}$. Building upon the aforementioned approximation techniques, we will first analyze the finite-dimensional stochastic ordinary differential equation within $\mathcal{H}^{1} _{k}$:
\begin{align}\label{t2}
		\begin{cases}
 \text{d}u^{k}(t)=\Lambda _{1}^{k}[\nu S_{1}(u)-S_{2}(u)+F(t,u)]\text{d}t+\Lambda _{1}^{k}G(t,u)\text{d}W^{k}(t), \\
u^{k}(0) =\Lambda _{1}^{k}u_{0}.
\end{cases}
	\end{align}
According to the definition of operators $S_{1}$, $S_{2}$ we ascertain the following for any $u\in \mathcal{H}^{2}$
\begin{align}\label{t3}
 \left \langle u ,\nu S_{1}(u) \right \rangle _{\mathcal{H}^{1}}&= \left \langle (I-\Delta )u ,\nu\Delta u \right \rangle _{\mathcal{H}^{0}  } \nonumber
 \\&=-\nu\left \| (I-\Delta ) u \right \| ^{2}_{\mathcal{H}^{0} }+\nu\left \langle (I-\Delta )u , u \right \rangle _{\mathcal{H}^{0}}\\&\le-\nu\left \| u\right \| ^{2}_{\mathcal{H}^{2}} +C_{\nu}\left \| u  \right \|^{2}_{\mathcal{H}^{1} } . \nonumber
	\end{align}
Furthermore, applying Young's inequality,
\begin{align}\label{t5}
 -\left \langle u ,S^{1}_{2}(u) \right \rangle _{\mathcal{H}^{1} }&\le \frac{\nu }{2}  \left \| (I-\Delta ) u \right \| ^{2}_{\mathcal{H}^{0} }+ \frac{1 }{2\nu}  \left \| (u ,\nabla )u \right \|^{2}_{\mathcal{H}^{0}}\nonumber
\\ &\le \frac{\nu }{2} \left \| u\right \| ^{2}_{\mathcal{H}^{2}} + \frac{1 }{2\nu}  \left \| \left | u \right |\left | \nabla u \right | \right \|^{2}_{\mathcal{H}^{0}}.
	\end{align}
For $\left \langle u ,\Psi _{N}(\left | u \right |^{2} )u \right \rangle _{\mathcal{H}^{1} }$, according to the definition of function $\Psi _{N}$, refer to Lemma 2.3 of \cite{ref1}, we obtain that
\begin{align}\label{t555}
-\left \langle u ,\Psi _{N}(\left | u \right |^{2} )u \right \rangle _{\mathcal{H}^{1} }&=-\left \langle (I-\Delta )u ,\Psi _{N}(\left | u \right |^{2} )u \right \rangle _{\mathcal{H}^{0} }\nonumber
\\&=-\int _{\mathbb{D} }\left | u \right |^{2}\Psi _{N}(\left | u \right |^{2} )\text{d}x-\sum_{n,i=1}^{3} \int _{\mathbb{D} }\partial _{i}u^{n}\partial _{i}[\Psi _{N}(\left | u \right |^{2} )u^{n}]\text{d}x\nonumber
\\&\le -\sum_{n,i=1}^{3} \int _{\mathbb{D} }\partial _{i}u^{n}[\Psi _{N}(\left | u \right |^{2} )\partial _{i}u^{n}+\Psi' _{N}(\left | u \right |^{2} )\partial _{i}\left | u \right |^{2}u^{n}]\text{d}x
\\&=-\int _{\mathbb{D} }\left | \nabla u \right |^{2}\Psi _{N}(\left | u \right |^{2} )\text{d}x-\frac{1}{2} \int _{\mathbb{D} }\Psi' _{N}(\left | u \right |^{2} )\left |\nabla  \left |  u\right |^{2}  \right |^{2}\text{d}x\nonumber
\\&\le -\frac{1 }{\nu}  \left \| \left | u \right |\left | \nabla u \right | \right \|^{2}_{\mathcal{H}^{0}}+\frac{N }{\nu}\left \| u  \right \|^{2}_{\mathcal{H}^{1} }.\nonumber
	\end{align}
In addition, for $G(t,u)$, we obtain
\begin{align}\label{t7}
 \left \| G(t,u) \right \|^{2}_{\mathscr{L}(K,\mathcal{H}^{1})}&= \left \| G(t,u) \right \|^{2}_{\mathscr{L}(K,\mathcal{H}^{0})}+ \left \| \nabla G(t,u) \right \|^{2}_{\mathscr{L}(K,\mathcal{H}^{0})}\nonumber
 \\&\le \int _{\mathbb{D} }\left \| \mathcal{K} (t,x) \right \| _{\mathscr{L}(K,U)}^{2}\left \| \nabla u \right \|_{\mathcal{H}^{0} }^{2}\text{d}x +C_{T}\left \| u \right \|_{\mathcal{H}^{0}}^{2}\nonumber
 \\&~~~+\int _{\mathbb{D} }[(\nabla \mathcal{K} (t,x),\nabla)u+(\mathcal{K} (t,x),\nabla)\nabla u]\text{d}x+C_{T}\left \| u \right \|_{\mathcal{H}^{1} }^{2}+M\\&\le\frac{\nu  }{73}\left \| u \right \| ^{2}_{\mathcal{H}^{2} }+C_{\nu,T}\left \| u  \right \|^{2}_{\mathcal{H}^{1} }+M.\nonumber
	\end{align}
Consequently, the assumptions \textbf{(H1)}, \eqref{t3},\eqref{t5},\eqref{t555} and \eqref{t7} imply that for any $\chi \in \mathcal{H}^{1}_{k}$,
\begin{align}\label{t6}
 2\left \langle \chi ,\Lambda _{1}^{k}[\nu S_{1}(\chi )-S_{2}(\chi )+F(t,\chi )] \right \rangle _{\mathcal{H}^{1}_{k} }+\left \| \Lambda _{1}^{k}G(t,\chi) \right \|^{2}_{\mathscr{L}(K^{k},\mathcal{H}^{1}_{k})}\le C_{k,T,\nu,M,N}(\left \| \chi  \right \|^{2}_{\mathcal{H}^{1}_{k} } +1).
	\end{align}
Furthermore, based on assumption \textbf{(H2)}, we establish that the mappings $\Lambda _{1}^{k}[\nu S_{1}(\chi )-S_{2}(\chi )+F(t,\chi )]$ and $\Lambda _{1}^{k}G(t,u)$ satisfy locally  weak monotonicity conditions, i.e., for any $\chi_{1,2} \in \mathcal{H}^{1}_{k}$ with $\left \| \chi_{1}-\chi_{2} \right \| _{\mathcal{H}^{1}_{k}}\le\zeta $,  we have
\begin{align}\label{t11}
 &\left \langle \chi_{1}-\chi_{2} ,\Lambda _{1}^{k}[\nu S_{1}(\chi_{1} )-S_{2}(\chi_{1} )+F(t,\chi_{1} )]-\Lambda _{1}^{k}[\nu S_{1}(\chi_{2} )-S_{2}(\chi_{2} )+F(t,\chi_{2} )] \right \rangle _{\mathcal{H}^{1}_{k} }\nonumber
 \\&\le C_{k,T,\nu}[\left \| \chi_{1} -\chi_{2} \right \|_{\mathcal{H}^{1}_{k}}^{2}+\mathcal{A} (\left \| \chi_{1} -\chi_{2} \right \|_{\mathcal{H}^{1}_{k}}^{2})],
	\end{align}
and
\begin{align}\label{t12}
 \left \| \Lambda _{1}^{k}G(t,\chi_{1})-\Lambda _{1}^{k}G(t,\chi_{2}) \right \|^{2}_{\mathscr{L}(K^{k},\mathcal{H}^{1}_{k})}\le  C_{k,T,\nu}[\left \| \chi_{1} -\chi_{2} \right \|_{\mathcal{H}^{1}_{k}}^{2}+\mathcal{A} (\left \| \chi_{1} -\chi_{2} \right \|_{\mathcal{H}^{1}_{k}}^{2})].
	\end{align}
For the sake of clarity in the subsequent discussion, we provide the following definition:
\begin{align*}
 \mathcal{F} (t,\chi):=\Lambda _{1}^{k}[\nu S_{1}(\chi )-S_{2}(\chi )+F(t,\chi ),\quad \mathcal{G} (t,\chi):=\Lambda _{1}^{k}G(t,\chi).
	\end{align*}
Let the support of $\rho \in C^{\infty }(\mathbb{R} )$ be contained in $\mathcal{O} ^{\mathcal{M} }=\{\chi\in \mathcal{H}^{1}_{k}:\left \| \chi  \right \| _{\mathcal{H}^{1}_{k}}\le \mathcal{M} \}$, i.e.,
\begin{align*}
\left\{\begin{matrix}
 \rho (\left \| \chi  \right \|_{\mathcal{H}^{1}_{k}} )=1, &\chi \in \mathcal{O} ^{\mathcal{M} }, \\
   \rho (\left \| \chi  \right \|_{\mathcal{H}^{1}_{k}} )=0, &\chi \notin  \mathcal{O} ^{\mathcal{M} }.&
\end{matrix}\right.
\end{align*}
For any $\chi\in \mathcal{H}^{1}_{k}$, let $\mathcal{F}^{\mathcal{M} } (t,\chi)=\mathcal{F} (t,\chi)\rho (\left \| \chi  \right \|_{\mathcal{H}^{1}_{k}} )$ and $\mathcal{G}^{\mathcal{M} } (t,\chi)=\mathcal{G} (t,\chi)\rho (\left \| \chi  \right \|_{\mathcal{H}^{1}_{k}} )$, which implies that $\mathcal{F}^{\mathcal{M} } (t,\chi)$, $\mathcal{G}^{\mathcal{M} } (t,\chi)$ are uniformly bounded on $\mathcal{H}^{1}_{k}$ and  locally weak
monotonicity. By Proposition 3.3 of \cite{ref67}, there exist two sequences of uniformly bounded local Lipschitz continuous functions $\mathcal{F}^{\mathcal{M} }_{m} (t,\chi)$ and $\mathcal{G}^{\mathcal{M} } _{m}(t,\chi)$ such that
 \begin{align}\label{t9}
\left \| \mathcal{F}^{\mathcal{M} }_{m} (t,\chi)-\mathcal{F}^{\mathcal{M} }(t,\chi) \right \|_{\mathcal{H}^{1}_{k}} \to 0,\quad \left \| \mathcal{G}^{\mathcal{M} }_{m} (t,\chi)-\mathcal{G}^{\mathcal{M} }(t,\chi) \right \|_{\mathscr{L}(K^{k},\mathcal{H}^{1}_{k})} \to 0,\quad \text{as}\quad m\to\infty,
\end{align}
 uniformly on each  bounded subset of $\mathcal{H}^{1}_{k}$ and $t\in[0,T]$.

 Furthermore, it is readily apparent that $\mathcal{F}^{\mathcal{M} }_{m} (t,\chi)$ and $\mathcal{G}^{\mathcal{M} } _{m}(t,\chi)$  satisfy  \eqref{t6} for any $m\ge 1$. Consequently, by the theory of stochastic differential equations (SDE), there exists a unique global strong solution $\chi _{m}(t)$ satisfying
\begin{align}\label{aaaaaa1}
\chi_{m}(t)=u^{k}_{0}+\int_{0}^{t} \mathcal{F}^{\mathcal{M} }_{m} (s,\chi_{m}(s))\text{d}s+\int_{0}^{t} \mathcal{G}^{\mathcal{M} }_{m} (s,\chi_{m}(s))\text{d}W^{k}(s).
\end{align}
Indeed, the weak limit point of  $\chi_{m}(t)$ as \( m \to \infty \) constitutes a weak solution $\chi(t)$ to the following system:
\begin{align}\label{t10}
\begin{cases}
 \text{d}\chi(t)=\mathcal{F}^{\mathcal{M} }(t,\chi)\text{d}t+\mathcal{G}^{\mathcal{M} }(t,\chi)\text{d}W^{k}(t), \\
\chi(0)=u^{k}_{0}.
\end{cases}
\end{align}
This proof adheres to a standard approach; for a detailed exposition, please refer to \textbf{Appendix \uppercase\expandafter{\romannumeral1}}.

We shall now proceed to establish the pathwise uniqueness of these weak solutions. Let us consider two stochastic processes $\chi_{1}(t)$ and $\chi_{2}(t)$ that satisfy the following form:
\begin{align*}
		\begin{cases}
\chi_{1}(t)= \mathcal{F}^{\mathcal{M} }(t,\chi_{1}(t))\text{d}s+ \mathcal{G}^{\mathcal{M} }(t,\chi_{1}(t))\text{d}W^{k}(t), \\
\chi_{1}(0)=\chi^{0}_{1}\in \mathcal{H}^{1}_{k},
\end{cases}
	\end{align*}
and
\begin{align*}
		\begin{cases}
\chi_{2}(t)= \mathcal{F}^{\mathcal{M} }(t,\chi_{2}(t))\text{d}s+ \mathcal{G}^{\mathcal{M} }(t,\chi_{2}(t))\text{d}W^{k}(t), \\
\chi_{2}(0)=\chi^{0}_{2}\in \mathcal{H}^{1}_{k}.
\end{cases}
	\end{align*}
Without loss of generality, we assume that $\left \| \chi^{0}_{1}-\chi^{0}_{2} \right \| _{\mathcal{H}^{1}_{k}}<1$ and $\beta   \in(\left \| \chi^{0}_{1}-\chi^{0}_{2} \right \| _{\mathcal{H}^{1}_{k}},1 ]$. Let us denote
\begin{align*}
		\tau _{R}=\underset{t\ge0}{\inf} \{\left \| \chi_{1}(t) \right \| _{\mathcal{H}^{1}_{k}}\vee \left \| \chi_{2}(t) \right \| _{\mathcal{H}^{1}_{k}}>R \},\quad \tau _{\beta }=\underset{t\ge0}{\inf} \{\left \| \chi_{1}(t)-\chi_{2} (t)\right \| _{\mathcal{H}^{1}_{k}}>\beta  \},
	\end{align*}
where $\left \| \chi^{0}_{1}\right \| _{\mathcal{H}^{1}_{k}}\vee \left \| \chi^{0}_{2} \right \| _{\mathcal{H}^{1}_{k}}<R$. Indeed, we assert that the following property holds:
\begin{align}\label{r12}
		 \lim_{\left \| \chi^{0}_{1}-\chi^{0}_{2} \right \| _{\mathcal{H}^{1}_{k}} \to 0} \mathbb{E} (\underset{r\in[0,t]}{\sup}\left \| \chi_{1}(r\wedge \tau _{R }\wedge \tau _{\beta  } )-\chi_{2}(r\wedge \tau _{R }\wedge \tau _{\beta } ) \right \|_{\mathcal{H}^{1}_{k}} ^{2})=0.
	\end{align}
By It$\hat{\text{o}} $'s formula, Young's inequality, the B-D-G inequality and \eqref{t11}, \eqref{t12}, we obtain
\begin{align}\label{r8}
		 & \mathbb{E} (\underset{r\in[0,t]}{\sup}\left \| \chi_{1}(r\wedge \tau _{R }\wedge \tau _{\beta  } )-\chi_{2}(r\wedge \tau _{R }\wedge \tau _{\beta } ) \right \| _{\mathcal{H}^{1}_{k}}^{2}) \nonumber
\\&=\left \| \chi^{0}_{1}-\chi^{0}_{2} \right \| _{\mathcal{H}^{1}_{k}}^{2}+\mathbb{E}\int_{0}^{t\wedge \tau _{R }\wedge \tau _{\beta }} [2\left \langle \chi_{1}(s)-\chi_{2}(s),\mathcal{F}^{\mathcal{M} }(t,\chi_{1}(s))-\mathcal{F}^{\mathcal{M} }(t,\chi_{2}(s)) \right \rangle _{\mathcal{H}^{1}_{k}}\nonumber \\&~~~+\left \| \mathcal{G}^{\mathcal{M} }(t,\chi_{1}(s))-\mathcal{G}^{\mathcal{M} }(t,\chi_{2}(s)) \right \|_{\mathscr{L}(K^{k},\mathcal{H}^{1}_{k})} ^{2}]\text{d}s   \nonumber
\\&~~~+2\mathbb{E} (\underset{r\in[0,t]}{\sup}\int_{0}^{r\wedge \tau _{R }\wedge \tau _{\beta }}[\chi_{1}(s)-\chi_{2}(s)]^{T}[ \mathcal{G}^{\mathcal{M} }(t,\chi_{1}(s))-\mathcal{G}^{\mathcal{M} }(t,\chi_{2}(s)) ]\text{d}W^{k}(s))
\\&\le \left \| \chi^{0}_{1}-\chi^{0}_{2} \right \| _{\mathcal{H}^{1}_{k}}^{2}+C_{k,T,\nu }\mathbb{E}\int_{0}^{t\wedge \tau _{R }\wedge \tau _{\beta }}[\mathcal{A} (\left \| \chi _{1}(s)-\chi _{2}(s) \right \|^{2} _{\mathcal{H}^{1}_{k}})+\left \| \chi _{1}(s)-\chi _{2}(s) \right \|^{2 } _{\mathcal{H}^{1}_{k}}]\text{d}s\nonumber
\\&~~~+\frac{1}{2} \mathbb{E} (\underset{r\in[0,t]}{\sup}\left \| \chi_{1}(r\wedge \tau _{R }\wedge \tau _{\beta  } )-\chi_{2}(r\wedge \tau _{R }\wedge \tau _{\beta } ) \right \|_{\mathcal{H}^{1}_{k}} ^{2}).\nonumber
	\end{align}
Therefore, by applying Jensen’s inequality, we have
\begin{align*}
		 & \mathbb{E} (\underset{r\in[0,t]}{\sup}\left \| \chi_{1}(r\wedge \tau _{R }\wedge \tau _{\beta  } )-\chi_{2}(r\wedge \tau _{R }\wedge \tau _{\beta } ) \right \|_{\mathcal{H}^{1}_{k}} ^{2}) \\&\le \left \| \chi^{0}_{1}-\chi^{0}_{2} \right \| _{\mathcal{H}^{1}_{k}}^{2}+C_{k,T,\nu }\int_{0}^{t}[\mathcal{A} (\mathbb{E} \underset{r\in[0,s]}{\sup}\left \| \chi _{1}(s\wedge \tau _{R }\wedge \tau _{\beta })-\chi _{2}(s\wedge \tau _{R }\wedge \tau _{\beta }) \right \|^{2 } _{\mathcal{H}^{1}_{k}})
\\&~~~+\mathbb{E} \underset{r\in[0,s]}{\sup}\left \| \chi _{1}(s\wedge \tau _{R }\wedge \tau _{\beta })-\chi _{2}(s\wedge \tau _{R }\wedge \tau _{\beta }) \right \|^{2} _{\mathcal{H}^{1}_{k}}]\text{d}s
\\&:=\varphi (t).
	\end{align*}
Let $\Gamma  (t)=\int_{\iota}^{t}(\mathcal{A}(s)+s)^{-1} \text{d}s $, which implies that $\Gamma (t)$ is a monotonically increasing function and  $\lim_{t \to 0^{+}} \Gamma (t)=-\infty$.
Thus, $\Gamma (t)$ satisfies $\Gamma (t)>-\infty$ for any $t > 0$. Therefore,
\begin{align}\label{r13}
		 \Gamma(\mathbb{E} (\underset{r\in[0,t]}{\sup}\left \| \chi_{1}(r\wedge \tau _{R }\wedge \tau _{\beta  } )-\chi_{2}(r\wedge \tau _{R }\wedge \tau _{\beta } ) \right \| _{\mathcal{H}^{1}_{k}}^{2}))\le\Gamma(\varphi (t)).
	\end{align}
We subsequently get
\begin{align}\label{r15}
		&\Gamma ( \varphi (t)) \nonumber\\&=\Gamma (\varphi  (0))+\int_{0}^{t}\Gamma '(\varphi (s))\text{d}\varphi(s) \nonumber
\\&=\Gamma ( C_{k,T,\nu }\left \| \chi^{0}_{1}-\chi^{0}_{2} \right \| _{\mathcal{H}^{1}_{k}}^{2})+C_{k,T,\nu }\int_{0}^{t} [\frac{\mathcal{A} (\mathbb{E} \underset{r\in[0,s]}{\sup}\left \| \chi _{1}(s\wedge \tau _{R }\wedge \tau _{\beta })-\chi _{2}(s\wedge \tau _{R }\wedge \tau _{\beta }) \right \|^{2 } _{\mathcal{H}^{1}_{k}})
}{\mathcal{A}(\varphi(s))+\varphi(s)}\nonumber
\\&~~~+ \frac{\mathbb{E} \underset{r\in[0,s]}{\sup}\left \| \chi _{1}(s\wedge \tau _{R }\wedge \tau _{\beta })-\chi _{2}(s\wedge \tau _{R }\wedge \tau _{\beta }) \right \|^{2} _{\mathcal{H}^{1}_{k}}}{\mathcal{A}(\varphi(s))+\varphi(s)}]\text{d}s
\\&\le\Gamma (C_{k,T,\nu }\left \| \chi^{0}_{1}-\chi^{0}_{2} \right \| _{\mathcal{H}^{1}_{k}}^{2})+C_{k,T,\nu }t.\nonumber
	\end{align}
By \eqref{r13} and \eqref{r15}, we have when $\left \| \chi^{0}_{1}-\chi^{0}_{2} \right \| _{\mathcal{H}^{1}_{k}}\to 0$, we have $$\Gamma(\mathbb{E} (\underset{r\in[0,t]}{\sup}\left \| \chi_{1}(r\wedge \tau _{R }\wedge \tau _{\beta  } )-\chi_{2}(r\wedge \tau _{R }\wedge \tau _{\beta } ) \right \| _{\mathcal{H}^{1}_{k}}^{2}))\to -\infty	,$$ which implies
\begin{align*}
	\lim_{\left \| \chi^{0}_{1}-\chi^{0}_{2} \right \| _{\mathcal{H}^{1}_{k}} \to 0} \mathbb{E} (\underset{r\in[0,t]}{\sup}\left \| \chi_{1}(r\wedge \tau _{R }\wedge \tau _{\beta  } )-\chi_{2}(r\wedge \tau _{R }\wedge \tau _{\beta } ) \right \|_{\mathcal{H}^{1}_{k}} ^{2})=0 ,	
	\end{align*}
for any $t\in [0,T]$.

Furthermore, by applying \eqref{r6} and Fatou's lemma, and letting $R\to \infty$, we derive
\begin{align*}
\lim_{\left \| \chi^{0}_{1}-\chi^{0}_{2} \right \| _{\mathcal{H}^{1}_{k}} \to 0} \mathbb{E} (\underset{r\in[0,t]}{\sup}\left \| \chi_{1}(r\wedge \tau _{\beta  } )-\chi_{2}(r\wedge \tau _{\beta } ) \right \|_{\mathcal{H}^{1}_{k}} ^{2})=0.
\end{align*}
 Then, by the definition of $\tau _{\beta  }$,
\begin{align*}
		 &\mathbb{E} (\underset{r\in[0,t]}{\sup}\left \| \chi_{1}(r\wedge \tau _{\beta  } )-\chi_{2}(r\wedge \tau _{\beta } ) \right \| _{\mathcal{H}^{1}_{k}} ^{2})\\&=\mathbb{P}(t\ge\tau _{\beta })\mathbb{E} (\underset{r\in[0,t]}{\sup}\left \| \chi_{1}(r\wedge \tau _{\beta } )-\chi_{2}(r\wedge \tau _{\beta} ) \right \| _{\mathcal{H}^{1}_{k}}^{2})
+\mathbb{P}(t<\tau _{\beta})\mathbb{E} (\underset{r\in[0,t]}{\sup}\left \| \chi_{1}(r\wedge \tau _{\beta } )-\chi_{1}(r\wedge \tau _{\beta } ) \right \|_{\mathcal{H}^{1}_{k}} ^{2})
\\&\ge\mathbb{P}(t\ge\tau _{\beta }) \beta ^{2},
	\end{align*}
which implies $\lim_{\left \| \chi^{0}_{1}-\chi^{0}_{2} \right \| _{\mathcal{H}^{1}_{k}} \to 0} \mathbb{P}(t\ge\tau _{\beta })=0$. Consequently,
\begin{align*}
		 &\lim_{\left \| \chi^{0}_{1}-\chi^{0}_{2} \right \| _{\mathcal{H}^{1}_{k}} \to 0} \mathbb{E} (\underset{r\in[0,t]}{\sup}\left \| \chi_{1}(r )-\chi_{2}(r ) \right \|_{\mathcal{H}^{1}_{k}} ^{2})\\&\le\lim_{\left \| \chi^{0}_{1}-\chi^{0}_{2} \right \| _{\mathcal{H}^{1}_{k}} \to 0}  \mathbb{E} (\underset{r\in[0,t]}{\sup}\left \| \chi_{1}(r\wedge \tau _{\beta  } )-\chi_{2}(r\wedge \tau _{\beta } ) \right \|_{\mathcal{H}^{1}_{k}} ^{2})\\&=0.
	\end{align*}
Therefore, when $\left \| \chi^{0}_{1}-\chi^{0}_{2} \right \| _{\mathcal{H}^{1}_{k}} \to 0$, we obtain
\begin{align}\label{i111}
		\mathbb{P}(\underset{r\in[0,t]}{\sup}\left \| \chi_{1}(r)-\chi_{2}(r ) \right \|_{\mathcal{H}^{1}_{k}} ^{2}=0) =1,
	\end{align}
i.e., the pathwise uniqueness for \eqref{t10} holds. Subsequently, by applying the Yamada-Watanabe theorem, we demonstrate the existence and uniqueness of the global strong solution for \eqref{t10}. Finally, for \eqref{t2},  using  a standard argument (e.g., see Lemma 3.1 of \cite{ref7} similarly), we establish the existence of a unique maximal local strong solution $u^{k}(t)$. Furthermore, let $\tau $ be the explosion time or life time. Then,  for any $t\in [0,\tau)$, by It$\hat{\text{o}} $'s formula and  \eqref{r6}, we have
\begin{align*}
		\left \| u^{k}(t) \right \| _{\mathcal{H}^{1}_{k}}^{2}\le  \left \| u^{k}_{0} \right \| _{\mathcal{H}^{1}_{k}}^{2}+C\int_{0}^{t} (\underset{r\in[0,s]}{\sup}\left \| u^{k}(r) \right \| _{\mathcal{H}^{1}_{k}}^{2}+1)\text{d}s+2\int_{0}^{t} ( u^{k}(r))^{T} \mathcal{G}^{N} (s,u^{k}(r))\text{d}W^{k}(s),
	\end{align*}
which implies that $\tau =\infty $ almost surely by the stochastic Gronwall lemma of \cite{ref8}. Consequently, the unique maximal local strong solution $u^{k}(t)$ is non-explosive almost surely within any finite time.
~\\
 \\\textbf{Step2: Existence and uniqueness of stochastic tamed 3D Navier–Stokes equations}

Our primary strategy involves first establishing the existence of weak solutions for \eqref{t1} through the application of the Galerkin-type approximation technique. In \textbf{Step1}, we have established the existence and uniqueness of the solution for the finite-dimensional system \eqref{t2} with the locally weak
monotonicity coefficients. Next, to address the existence problem of the weak solution for the system \eqref{t1}, we require a priori estimates of the solution $u^{k}(t)$ and the compactness of the space. And a similar proof procedure for the existence of a weak solution for System \eqref{t1} is presented  in Theorem 3.8 of \cite{ref1}, where a detailed proof procedure is provided. Specifically, as $k\to\infty $, the strong solution $u^{k}(t)$ of the finite-dimensional system \eqref{t2} can yield a weak solution of the original system \eqref{t1}. Subsequently, we will proceed to establish the pathwise uniqueness of these weak solutions. Then, by applying the Yamada-Watanabe theorem, we demonstrate the existence of a global strong solution for \eqref{t1}.

Suppose that two stochastic processes $u_{1}(t), u_{2}(t)$ are weak solution of \eqref{t1} with initial conditions $u_{1}(0)=u_{1}^{0}\in \mathcal{H} ^{1}, u_{2}(0)=u_{2}^{0}\in \mathcal{H} ^{1}$. Let us denote
\begin{align*}
		\tau _{R}=\underset{t\ge0}{\inf} \{\left \| u_{1}(t) \right \| _{\mathcal{H}^{1}}\vee \left \| u_{2}(t) \right \| _{\mathcal{H}^{1}}>R \},\quad \tau _{\beta }=\underset{t\ge0}{\inf} \{\left \| u_{1}(t)-u_{2} (t)\right \| _{\mathcal{H}^{0}}>\beta  \}.
	\end{align*}
By It$\hat{\text{o}} $'s formula, the B-D-G inequality, Young's inequality and \eqref{k1}, we obtain
\begin{align}\label{t21}
		 & \mathbb{E} (\underset{r\in[0,t]}{\sup}\left \| u_{1}(r\wedge \tau _{R }\wedge \tau _{\beta  } )-u_{2}(r\wedge \tau _{R }\wedge \tau _{\beta } ) \right \| _{\mathcal{H} ^{0}}^{2}) \nonumber
\\&=\left \| u^{0}_{1}-u^{0}_{2} \right \| _{\mathcal{H} ^{0}}^{2}+\mathbb{E}\underset{r\in[0,t]}{\sup}\int_{0}^{r\wedge \tau _{R }\wedge \tau _{\beta }} [2\nu \left \langle u_{1}(s)-u_{2}(s),\Delta u_{1}(s)-\Delta u_{2}(s) \right \rangle _{\mathcal{H} ^{0}}\nonumber
\\&~~~-2  \left \langle u_{1}(s)-u_{2}(s),(u_{1}(s)\cdot \nabla)u_{1}(s)-(u_{2}(s)\cdot \nabla)u_{2}(s) \right \rangle _{\mathcal{H} ^{0}}\nonumber
\\&~~~-2\left \langle u_{1}(s)-u_{2}(s), \Psi _{N}(\left | u_{1}(s) \right |^{2} )  u_{1}(s)-\Psi _{N}(\left |  u_{2}(s) \right |^{2} )  u_{2}(s) \right \rangle_{\mathcal{H} ^{0}}\nonumber
\\&~~~+2 \left \langle u_{1}(s)-u_{2}(s),f(s,u_{1}(s))-f(s,u_{2}(s)) \right \rangle _{\mathcal{H} ^{0}}\nonumber +\left \| G(s,u_{1}(s))-G(s,u_{2}(s)) \right \|_{\mathscr{L}(K,\mathcal{H} ^{0})} ^{2}]\text{d}s   \nonumber
\\&~~~+2\mathbb{E} (\underset{r\in[0,t]}{\sup}\int_{0}^{r\wedge \tau _{R }\wedge \tau _{\beta }}\left \langle u_{1}(s)-u_{2}(s),G(s,u_{1}(s))-G(s,u_{2}(s)) \right \rangle _{\mathcal{H} ^{0} }\text{d}W(s))\nonumber
\\&\le \left \| u^{0}_{1}-u^{0}_{2} \right \| _{\mathcal{H} ^{0}}^{2}+\mathbb{E}\underset{r\in[0,t]}{\sup}\int_{0}^{r\wedge \tau _{R }\wedge \tau _{\beta }} [-2\nu \left \| \nabla u_{1}(s)- \nabla u_{2}(s) \right \|^{2}_{\mathcal{H} ^{0} }\nonumber
\\&~~~+2\left \langle \nabla u_{1}(s)- \nabla u_{2}(s),u^{T}_{1}(s)u_{1}(s)-u^{T}_{2}(s)u_{2}(s)\right \rangle_{\mathcal{H} ^{0} }+8\left \| \left | u_{1}(s)-u_{1}(s) \right |(\left | u_{1}(s) \right |
+\left | u_{2}(s) \right |  )  \right \|^{2}_{\mathcal{H} ^{0} } \nonumber
\\&~~~+C_{T}\mathcal{A} (\left \| u_{1}(s)-u_{2}(s) \right \|^{2}_{\mathcal{H} ^{0} })+\sup_{t\in[0,t],x\in \mathbb{D} } \left \| \mathcal{K} (t,x) \right \| _{\mathscr{L}(K,\mathcal{H} ^{0})}(\left \| \nabla u_{1}(s)-\nabla u_{2}(s) \right \| ^{2}_{\mathcal{H} ^{0} })
\\&~~~+12(\mathbb{E}\int_{0}^{t\wedge \tau _{R }\wedge \tau _{\beta }}\left \| u_{1}(s)-u_{2}(s) \right \|^{2}_{\mathcal{H} ^{0}}\left \| G(s,u_{1}(s))-G(s,u_{2}(s)) \right \|_{\mathscr{L}(K,\mathcal{H} ^{0})} ^{2} \text{d}s)^{\frac{1}{2} }\nonumber
\\&\le \left \| u^{0}_{1}-u^{0}_{2} \right \| _{\mathcal{H} ^{0}}^{2}+\int_{0}^{t\wedge \tau _{R }\wedge \tau _{\beta }} [-\nu\left \| \nabla u_{1}(s)- \nabla u_{2}(s) \right \|^{2}_{\mathcal{H} ^{0} }+\frac{\nu }{73}\left \| \nabla u_{1}(s)- \nabla u_{2}(s) \right \|^{2}_{\mathcal{H} ^{0} }\nonumber
\\&~~~+2C_{\nu ,R}\left \| u_{1}(s)-u_{2}(s) \right \|^{\frac{1}{2} }_{\mathcal{H} ^{0} } \left \| u_{1}(s)-u_{2}(s) \right \|^{\frac{3}{2} }_{\mathcal{H} ^{1} }
+C_{T}(\mathcal{A} (\left \| u_{1}(s)-u_{2}(s) \right \|^{2}_{\mathcal{H} ^{0} })]\text{d}s\nonumber
\\&~~~+72\mathbb{E}\int_{0}^{t\wedge \tau _{R }\wedge \tau _{\beta }}\left \| G(s,u_{1}(s))-G(s,u_{2}(s)) \right \|_{\mathscr{L}(K,\mathcal{H} ^{0})} ^{2} \text{d}s\nonumber
\\&~~~+\frac{1}{2}  \mathbb{E} (\underset{r\in[0,t]}{\sup}\left \| u_{1}(r\wedge \tau _{R }\wedge \tau _{\beta  } )-u_{2}(r\wedge \tau _{R }\wedge \tau _{\beta } ) \right \| _{\mathcal{H} ^{0}}^{2})\nonumber
\\&\le \left \| u^{0}_{1}-u^{0}_{2} \right \| _{\mathcal{H} ^{0}}^{2}+\int_{0}^{t\wedge \tau _{R }\wedge \tau _{\beta }}[C_{R,\nu }\left \| u_{1}(s)-u_{2}(s) \right \|^{2 }_{\mathcal{H} ^{0}}+C_{T}\mathcal{A} (\left \| u_{1}(s)-u_{2}(s) \right \|^{2}_{\mathcal{H} ^{0} })]\text{d}s\nonumber
\\&~~~+\frac{1}{2}  \mathbb{E} (\underset{r\in[0,t]}{\sup}\left \| u_{1}(r\wedge \tau _{R }\wedge \tau _{\beta  } )-u_{2}(r\wedge \tau _{R }\wedge \tau _{\beta } ) \right \| _{\mathcal{H} ^{0}}^{2}).\nonumber
	\end{align}
Based on  the conclusion of \eqref{t21}, we will again employ a proof, similar to the proof procedures outlined in \eqref{r12}-\eqref{i111}, to establish the pathwise uniqueness of the weak solution. Subsequently, we will apply the Yamada-Watanabe theorem to demonstrate the existence and uniqueness of the strong solution for system  \eqref{t1}.

We will now examine the regularity of solutions to \eqref{t1}. Utilizing \eqref{t3}, \eqref{t5}, \eqref{t555}, Young's inequality and  \eqref{t7},  we obtain
\begin{align}\label{t22}
		 &\mathbb{E} \left \| u(t) \right \|_{\mathcal{H}^{1} }^{2p}\nonumber
 \\&= \left \| u_{0} \right \|_{\mathcal{H}^{1} }^{2p}+\mathbb{E}\int_{0}^{t}p\left \| u(s) \right \|_{\mathcal{H}^{1} }^{2p-2}[2\left \langle u(s),\nu \Delta u(s) \right \rangle_{\mathcal{H}^{1} }-2\left \langle u(s),(u(s),\nabla )u(s) \right \rangle_{\mathcal{H}^{1} }\nonumber
 \\&~~~+2\left \langle u(s),\Psi _{N}(\left |u(s)  \right |^{2} )u(s)  \right \rangle _{\mathcal{H}^{1} }+2\left \langle u(s),f(s,u(s))  \right \rangle _{\mathcal{H}^{1} }+\left \| G(s,u(s)) \right \|^{2}_{\mathscr{L}(K,\mathcal{H}^{1})}   ]   \text{d}s\nonumber
 \\&~~~+2p(p-1)\mathbb{E}\int_{0}^{t}\left \| u(s) \right \|_{\mathcal{H}^{1} } ^{2p-2} \left \| G(s,u(s)) \right \|^{2}_{\mathscr{L}(K,\mathcal{H}^{1})} \text{d}s
\\&\le \left \| u_{0} \right \|_{\mathcal{H}^{1} }^{2p}+\mathbb{E}\int_{0}^{t}p\left \| u(s) \right \|_{\mathcal{H}^{1} }^{2p-2}[-2\nu\left \| u(s)\right \| ^{2}_{\mathcal{H}^{2}} +\nu \left \| u \right \| ^{2}_{\mathcal{H}^{2} }+\frac{1}{\nu }  \left \| \left | u \right |\left | \nabla u \right | \right \|^{2}_{\mathcal{H}^{0}}-\frac{2 }{\nu}  \left \| \left | u \right |\left | \nabla u \right | \right \|^{2}_{\mathcal{H}^{0}}\nonumber
\\&~~~+C_{N ,\nu}\left \| u  \right \|^{2}_{\mathcal{H}^{1} }
+2\left \| u(s) \right \|_{\mathcal{H}^{2} }\left \| f(s,u(s)) \right \|_{L^{2} } \text{d}s\nonumber
\\&~~~+p(2p-1)\mathbb{E}\int_{0}^{t}\left \| u(s) \right \|_{\mathcal{H}^{1} } ^{2p-2} [  \frac{p\nu -\epsilon }{p(2p-1)}\left \| u (s)\right \| ^{2}_{\mathcal{H}^{2} }+C_{\nu,T}\left \| u(s)  \right \|^{2}_{\mathcal{H}^{1} }+M] \text{d}s\nonumber
\\&\le \left \| u_{0} \right \|_{\mathcal{H}^{1} }^{2p}+\mathbb{E}\int_{0}^{t}\left \| u(s) \right \|_{\mathcal{H}^{1} }^{2p-2}[-\epsilon \left \| u(s)\right \| ^{2}_{\mathcal{H}^{2}}-\frac{1 }{\nu}  \left \| \left | u \right |\left | \nabla u \right | \right \|^{2}_{\mathcal{H}^{0}} +C_{\nu,T,M,\epsilon,p,N}(\left \| u (s) \right \|^{2}_{\mathcal{H}^{1} }+1)] \text{d}s.\nonumber
	\end{align}
By applying Gronwall’s lemma, we observe that
\begin{align}\label{t23}
		 &\sup_{t\in[0,T]} \mathbb{E} \left \| u(t) \right \|_{\mathcal{H}^{1} }^{2p}+\mathbb{E}\int_{0}^{T}\left \| u(s) \right \|_{\mathcal{H}^{1} }^{2p-2}[ \left \| u(s)\right \| ^{2}_{\mathcal{H}^{2}}+  \left \| \left | u(s) \right |\left | \nabla u(s) \right | \right \|^{2}_{\mathcal{H}^{0}}\text{d}s\nonumber
\\&\le C_{\nu,T,M,\epsilon,p,N}(\left \| u_{0} \right \|_{\mathcal{H}^{1} }^{2p}+1)
	\end{align}
for any $T>0$. Furthermore, by utilizing \eqref{t22}, Young’s inequality and the B-D-G inequality, we have
\begin{align}\label{t112}
		 \mathbb{E} (\sup_{t\in[0,T]}\left \| u(t) \right \|^{2p}_{\mathcal{H}^{1} }  )
 &\le \left \| u_{0} \right \|_{\mathcal{H}^{1} }^{2p}+C_{\nu,T,M,\epsilon,N}\mathbb{E}\int_{0}^{T}(\left \| u (s) \right \|^{2p}_{\mathcal{H}^{1} }+1) \text{d}s\nonumber
 \\&~~~+2p\mathbb{E}\sup_{t\in[0,T]}\int_{0}^{t}\left \| u (s) \right \|^{2p-2}_{\mathcal{H}^{1} }\left \langle u(s),G(s,u(s))\text{d}W(s) \right \rangle _{\mathscr{L}(K,\mathcal{H}^{1}) }\nonumber
 \\&\le  \left \| u_{0} \right \|_{\mathcal{H}^{1} }^{2p}+C_{\nu,T,M,\epsilon,N}\mathbb{E}\int_{0}^{T}(\left \| u (s) \right \|^{2p}_{\mathcal{H}^{1} }+1) \text{d}s\nonumber
 \\&~~~+C\mathbb{E}(\int_{0}^{T}\left \| u(s) \right \|_{\mathcal{H}^{2} }^{4p-2}\left \|G(s,u(s)) \right \| ^{2}_{\mathscr{L}(K,\mathcal{H}^{0}) }\text{d}s)^{\frac{1}{2} }
 \\&\le  \left \| u_{0} \right \|_{\mathcal{H}^{1} }^{2p}+C_{\nu,T,M,\epsilon,N}\mathbb{E}\int_{0}^{T}(\left \| u (s) \right \|^{2p}_{\mathcal{H}^{1} }+1) \text{d}s\nonumber
 \\&~~~+\frac{1}{2} \mathbb{E} (\sup_{t\in[0,T]}\left \| u(s) \right \|^{2p}_{\mathcal{H}^{1} })+C_{p}(\mathbb{E}\int_{0}^{T} \left \|G(s,u(s)) \right \| ^{2}_{\mathscr{L}(K,\mathcal{H}^{1}) } \text{d}s)^{p}\nonumber
 \\&\le C_{\nu,T,M,\epsilon}(\left \| u_{0} \right \|_{\mathcal{H}^{1} }^{2p}+1)+\frac{1}{2}\mathbb{E} (\sup_{t\in[0,T]}\left \| u(t) \right \|^{2p}_{\mathcal{H}^{1} }  )+ C_{\nu,T,M,\epsilon,N}\mathbb{E}\int_{0}^{T}\left \| u (s) \right \|^{2p}_{\mathcal{H}^{1} } \text{d}s.\nonumber
	\end{align}
Then, in view of Gronwall’s lemma, we obtain
\begin{align*}
		\mathbb{E} (\sup_{t\in[0,T]}\left \| u(t) \right \|^{2p}_{\mathcal{H}^{1} }  )\le C_{\nu,T,M,\epsilon,N}(\left \| u_{0} \right \|_{\mathcal{H}^{1} }^{2p}+1).
	\end{align*}
\end{proof}
\begin{remark}\label{re1}
Indeed, similar to the proof of the aforementioned theorem, we can also derive the following estimate:
\begin{align*}
		 \mathbb{E} (\sup_{t\in[0,T]}\left \| u(t) \right \|_{\mathcal{H}^{0} }^{2p})+\mathbb{E}\int_{0}^{T}\left \| u(t) \right \|_{\mathcal{H}^{0} }^{2p-2} \left \| u(s)\right \| ^{2}_{\mathcal{H}^{1}}\text{d}s\nonumber
\le C_{\nu,T,M,\epsilon,p,N}(\left \| u_{0} \right \|_{\mathcal{H}^{0} }^{2p}+1),
	\end{align*}
for any $T>0$.
\end{remark}
\section{\textup{The averaging principle of Stochastic tamed 3D Navier–Stokes equations  with locally weak monotonicity coefficients}}
in this paper, we consider the following system
\begin{align}\label{t222}
		\begin{cases}
 \text{d}u(t)=[\eta _{1}(\frac{t }{\varepsilon } ) \Delta  u-\eta _{2}(\frac{t }{\varepsilon })[( u,\nabla  )  u+\Psi _{N}(\left | u \right |^{2} ) u]+\nabla p_{1}+f(\frac{t }{\varepsilon },u)]\text{d}t\\~~~~~~~~~~+[(\mathcal{K} (\frac{t }{\varepsilon },x)\cdot \nabla )u+\nabla p_{2}+g(\frac{t }{\varepsilon },u)]\text{d}W(t),\\
 \text{div} u(t)=0,\quad u(0)=u^{\varepsilon}_{0},
\end{cases}
	\end{align}
where $\varepsilon \in (0,1]$ and the functions
$\eta _{1},\eta _{2}\in C([0,+\infty ))$, such that $0<a_{1}\le \eta _{1}(t)\le a_{2}$, and $0< a_{3}\le \eta _{2}(t)  \le a_{4}$ for any $t\in[0,+\infty ) $. Letting $\nu =a_{1}$ and  applying the operator $\Pi$ act to both sides of  \eqref{t222}, we define $\eta _{1}^{\varepsilon}(t)=\eta _{1}(\frac{t }{\varepsilon } )$, $\eta _{2}^{\varepsilon}(t)=\eta _{1}(\frac{t }{\varepsilon } )$, $f^{\varepsilon}(t,u )=\Pi f(\frac{t}{\varepsilon } ,u)$, $S^{\varepsilon}_{3}(t,x,u)=\Pi (\mathcal{K} (\frac{t }{\varepsilon },x)\cdot \nabla )u$ and $g^{\varepsilon}(t,\varphi )=\Pi g(\frac{t}{\varepsilon } ,\varphi)$. Then we eliminate the pressure, thereby obtain an evolution
system for the velocity field
\begin{align}\label{a2}
		\begin{cases}
 \text{d}u(t)=[\eta _{1}^{\varepsilon}(t ) S_{1}(u)-\eta _{2}^{\varepsilon}(t )S_{2}(u)+f^{\varepsilon}(t  ,u)]\text{d}t+[S^{\varepsilon}_{3}(t,x,u)+g^{\varepsilon}(t  ,u)]\text{d}W(t)\\
 \text{div} u(t)=0,\quad u(0)=u^{\varepsilon}_{0}.
 \end{cases}
	\end{align}
To facilitate a more comprehensive analysis of the averaging principle applied to systems  \eqref{a2} with locally weak
monotonicity coefficients, it is essential to introduce the following assumptions:
~\\
\\\textbf{(H4)} There exist  constants $\eta ^{*}_{1}$, $\eta ^{*}_{2}$ with $0<a_{1}\le \eta^{*} _{1}\le a_{2}$, $0<a_{3}\le \eta^{*} _{2}\le a_{3}$ and  functions $\mathcal{R}_{1} $, $\mathcal{R}_{2} $ such that for any $T>0$, $t\in [0,T]$,
\begin{align*}
		\frac{1}{T}\left | \int_{t}^{t+T}\eta _{i}(s)- \eta ^{*}_{i}\text{d}s \right | \le \mathcal{R} _{i}(T),
	\end{align*}
  $i=1,2$.
\\\textbf{(H5)} There exist functions $\mathcal{R}_{3} $, $\mathcal{R}_{4} $, $\mathcal{R}_{5} $  and $f^{*}\in C(L^{2}(\mathbb{D} ),L^{2}(\mathbb{D} ))$, $g^{*}\in C(L^{2}(\mathbb{D} ),\mathscr{L}(K,L^{2}(\mathbb{D} )) )$ and $\mathcal{K} ^{*}\in L^{2}(\mathbb{D} )$ such that for any $T>0$, $t\in [0,T]$, $u \in L^{2}(\mathbb{D} )$ and $x\in \mathbb{D}$,
\begin{align*}
		\frac{1}{T} \left \| \int_{t}^{t+T}[f(s,u ) -f^{*}(u )] \text{d}s\right \| _{L^{2}}\le \mathcal{R}_{3}(T)(\left \| u \right \|_{L^{2}}+M ),
	\end{align*}
\begin{align*}
		\frac{1}{T}\int_{t}^{t+T} \left \| g(s,u ) -g^{*}(u )\right \|^{2} _{\mathscr{L}(K,L^{2}) } \text{d}s\le \mathcal{R}_{4}(T)(\left \| u \right \|^{2}_{L^{2}}+M ),
	\end{align*}
 and
 \begin{align*}
		\frac{1}{T}  \int_{t}^{t+T}\left \|\mathcal{K} (s,x ) -\mathcal{K} ^{*}(x )\right \|^{2} _{\mathscr{L}(K,L^{2})} \text{d}s\le \mathcal{R}_{5}(T),
	\end{align*}
where $\mathcal{R}_{i} $ is decreasing, positive bounded functions and $\lim_{T \to \infty} \mathcal{R} _{i}(T)=0$ for $i=1,2,3,4,5$.

Now we consider the following averaged equation
\begin{align}\label{a22222}
		\begin{cases}
 \text{d}u(t)=[\eta _{1}^{*} S_{1}(u)-\eta _{2}^{*} S_{2}(u)+f^{*}(u)]\text{d}t+[S^{*}_{3}(x,u)+g^{*}( u)]\text{d}W(t),\\
 \text{div} u(t)=0,\quad u(0)=u^{*}_{0},
\end{cases}
	\end{align}
where $S^{*}_{3}(x,u)=\Pi (\mathcal{K} ^{*}(x)\cdot \nabla )u$.
\begin{theorem}\label{co1}Consider \eqref{a22222}. If $f$, $\mathcal{K} $ and $g$ satisfy \textbf{(H1)}$-$\textbf{(H5)} and $a_{4}\le2a_{3}$, then $f^{*}$, $\mathcal{K} ^{*}$ and $g^{*}$ also
satisfy \textbf{(H1)}$-$\textbf{(H3)}. Consequently, the following statement holds: for any $u^{*}_{0}\in \mathcal{H} ^{1} $, there exist a  unique strong  solution $u^{*}(t)$  to \eqref{a22222} with $u^{*}(0)=u^{*}_{0}$ and
\begin{align*}
		& \mathbb{E} (\sup_{t\in[0,T]}\left \| u^{*}(t) \right \|_{\mathcal{H}^{1} }^{2p})+\mathbb{E}\int_{0}^{T}\left \| u^{*}(t) \right \|_{\mathcal{H}^{1} }^{2p-2}[ \left \| u^{*}(s)\right \| ^{2}_{\mathcal{H}^{2}}+  \left \| \left | u^{*} \right |\left | \nabla u^{*} \right | \right \|^{2}_{\mathcal{H}^{0}}\text{d}s\nonumber
\\&\le C_{\nu,T,M,\epsilon,p,N}(\left \| u^{*}_{0} \right \|_{\mathcal{H}^{1} }^{2p}+1)
	\end{align*}
for any $t\in[0,T]$.
\end{theorem}
\begin{proof}
A similar proof can be found in Lemma 4.2 of \cite{ref9}.
\end{proof}

For any given function $\psi $, define a piecewise function $\overline{\psi } $  such that
\begin{align}\label{a5}
	\overline{\psi } (t)=\left\{\begin{matrix}
  \psi (0)&t\in [0,d),\\
   \psi(d)  & t\in [d,2d),\\
  ...&...\\
\psi (nd)&t\in [nd,(n+1)d),\\
   ...&...,
\end{matrix}\right.
	\end{align}
 where $n\in\mathbb{N} ^{+}$ and $d$ is a fixed constant. Without loss of generality, we let $d\le 1$. To delve deeper into the long-term asymptotic behavior of systems \eqref{a2} and \eqref{a222}, it is essential to introduce the ensuing lemma:
\begin{lemma}\label{le1}If $f$, $k$ and $g$ satisfy \textbf{(H1)}$-$\textbf{(H5)} and $a_{4}\le2a_{3}$, the following statements hold:  for any $T>0$ and  $u_{0} ^{\varepsilon }, u_{0} ^{* }\in\mathcal{H} ^{1}$,
\begin{align*}
		\mathbb{E} \int_{0}^{T} \left \| u^{\varepsilon }(s;u_{0} ^{\varepsilon }) -\overline{u}^{\varepsilon }(s;u_{0} ^{\varepsilon }) \right \|^{2}_{\mathcal{H} ^{1}} \text{d}s\le C_{T}(\left \| u^{\varepsilon }_{0}\right \|_{\mathcal{H} ^{1}}^{2}+\left \| u^{\varepsilon }_{0}\right \|_{\mathcal{H} ^{1}}^{6}+1)d^{\frac{1}{2} } ,
	\end{align*}
and
\begin{align*}
		\mathbb{E} \int_{0}^{T} \left \| u^{* }(s;u_{0} ^{* }) -\overline{u}^{*}(s;u_{0} ^{* }) \right \|^{2}_{\mathcal{H} ^{1}} \text{d}s\le C_{T}(\left \| u^{* }_{0}\right \|_{\mathcal{H} ^{1}}^{2}+\left \| u^{* }_{0}\right \|_{\mathcal{H} ^{1}}^{6}+1)d^{\frac{1}{2} },
	\end{align*}
where  $u^{\varepsilon }(t;u_{0} ^{\varepsilon })$  is the solution of \eqref{a2} with the initial value $u^{\varepsilon }(0)=u_{0} ^{\varepsilon }$ and
$u^{*}(t;u_{0} ^{*})$  is the solution of \eqref{a222} with the initial value $u^{*}(0)=u_{0} ^{*}$.
\end{lemma}
\begin{proof}
Let $N_{d}$ denote the integer part of  $\frac{T}{d} $ and by applying It$\hat{\text{o}} $'s formula, we have
\begin{align}\label{a222}
		&\mathbb{E} \int_{0}^{T}\left \| u^{\varepsilon }(s)-\bar{u}^{\varepsilon }(s)  \right \|^{2}_{\mathcal{H}^{1} }  \text{d}s\nonumber
\\&=\mathbb{E} \int_{0}^{d}\left \| u^{\varepsilon }(s)-u^{\varepsilon }_{0}  \right \|^{2}_{\mathcal{H}^{1} }  \text{d}s+\mathbb{E} \sum_{n=1}^{N_{d}-1} \int_{nd}^{(n+1)d}\left \| u^{\varepsilon }(s)-u^{\varepsilon }(nd)  \right \|^{2}_{\mathcal{H}^{1} }  \text{d}s\nonumber
\\&~~~+\mathbb{E} \int_{N_{d}d}^{T}\left \| u^{\varepsilon }(s)-u^{\varepsilon }(N_{d}d)  \right \|^{2}_{\mathcal{H}^{1} }  \text{d}s
\\&\le C_{T}(\left \| u^{\varepsilon }_{0}  \right \|^{2}_{\mathcal{H}^{1} }+1)d+2\mathbb{E} \sum_{n=1}^{N_{d}-1} \int_{nd}^{(n+1)d}\left \| u^{\varepsilon }(s)-u^{\varepsilon }(s-d)  \right \|^{2}_{\mathcal{H}^{1} }  \text{d}s\nonumber
\\&~~~+2\mathbb{E} \sum_{n=1}^{N_{d}-1} \int_{nd}^{(n+1)d}\left \| u^{\varepsilon }(s-d)-u^{\varepsilon }(nd)  \right \|^{2}_{\mathcal{H}^{1} }  \text{d}s.\nonumber
	\end{align}
Then, by \textbf{(H1)} and  Young's inequality, we have
\begin{align*}
		&\left \| u^{\varepsilon }(s)-u^{\varepsilon }(s-d)  \right \|^{2}_{\mathcal{H}^{1} }
\\&=\int_{s-d}^{s} [2\left \langle u^{\varepsilon }(r)-u^{\varepsilon }(s-d),\eta _{1}^{\varepsilon }(r)S_{1}(u^{\varepsilon }(r)) \right \rangle_{\mathcal{H}^{1}  }-2\left \langle u^{\varepsilon }(r)-u^{\varepsilon }(s-d),\eta _{2}^{\varepsilon }(r)S_{2}(u^{\varepsilon }(r)) \right \rangle_{\mathcal{H}^{1}  }
\\&~~~+2\left \langle u^{\varepsilon }(r)-u^{\varepsilon }(s-d),f^{\varepsilon }(r,u^{\varepsilon }(r)) \right \rangle_{\mathcal{H}^{1}  }+\left \| S^{\varepsilon}_{3}(r,x,u^{\varepsilon }(r)) \right \|^{2}_{\mathscr{L}(K,\mathcal{H} ^{1})}+\left \| g^{\varepsilon }(r,u^{\varepsilon }(r)) \right \|^{2}_{\mathscr{L}(K,\mathcal{H} ^{1})}]\text{d}r
\\&~~~+2\int_{s-d}^{s}\left \langle u^{\varepsilon }(r)-u^{\varepsilon }(s-d), S^{\varepsilon}_{3}(r,x,u^{\varepsilon }(r)) \text{d}W(r) \right \rangle_{\mathcal{H}^{1} }+2\int_{s-d}^{s}\left \langle u^{\varepsilon }(r)-u^{\varepsilon }(s-d), g^{\varepsilon }(r,u^{\varepsilon }(r))\text{d}W(r) \right \rangle_{\mathcal{H}^{1} }
\\&\le \int_{s-d}^{s} [2\left \langle u^{\varepsilon }(r)-u^{\varepsilon }(s-d),\eta _{1}^{\varepsilon }(r)S_{1}(u^{\varepsilon }(r)) \right \rangle_{\mathcal{H}^{1}  }-2\left \langle u^{\varepsilon }(r)-u^{\varepsilon }(s-d),\eta _{2}^{\varepsilon }(r)S_{2}(u^{\varepsilon }(r)) \right \rangle_{\mathcal{H}^{1}  }
\\&~~~+2\left \|u^{\varepsilon }(r)-u^{\varepsilon }(s-d) \right\| _{\mathcal{H}^{1} }\left \|f^{\varepsilon }(r,u^{\varepsilon }(r)) \right\| _{\mathcal{H}^{1} }+a^{*}\left \|  u^{\varepsilon }(r) \right \| _{\mathcal{H}^{2} }+C_{T}\left \|u^{\varepsilon }(r) \right\| ^{2}_{\mathcal{H}^{1} }+M]\text{d}r
\\&~~~+2\int_{s-d}^{s}\left \langle u^{\varepsilon }(r)-u^{\varepsilon }(s-d), S^{\varepsilon}_{3}(r,x,u^{\varepsilon }(r)) \text{d}W(r) \right \rangle_{\mathcal{H}^{0} }+2\int_{s-d}^{s}\left \langle u^{\varepsilon }(r)-u^{\varepsilon }(s-d), g^{\varepsilon }(r,u^{\varepsilon }(r))\text{d}W(r) \right \rangle_{\mathcal{H}^{0} }
,
	\end{align*}
where
\begin{align*}
		&2\int_{s-d}^{s} [\left \langle u^{\varepsilon }(r)-u^{\varepsilon }(s-d),\eta _{1}^{\varepsilon }(r)S_{1}(u^{\varepsilon }(r)) \right \rangle_{\mathcal{H}^{1}  }\text{d}r
\\&=2\int_{s-d}^{s} \left \langle (I-\Delta )(u^{\varepsilon }(r)-u^{\varepsilon }(s-d)),\eta _{1}^{\varepsilon }(r)S_{1}(u^{\varepsilon }(r)) \right \rangle_{\mathcal{H}^{0}  }\text{d}r
\\&\le \int_{s-d}^{s}[-2a_{1}\left \| \nabla u^{\varepsilon }(r) \right\| ^{2}_{\mathcal{H}^{0} }+2\eta ^{\varepsilon }_{1}(r)\left \langle  \nabla u^{\varepsilon }(s-d),\nabla u^{\varepsilon }(r) \right \rangle_{\mathcal{H}^{0} }]\text{d}r+2a_{2}\int_{s-d}^{s}\left \| u^{\varepsilon }(r) \right\| ^{2}_{\mathcal{H}^{2} }\text{d}r
\\&~~~+\int_{s-d}^{s}\left \| u^{\varepsilon }(s-d) \right\| ^{2}_{\mathcal{H}^{2} }\text{d}r
\\&\le C[\int_{s-d}^{s}\left \| u^{\varepsilon }(r) \right\| ^{2}_{\mathcal{H}^{2} }\text{d}r+\int_{s-d}^{s}\left \| u^{\varepsilon }(s-d) \right\| ^{2}_{\mathcal{H}^{2} }\text{d}r],
	\end{align*}
and by H$\ddot{\text{o}}$lder's inequality and \eqref{k2}, we have
\begin{align*}
		&-2\int_{s-d}^{s} [\left \langle u^{\varepsilon }(r)-u^{\varepsilon }(s-d),\eta _{2}^{\varepsilon }(r)S_{2}(u^{\varepsilon }(r)) \right \rangle_{\mathcal{H}^{1}  }\text{d}r
\\&=-2\int_{s-d}^{s} \left \langle (I-\Delta )(u^{\varepsilon }(r)-u^{\varepsilon }(s-d)),\eta _{2}^{\varepsilon }(r)S_{2}(u^{\varepsilon }(r)) \right \rangle_{\mathcal{H}^{0}  }\text{d}r
\\&\le C(\int_{s-d}^{s}\left \| u^{\varepsilon }(r)-u^{\varepsilon }(s-d) \right \|^{2}_{\mathcal{H}^{2}  }\text{d}r )^{\frac{1}{2} }(\int_{s-d}^{s}\left \| \left | u^{\varepsilon }(r) \right |\left | \nabla u^{\varepsilon }(r) \right | \right \|^{2}_{\mathcal{H}^{0}}\text{d}r)^{\frac{1}{2} }
\\&~~~+C\int_{s-d}^{s}\left \| u^{\varepsilon }(r)-u^{\varepsilon }(s-d) \right \|_{\mathcal{H}^{2}} \left \| \Psi _{N}(\left | u^{\varepsilon }(r) \right |^{2} )u^{\varepsilon }(r) \right \|_{\mathcal{H}^{0}}\text{d}r
\\&\le C[\int_{s-d}^{s}\left \| u^{\varepsilon }(r) \right\| ^{2}_{\mathcal{H}^{2} }\text{d}r+\int_{s-d}^{s}\left \| u^{\varepsilon }(s-d) \right\| ^{2}_{\mathcal{H}^{2} }\text{d}r+\int_{s-d}^{s}\left \| \left | u^{\varepsilon }(r) \right |\left | \nabla u^{\varepsilon }(r) \right | \right \|^{2}_{\mathcal{H}^{0}}\text{d}r
\\&~~~+\int_{s-d}^{s} \left \| u^{\varepsilon }(r) \right \|^{4}_{\mathcal{H}^{1}} \left \| u^{\varepsilon }(r) \right \|^{2}_{\mathcal{H}^{2}}\text{d}r].
	\end{align*}
With the help of Fubini’s theorem, we obtain
\begin{align}\label{a223}
		&\mathbb{E} \int_{nd}^{(n+1)d}\left \| u^{\varepsilon }(s)-u^{\varepsilon }(s-d)  \right \|^{2}_{\mathcal{H}^{0} }   \text{d}s\nonumber
\\&\le dC\mathbb{E} \int_{nd}^{(n+1)d}\left \| u^{\varepsilon }(s-d) \right\| ^{2}_{\mathcal{H}^{2} } \text{d}s+C\mathbb{E} \int_{nd}^{(n+1)d}\int_{s-d}^{s}[\left \|   u^{\varepsilon }(r) \right\| ^{2 }_{\mathcal{H}^{2} }+1]\text{d}r  \text{d}s\nonumber
\\&~~~+\mathbb{E} \int_{nd}^{(n+1)d}\int_{s-d}^{s}\left \| \left | u^{\varepsilon }(r) \right |\left | \nabla u^{\varepsilon }(r) \right | \right \|^{2}_{\mathcal{H}^{0}}+ \left \| u^{\varepsilon }(r) \right \|^{4}_{\mathcal{H}^{1}} \left \| u^{\varepsilon }(r) \right \|^{2}_{\mathcal{H}^{2}}\text{d}r\text{d}s\nonumber
\\&~~~+2\mathbb{E} \int_{nd}^{(n+1)d}[\int_{s-d}^{s}\left \langle u^{\varepsilon }(r)-u^{\varepsilon }(s-d), S^{\varepsilon}_{3}(r,x,u^{\varepsilon }(r)) \text{d}W(r) \right \rangle_{\mathcal{H}^{1} }\nonumber
\\&~~~+\int_{s-d}^{s}\left \langle u^{\varepsilon }(r)-u^{\varepsilon }(s-d), g^{\varepsilon }(r,u^{\varepsilon }(r))\text{d}W(r) \right \rangle_{\mathcal{H}^{1} }]\text{d}s
\\&\le dC[\mathbb{E} \int_{(n-1)d}^{(n+1)d}[\left \|   u^{\varepsilon }(r) \right\| ^{2 }_{\mathcal{H}^{2} }+1]\text{d}r+\mathbb{E} \int_{(n-1)d}^{(n+1)d}\left \| \left | u^{\varepsilon }(r) \right |\left | \nabla u^{\varepsilon }(r) \right | \right \|^{2}_{\mathcal{H}^{0}}+ \left \| u^{\varepsilon }(r) \right \|^{4}_{\mathcal{H}^{1}} \left \| u^{\varepsilon }(r) \right \|^{2}_{\mathcal{H}^{2}}\text{d}r] \nonumber
\\&~~~+2\mathbb{E} \int_{nd}^{(n+1)d}[\int_{s-d}^{s}\left \langle u^{\varepsilon }(r)-u^{\varepsilon }(s-d), S^{\varepsilon}_{3}(r,x,u^{\varepsilon }(r)) \text{d}W(r) \right \rangle_{\mathcal{H}^{1} }\nonumber
\\&~~~+\int_{s-d}^{s}\left \langle u^{\varepsilon }(r)-u^{\varepsilon }(s-d), g^{\varepsilon }(r,u^{\varepsilon }(r))\text{d}W(r) \right \rangle_{\mathcal{H}^{1} }]\text{d}s.\nonumber
	\end{align}
Furthermore, by applying the B-D-G inequality, \textbf{(H2)}, \textbf{(H3)}, H$\ddot{\text{o}}$lder's inequality and Jensen’s inequality we have
\begin{align}\label{a225}
		&\mathbb{E} \int_{nd}^{(n+1)d}[\int_{s-d}^{s}\left \langle u^{\varepsilon }(r)-u^{\varepsilon }(s-d), S^{\varepsilon}_{3}(r,x,u^{\varepsilon }(r)) \text{d}W(r) \right \rangle_{\mathcal{H}^{1} }\nonumber
\\&+\int_{s-d}^{s}\left \langle u^{\varepsilon }(r)-u^{\varepsilon }(s-d), g^{\varepsilon }(r,u^{\varepsilon }(r))\text{d}W(r) \right \rangle_{\mathcal{H}^{1} }]\text{d}s\nonumber
\\&\le C\int_{nd}^{(n+1)d} [\mathbb{E} (\int_{s-d}^{s}\left \| u^{\varepsilon }(r)-u^{\varepsilon }(s-d) \right \|_{\mathcal{H} ^{2}}^{2} \left \| S^{\varepsilon}_{3}(r,x,u^{\varepsilon }(r)) \right \|^{2}_{\mathscr{L}(K,\mathcal{H} ^{0})}\text{d}r)^{\frac{1}{2} }\nonumber
\\&~~~+\mathbb{E} (\int_{s-d}^{s}\left \|  u^{\varepsilon }(r)-u^{\varepsilon }(s-d) \right \|_{\mathcal{H} ^{1}}^{2} \left \| g^{\varepsilon }(r,u^{\varepsilon }(r)) \right \|^{2}_{\mathscr{L}(K,\mathcal{H} ^{1})}\text{d}r)^{\frac{1}{2} }]\text{d}s \nonumber
\\&\le C\int_{nd}^{(n+1)d} [(\mathbb{E} \sup_{r\in[0,T]}\left \| u^{\varepsilon }(r)\right \|_{\mathcal{H} ^{2}}^{2}) ^{\frac{1}{2} } (\mathbb{E}\int_{s-d}^{s} \left \|\nabla u^{\varepsilon }(r)  \right \| _{\mathcal{H} ^{0}}^{2} \text{d}r)^{\frac{1}{2} }
\\&~~~+(\mathbb{E} \sup_{r\in[0,T]}\left \| u^{\varepsilon }(r)\right \|_{\mathcal{H} ^{1}}^{2}+1) ^{\frac{1}{2} } (\mathbb{E}\int_{s-d}^{s}(\left \| u^{\varepsilon }(r)\right \|_{\mathcal{H} ^{1}}^{2}+\left \|u^{\varepsilon }(s-d) \right \|_{\mathcal{H} ^{1}}^{2}) \text{d}r)^{\frac{1}{2} }]\text{d}s \nonumber
\\&\le C(\left \| u^{\varepsilon }_{0}\right \|_{\mathcal{H} ^{1}}^{2}+1)[ (\mathbb{E}\int_{nd}^{(n+1)d} \int_{s-d}^{s} \left \|\nabla u^{\varepsilon }(r)  \right \| _{\mathcal{H} ^{0}}^{2} \text{d}r\text{d}s)^{\frac{1}{2} } \nonumber
\\&~~~ +d^{\frac{1}{2} }(\mathbb{E}\int_{nd}^{(n+1)d} \int_{s-d}^{s}(\left \| u^{\varepsilon }(r)\right \|_{\mathcal{H} ^{1}}^{2}+\left \|u^{\varepsilon }(s-d) \right \|_{\mathcal{H} ^{1}}^{2}) \text{d}r\text{d}s)^{\frac{1}{2} }]\nonumber
\\&\le d^{\frac{1}{2} }C(\left \| u^{\varepsilon }_{0}\right \|_{\mathcal{H} ^{1}}^{2}+1)(\mathbb{E}\int_{(n-1)d}^{(n+1)d}  \left \| u^{\varepsilon }(r)  \right \| _{\mathcal{H} ^{1}}^{2} \text{d}r)^{\frac{1}{2} }
.\nonumber
	\end{align}
Thus, from \eqref{t23}, we have
\begin{align}\label{a226}
		&\mathbb{E} \sum_{n=1}^{N_{d}-1} \int_{nd}^{(n+1)d}\left \| u^{\varepsilon }(s)-u^{\varepsilon }(s-d)  \right \|^{2}_{\mathcal{H}^{0} }  \text{d}s\nonumber
\\&\le dC[\mathbb{E} \int_{0}^{T}[\left \|   u^{\varepsilon }(r) \right\| ^{2 }_{\mathcal{H}^{2} }+1]\text{d}r+\mathbb{E} \int_{0}^{T}\left \| \left | u^{\varepsilon }(r) \right |\left | \nabla u^{\varepsilon }(r) \right | \right \|^{2}_{\mathcal{H}^{0}}+ \left \| u^{\varepsilon }(r) \right \|^{4}_{\mathcal{H}^{1}} \left \| u^{\varepsilon }(r) \right \|^{2}_{\mathcal{H}^{2}}\text{d}r]\nonumber
\\&~~~+d^{\frac{1}{2} }C(\left \| u^{\varepsilon }_{1}\right \|_{\mathcal{H} ^{1}}^{2}+1)\sum_{n=1}^{N_{d}-1}(\mathbb{E}\int_{(n-1)d}^{(n+1)d}  \left \| u^{\varepsilon }(r)  \right \| _{\mathcal{H} ^{1}}^{2} \text{d}r)^{\frac{1}{2} }
\\&\le dC_{T}(\left \| u^{\varepsilon }_{0}\right \|_{\mathcal{H} ^{1}}^{2}+\left \| u^{\varepsilon }_{0}\right \|_{\mathcal{H} ^{1}}^{6}+1)+d^{\frac{1}{2} }C(\left \| u^{\varepsilon }_{0}\right \|_{\mathcal{H} ^{1}}^{2}+1)(\mathbb{E}\int_{0}^{T}  \left \| u^{\varepsilon }(r)  \right \| _{\mathcal{H} ^{1}}^{2} \text{d}r)^{\frac{1}{2} }\nonumber
\\&\le d^{\frac{1}{2} }C_{T}(\left \| u^{\varepsilon }_{0}\right \|_{\mathcal{H} ^{1}}^{2}+\left \| u^{\varepsilon }_{0}\right \|_{\mathcal{H} ^{1}}^{6}+1).\nonumber
	\end{align}
Similarly, we have
\begin{align}\label{a227}
		\mathbb{E} \sum_{n=1}^{N_{d}-1} \int_{nd}^{(n+1)d}\left \| u^{\varepsilon }(s-d)-u^{\varepsilon }(nd)  \right \|^{2}_{\mathcal{H}^{1} }  \text{d}s\le d^{\frac{1}{2} }C_{T}(\left \| u^{\varepsilon }_{0}\right \|_{\mathcal{H} ^{1}}^{2}+\left \| u^{\varepsilon }_{0}\right \|_{\mathcal{H} ^{2}}^{6}+1).
	\end{align}
 Substituting  \eqref{a226} and \eqref{a227} into \eqref{a222} gives
\begin{align}\label{a228}
		\mathbb{E} \int_{0}^{T}\left \| u^{\varepsilon }(s;u^{\varepsilon }_{0})-\bar{u}^{\varepsilon }(s;u^{\varepsilon }_{0})  \right \|^{2}_{\mathcal{H}^{1} }  \text{d}s\le C_{T}(\left \| u^{\varepsilon }_{0}\right \|_{\mathcal{H} ^{1}}^{2}+\left \| u^{\varepsilon }_{0}\right \|_{\mathcal{H} ^{1}}^{6}+1)d^{\frac{1}{2} }.
	\end{align}
Then, by applying Theorem \ref{co1} and  following the same steps as in the proof of \eqref{a228}, we have
\begin{align}\label{a229}
		\mathbb{E} \int_{0}^{T} \left \| u^{* }(s;u_{0} ^{* }) -\overline{u}^{*}(s;u_{0} ^{* }) \right \|^{2}_{\mathcal{H} ^{1}} \text{d}s\le C_{T}(\left \| u_{0} ^{* } \right \|^{2}_{\mathcal{H} ^{1}}+\left \| u_{0} ^{* } \right \|^{6}_{\mathcal{H} ^{2}}+1 )d^{\frac{1}{2} }.
	\end{align}

\end{proof}

Now we establish the following the averaging principle of  stochastic tamed 3D Navier–Stokes equations  with  locally weak
monotonicity coefficients. Below we need additional conditions:
\\\textbf{(H6)}    The functions $f$, $g$ satisfy,  for all $t\in[0,T]$ and $u,v\in L^{2}(\mathbb{D} ;\mathbb{R}^{3})  $ with $\left \| u-v \right \| _{L^{2}}\le\zeta$,
\begin{align*}
		\|f(t,u)- f(t,v )\|^{2}_{L^{2}}  \le c\mathcal{A}(\left \| u -v  \right \|^{2} _{L^{2}}),
	\end{align*}
and for all $t\in[0,T]$ and $u,v\in \mathcal{H}^{1}(\mathbb{D} ;\mathbb{R}^{3})  $ with $\left \| u-v \right \| _{\mathcal{H}^{1}}\le\zeta $,
\begin{align*}
\|f(t,u)- f(t,v )\|^{2}_{\mathcal{H}^{1}}  \le c\mathcal{A}(\left \| u -v  \right \|^{2} _{\mathcal{H}^{1}}).
	\end{align*}
\begin{theorem}\label{th2} Consider \eqref{a2} and \eqref{a222}. Suppose that the assumptions \textbf{(H1)}-\textbf{(H6)} hold and  $a_{4}\le2a_{3}$.  For any initial values $u_{0}^{\varepsilon }, u_{0} ^{*}\in \mathcal{H} ^{1}$ and $T > 0$, assume further that $\lim_{\varepsilon  \to 0} \left \|u_{0}^{\varepsilon }- u_{0} ^{*} \right \|^{2}_{\mathcal{H} ^{0}}=0$. Then, we have
\begin{align}\label{c1}
		\lim_{\varepsilon  \to 0} \mathbb{E}\underset{t\in [0,T]}{\sup}   \left \|u^{\varepsilon }(t;u_{0}^{\varepsilon }) -u^{* }(t;u_{0} ^{*}) \right \|^{2}_{\mathcal{H} ^{0}}=0.
	\end{align}
\end{theorem}
\begin{proof} Let $\tau _{R}=\underset{t\in[0,T]}{\inf} \{\left \| u^{\varepsilon}(t;u^{\varepsilon}_{0} ) \right \| _{\mathcal{H} ^{1}}^{2}\vee \left \| u^{*}(t;u^{*}_{0}) \right \| _{\mathcal{H} ^{1}}^{2}>R \}$. By applying It$\hat{\text{o}} $'s formula formula to $u^{\varepsilon }(t;u_{0} ^{\varepsilon }) -u^{* }(t;u_{0}^{*})$, and utilizing \textbf{(H3)}, Young's inequality and the B-D-G inequality, we obtain
\begin{align}\label{t333}
		 & \mathbb{E} (\underset{t\in[0,T\wedge \tau _{R}]}{\sup}\left \| u^{\varepsilon}(t;u^{\varepsilon}_{0} )-u^{*}(t;u^{*}_{0}) \right \| _{\mathcal{H} ^{0}}^{2}) \nonumber
\\&=\left \| u^{\varepsilon}_{0}-u^{*}_{0} \right \| _{\mathcal{H} ^{0}}^{2}+\mathbb{E}\underset{t\in[0,T\wedge \tau _{R}]}{\sup}\int_{0}^{t} [2 \left \langle u^{\varepsilon}(s)-u^{*}(s),\eta _{1}^{\varepsilon } (s)S_{1}(u^{\varepsilon}(s) )-\eta _{1}^{* } S_{1}(u^{*}(s) ) \right \rangle _{\mathcal{H} ^{0}}\nonumber
\\&~~~-2  \left \langle u^{\varepsilon}(s)-u^{*}(s),\eta _{2}^{\varepsilon } (s)S_{2}(u^{\varepsilon}(s) )-\eta _{2}^{* } S_{2}(u^{*}(s) ) \right \rangle _{\mathcal{H} ^{0}}+2 \left \langle u^{\varepsilon}(s)-u^{*}(s),f^{\varepsilon}(s,u^{\varepsilon}(s))-f^{*}(u^{*}(s)) \right \rangle _{\mathcal{H} ^{0}}\nonumber
\\&~~~+\left \| S^{\varepsilon}_{3}(s,x,u^{\varepsilon }(s))-S^{*}_{3}(x,u^{*}(s)) \right \|_{\mathscr{L}(K,\mathcal{H} ^{0})} ^{2}+\left \| g^{\varepsilon}(s,u^{\varepsilon}(s))-g^{*}(u^{*}(s)) \right \|_{\mathscr{L}(K,\mathcal{H} ^{0})} ^{2}]\text{d}s   \nonumber
\\&~~~+2\mathbb{E} (\underset{t\in[0,T\wedge \tau _{R}]}{\sup}\int_{0}^{t}\left \langle u^{\varepsilon}(s)-u^{*}(s),[G^{\varepsilon}(t,u^{\varepsilon}(s))-G^{*}(u^{*}(s))]\text{d}W(s) \right \rangle _{\mathcal{H} ^{0} })
\\&\le \left \| u^{\varepsilon}_{0}-u^{*}_{0} \right \| _{\mathcal{H} ^{0}}^{2}+\mathbb{E}\underset{t\in[0,T\wedge \tau _{R}]}{\sup}\int_{0}^{t} [2 \left \langle u^{\varepsilon}(s)-u^{*}(s),\eta _{1}^{\varepsilon } (s)S_{1}(u^{\varepsilon}(s) )-\eta _{1}^{* } S_{1}(u^{*}(s) ) \right \rangle _{\mathcal{H} ^{0}}\nonumber
\\&~~~-2  \left \langle u^{\varepsilon}(s)-u^{*}(s),\eta _{2}^{\varepsilon } (s)S_{2}(u^{\varepsilon}(s) )-\eta _{2}^{* } S_{2}(u^{*}(s) ) \right \rangle _{\mathcal{H} ^{0}}\nonumber
\\&~~~+2 \left \langle u^{\varepsilon}(s)-u^{*}(s),f^{\varepsilon}(s,u^{\varepsilon}(s))-f^{*}(u^{*}(s)) \right \rangle _{\mathcal{H} ^{0}}]\text{d}s +\frac{1}{2}\mathbb{E} (\underset{t\in[0,T\wedge \tau _{R}]}{\sup}\left \| u^{\varepsilon}(t;u^{\varepsilon}_{0} )-u^{*}(t;u^{*}_{0}) \right \| _{\mathcal{H} ^{0}}^{2})  \nonumber
\\&~~~+73\mathbb{E} (\int_{0}^{T\wedge \tau _{R}}[\left \| S^{\varepsilon}_{3}(s,x,u^{\varepsilon }(s))-S^{*}_{3}(x,u^{*}(s)) \right \|_{\mathscr{L}(K,\mathcal{H} ^{0})} ^{2}+\left \| g^{\varepsilon}(s,u^{\varepsilon}(s))-g^{*}(u^{*}(s)) \right \|_{\mathscr{L}(K,\mathcal{H} ^{0})} ^{2}]\text{d}s, \nonumber
	\end{align}
which implies that
 \begin{align}\label{t335}
		 & \mathbb{E} (\underset{t\in[0,T\wedge \tau _{R}]}{\sup}\left \| u^{\varepsilon}(t;u^{\varepsilon}_{0} )-u^{*}(t;u^{*}_{0}) \right \| _{\mathcal{H} ^{0}}^{2}) \nonumber
 \\&\le 2\left \| u^{\varepsilon}_{0}-u^{*}_{0} \right \| _{\mathcal{H} ^{0}}^{2}+\mathbb{E}\underset{t\in[0,T\wedge \tau _{R}]}{\sup}\int_{0}^{t} [4 \left \langle u^{\varepsilon}(s)-u^{*}(s),\eta _{1}^{\varepsilon } (s)S_{1}(u^{\varepsilon}(s) )-\eta _{1}^{* } S_{1}(u^{*}(s) ) \right \rangle _{\mathcal{H} ^{0}}\nonumber
\\&~~~-4  \left \langle u^{\varepsilon}(s)-u^{*}(s),\eta _{2}^{\varepsilon } (s)S_{2}(u^{\varepsilon}(s) )-\eta _{2}^{* } S_{2}(u^{*}(s) ) \right \rangle _{\mathcal{H} ^{0}}
\\&~~~+4 \left \langle u^{\varepsilon}(s)-u^{*}(s),f^{\varepsilon}(s,u^{\varepsilon}(s))-f^{*}(u^{*}(s)) \right \rangle _{\mathcal{H} ^{0}}]\text{d}s   \nonumber
\\&~~~+146\mathbb{E} (\int_{0}^{T\wedge \tau _{R}}[\left \| S^{\varepsilon}_{3}(s,x,u^{\varepsilon }(s))-S^{*}_{3}(x,u^{*}(s)) \right \|_{\mathscr{L}(K,\mathcal{H} ^{0})} ^{2}+\left \| g^{\varepsilon}(s,u^{\varepsilon}(s))-g^{*}(u^{*}(s)) \right \|_{\mathscr{L}(K,\mathcal{H} ^{0})} ^{2}]\text{d}s \nonumber
\\&:=2\left \| u^{\varepsilon}_{0}-u^{*}_{0} \right \| _{\mathcal{H} ^{0}}^{2}+4\Sigma _{1}+4\Sigma _{2}+4\Sigma _{3}+146\Sigma _{4}+146\Sigma _{5}. \nonumber
	\end{align}
Then for $\Sigma _{1}$,
\begin{align}\label{s1}
	&\mathbb{E}\underset{t\in [0,T]}{\sup}\int_{0}^{t}\left \langle u^{\varepsilon}(s)-u^{*}(s),\eta _{1}^{\varepsilon } (s)S_{1}(u^{\varepsilon}(s) )-\eta _{1}^{* } S_{1}(u^{*}(s) ) \right \rangle _{\mathcal{H} ^{0}}\text{d}s\nonumber
\\&= \mathbb{E}\underset{t\in [0,T]}{\sup}\int_{0}^{t}\left \langle u^{\varepsilon}(s)-u^{*}(s),\eta _{1}^{\varepsilon } (s)S_{1}(u^{\varepsilon}(s) )-\eta _{1}^{\varepsilon}(s) S_{1}(u^{*}(s) ) \right \rangle _{\mathcal{H} ^{0}}\text{d}s
\\&~~~+\mathbb{E}\underset{t\in [0,T]}{\sup}\int_{0}^{t}\left \langle u^{\varepsilon}(s)-u^{*}(s),\eta _{1}^{\varepsilon } (s)S_{1}(u^{*}(s) )-\eta _{1}^{* } S_{1}(u^{*}(s) ) \right \rangle _{\mathcal{H} ^{0}}\text{d}s\nonumber
\\&:=\Sigma _{1}^{1}+\Sigma _{1}^{2}. \nonumber
	\end{align}
By  \eqref{t3}, we have
\begin{align}\label{s2}
	\Sigma _{1}^{1}&=\mathbb{E}\underset{t\in [0,T]}{\sup}\int_{0}^{t}\left \langle u^{\varepsilon}(s)-u^{*}(s),\eta _{1}^{\varepsilon } (s)S_{1}(u^{\varepsilon}(s) )-\eta _{1}^{\varepsilon}(s) S_{1}(u^{*}(s) ) \right \rangle _{\mathcal{H} ^{0}}\text{d}s \nonumber
\\&=-a_{1}\mathbb{E}\underset{t\in [0,T]}{\sup}\int_{0}^{t}\left \| \nabla u^{\varepsilon}(s)-\nabla u^{*}(s) \right \| ^{2}_{\mathcal{H} ^{0}}\text{d}s,
	\end{align}
then
\begin{align}\label{s3}
	\Sigma _{1}^{2}&=\mathbb{E}\underset{t\in [0,T]}{\sup}\int_{0}^{t}\left \langle u^{\varepsilon}(s)-u^{*}(s),\eta _{1}^{\varepsilon } (s)S_{1}(u^{*}(s) )-\eta _{1}^{* } S_{1}(u^{*}(s) ) \right \rangle _{\mathcal{H} ^{0}}\text{d}s,\nonumber
\\&\le \mathbb{E}\underset{t\in [0,T]}{\sup}\int_{0}^{t}\left \langle u^{\varepsilon}(s)-\overline{u }^{\varepsilon}(s),\eta _{1}^{\varepsilon } (s)S_{1}(u^{*}(s) )-\eta _{1}^{* } S_{1}(u^{*}(s) ) \right \rangle _{\mathcal{H} ^{0}}\text{d}s\nonumber
\\&~~~+\mathbb{E}\underset{t\in [0,T]}{\sup}\int_{0}^{t}\left \langle \overline{u }^{\varepsilon}(s)-\overline{u}^{*}(s),\eta _{1}^{\varepsilon } (s)S_{1}(u^{*}(s) )-\eta _{1}^{* } S_{1}(u^{*}(s) ) \right \rangle _{\mathcal{H} ^{0}}\text{d}s
\\&~~~+\mathbb{E}\underset{t\in [0,T]}{\sup}\int_{0}^{t}\left \langle \overline{u}^{*}(s)-u^{*}(s),\eta _{1}^{\varepsilon } (s)S_{1}(u^{*}(s) )-\eta _{1}^{* } S_{1}(u^{*}(s) ) \right \rangle _{\mathcal{H} ^{0}}\text{d}s\nonumber
\\&:=\Sigma _{1}^{2,1}+\Sigma _{1}^{2,2}+\Sigma _{1}^{2,3}.\nonumber
	\end{align}
For $\Sigma _{1}^{2,1}$, by H\"older's inequality, \eqref{t23} and Lemma \ref{le1}, we get
\begin{align}\label{s5}
	\Sigma _{1}^{2,1}&=\mathbb{E}\underset{t\in [0,T]}{\sup}\int_{0}^{t}\left \langle u^{\varepsilon}(s)-\overline{u }^{\varepsilon}(s),\eta _{1}^{\varepsilon } (s)S_{1}(u^{*}(s) )-\eta _{1}^{* } S_{1}(u^{*}(s) ) \right \rangle _{\mathcal{H} ^{0}}\text{d}s\nonumber
\\&\le\mathbb{E}\underset{t\in [0,T]}{\sup}\int_{0}^{t}\left \| u^{\varepsilon}(s)-\overline{u }^{\varepsilon}(s) \right \|_{\mathcal{H} ^{0}}[ \left \|\eta _{1}^{\varepsilon } (s)S_{1}(u^{*}(s) )\right \|_{\mathcal{H} ^{0}}+ \left \|\eta _{1}^{* } S_{1}(u^{*}(s) ) \right \| _{\mathcal{H} ^{0}}]\text{d}s\nonumber
\\&\le 2a_{2}(\mathbb{E}\int_{0}^{T} \left \|u^{*}(s) \right \|_{\mathcal{H} ^{2}}^{2}\text{d}s)^{\frac{1}{2} }(\mathbb{E}\int_{0}^{T} \left \| u^{\varepsilon}(s)-\overline{u }^{\varepsilon}(s) \right \|_{\mathcal{H} ^{0}}^{2}\text{d}s)^{\frac{1}{2} }
\\&\le C_{T}(\left \| u_{0}^{*} \right \|_{\mathcal{H}^{1} }^{2}+\left \| u_{0} ^{\varepsilon } \right \|^{2}_{\mathcal{H} ^{1}}+\left \| u_{0} ^{\varepsilon } \right \|^{6}_{\mathcal{H} ^{1}}+1 )d^{\frac{1}{4} }.\nonumber
	\end{align}
Similarly, for $\Sigma _{1}^{2,3}$,
\begin{align}\label{s7}
	\Sigma _{1}^{2,3}&=\mathbb{E}\underset{t\in [0,T]}{\sup}\int_{0}^{t}\left \langle \overline{u}^{*}(s)-u^{*}(s),\eta _{1}^{\varepsilon } (s)S_{1}(u^{*}(s) )-\eta _{1}^{* } S_{1}(u^{*}(s) ) \right \rangle _{\mathcal{H} ^{0}}\text{d}s
\\&\le C_{T}(\left \| u_{0}^{*} \right \|_{\mathcal{H}^{1} }^{2}+\left \| u_{0} ^{* } \right \|^{6}_{\mathcal{H} ^{1}}+1 )d^{\frac{1}{4} }. \nonumber
	\end{align}
Next, the key problem is to estimate $\Sigma _{1}^{2,2}$:
\begin{align}\label{s8}
	\Sigma _{1}^{2,2}&=\mathbb{E}\underset{t\in [0,T]}{\sup}\int_{0}^{t}\left \langle \overline{u }^{\varepsilon}(s)-\overline{u}^{*}(s),\eta _{1}^{\varepsilon } (s)S_{1}(u^{*}(s) )-\eta _{1}^{* } S_{1}(u^{*}(s) ) \right \rangle _{\mathcal{H} ^{0}}\text{d}s \nonumber
\\&\le\mathbb{E}\underset{t\in [0,T]}{\sup}\int_{0}^{t}\left \langle \overline{u }^{\varepsilon}(s)-\overline{u}^{*}(s),\eta _{1}^{\varepsilon } (s)S_{1}(u^{*}(s) )-\eta _{1}^{\varepsilon }(s) S_{1}(\overline{u }^{*}(s) ) \right \rangle _{\mathcal{H} ^{0}}\text{d}s \nonumber
\\&~~~+\mathbb{E}\underset{t\in [0,T]}{\sup}\int_{0}^{t}\left \langle \overline{u }^{\varepsilon}(s)-\overline{u}^{*}(s),\eta _{1}^{\varepsilon }(s) S_{1}(\overline{u }^{*}(s) )-\eta _{1}^{* } S_{1}(\overline{u }^{*}(s) ) \right \rangle _{\mathcal{H} ^{0}}\text{d}s
\\&~~~+\mathbb{E}\underset{t\in [0,T]}{\sup}\int_{0}^{t}\left \langle \overline{u }^{\varepsilon}(s)-\overline{u}^{*}(s),\eta _{1}^{* } S_{1}(\overline{u }^{*}(s) )-\eta _{1}^{* } S_{1}(u^{*}(s) ) \right \rangle _{\mathcal{H} ^{0}}\text{d}s \nonumber
\\&:=\Sigma _{1}^{2,2,1}+\Sigma _{1}^{2,2,2}+\Sigma _{1}^{2,2,3}, \nonumber
	\end{align}
where
\begin{align}\label{s9}
	\Sigma _{1}^{2,2,1}&=\mathbb{E}\underset{t\in [0,T]}{\sup}\int_{0}^{t}\left \langle \overline{u }^{\varepsilon}(s)-\overline{u}^{*}(s),\eta _{1}^{\varepsilon } (s)S_{1}(u^{*}(s) )-\eta _{1}^{\varepsilon }(s) S_{1}(\overline{u }^{*}(s) ) \right \rangle _{\mathcal{H} ^{0}}\text{d}s\nonumber
\\&\le a_{2}\mathbb{E}\underset{t\in [0,T]}{\sup}\int_{0}^{t}\left \langle \Delta \overline{u }^{\varepsilon}(s)-\Delta \overline{u}^{*}(s),u^{*}(s) -\overline{u }^{*}(s)  \right \rangle _{\mathcal{H} ^{0}}\text{d}s\nonumber
\\&\le a_{2}\mathbb{E}\int_{0}^{T}\left \| \Delta \overline{u }^{\varepsilon}(s)-\Delta \overline{u}^{*}(s)\right\| _{\mathcal{H} ^{0}}\left \|u^{*}(s) -\overline{u }^{*}(s)  \right \| _{\mathcal{H} ^{0}}\text{d}s
\\&\le a_{2}(\mathbb{E}\int_{0}^{T} \left \| \overline{u }^{\varepsilon}(s)\right \|_{\mathcal{H} ^{2}}^{2}+\left \| \overline{u }^{*}(s)\right \|_{\mathcal{H} ^{2}}^{2}\text{d}s)^{\frac{1}{2} }(\mathbb{E}\int_{0}^{T} \left \| u^{*}(s)-\overline{u }^{*}(s) \right \|_{\mathcal{H} ^{0}}^{2}\text{d}s)^{\frac{1}{2} }\nonumber
\\&\le C_{T}(\left \| u_{0}^{\varepsilon} \right \|_{\mathcal{H}^{1} }^{2}+\left \| u_{0} ^{* } \right \|^{2}_{\mathcal{H} ^{1}}+\left \| u_{0} ^{* } \right \|^{6}_{\mathcal{H} ^{1}}+1 )d^{\frac{1}{4} },\nonumber
	\end{align}
and
\begin{align}\label{s10}
	\Sigma _{1}^{2,2,3}&=\mathbb{E}\underset{t\in [0,T]}{\sup}\int_{0}^{t}\left \langle \overline{u }^{\varepsilon}(s)-\overline{u}^{*}(s),\eta _{1}^{* } S_{1}(\overline{u }^{*}(s) )-\eta _{1}^{* } S_{1}(u^{*}(s) ) \right \rangle _{\mathcal{H} ^{0}}\text{d}s\nonumber
\\&\le C_{T}(\left \| u_{0}^{\varepsilon} \right \|_{\mathcal{H}^{1} }^{2}+\left \| u_{0} ^{* } \right \|^{2}_{\mathcal{H} ^{1}}+\left \| u_{0} ^{* } \right \|^{6}_{\mathcal{H} ^{1}}+1 )d^{\frac{1}{4} },
	\end{align}
In the subsequent step, we will employ the time discretization technique to address $\Sigma _{1}^{2,2,2}$. Let $\left [ t \right ] $ denote the integer part of $t$; then, it is important to observe that
\begin{align}\label{s11}
	\Sigma _{1}^{2,2,2}&=\mathbb{E}\underset{t\in [0,T]}{\sup}\int_{0}^{t}\left \langle \overline{u }^{\varepsilon}(s)-\overline{u}^{*}(s),\eta _{1}^{\varepsilon }(s) S_{1}(\overline{u }^{*}(s) )-\eta _{1}^{* } S_{1}(\overline{u }^{*}(s) ) \right \rangle _{\mathcal{H} ^{0}}\text{d}s  \nonumber
\\&=\mathbb{E}\underset{t\in [0,T]}{\sup}\sum_{n=0}^{\left [ t/d  \right ]-1 } \int_{nd}^{(n+1)d}\left \langle u ^{\varepsilon}(nd)-u^{*}(nd),\eta _{1}^{\varepsilon }(s) S_{1}(u ^{*}(nd) )-\eta _{1}^{* } S_{1}(u ^{*}(nd) ) \right \rangle _{\mathcal{H} ^{0}}\text{d}s  \nonumber
\\&~~~+\mathbb{E}\int_{\left [ t/d  \right ]d}^{t}\left \langle u ^{\varepsilon}(\left [ t/d  \right ]d)-u^{*}(\left [ t/d  \right ]d),\eta _{1}^{\varepsilon }(s) S_{1}(u ^{*}(\left [ t/d  \right ]d) )-\eta _{1}^{* } S_{1}(u ^{*}(\left [ t/d  \right ]d) ) \right \rangle _{\mathcal{H} ^{0}}\text{d}s  \nonumber
\\&\le -\mathbb{E}\underset{t\in [0,T]}{\sup}\sum_{n=0}^{\left [ t/d  \right ]-1 } \int_{nd}^{(n+1)d}\left \langle \nabla u ^{\varepsilon}(nd)-\nabla u^{*}(nd),\eta _{1}^{\varepsilon }(s) \nabla u ^{*}(nd) -\eta _{1}^{* } \nabla u ^{*}(nd)  \right \rangle _{\mathcal{H} ^{0}}\text{d}s  \nonumber
\\&~~~+C(\left \| u_{0}^{\varepsilon} \right \|_{\mathcal{H}^{1} }^{2}+\left \| u_{0} ^{* } \right \|^{2}_{\mathcal{H} ^{1}}+1 )d\nonumber
\\&\le \mathbb{E}\underset{t\in [0,T]}{\sup}\sum_{n=0}^{\left [t/d  \right ]-1 } \left \|\int_{nd}^{(n+1)d} \eta _{1}^{\varepsilon }(s) \nabla u ^{*}(nd) -\eta _{1}^{* }\nabla u ^{*}(nd)  \text{d}s\right \|_{\mathcal{H} ^{0}}\left \| u ^{\varepsilon}(nd)-u ^{*}(nd) \right \| _{\mathcal{H} ^{1}}
\\&~~~+C(\left \| u_{0}^{\varepsilon} \right \|_{\mathcal{H}^{1} }^{2}+\left \| u_{0} ^{* } \right \|^{2}_{\mathcal{H} ^{1}}+1 )d \nonumber
\\&\le\sum_{n=0}^{\left [ T/d  \right ]-1 } (\left \|\mathbb{E}\int_{nd}^{(n+1)d} \eta _{1}^{\varepsilon }(s) \nabla u ^{*}(nd) -\eta _{1}^{* }\nabla u ^{*}(nd)  \text{d}s\right \|^{2}_{\mathcal{H} ^{0}})^{\frac{1}{2} } (\mathbb{E} \left \|u ^{\varepsilon}(nd)-u ^{*}(nd) \right \| ^{2}_{\mathcal{H} ^{1}} )^{\frac{1}{2} } \nonumber
\\&~~~+C(\left \| u_{0}^{\varepsilon} \right \|_{\mathcal{H}^{1} }^{2}+\left \| u_{0} ^{* } \right \|^{2}_{\mathcal{H} ^{1}}+1 )d \nonumber
\\&\le \frac{T}{d} \max_{0\le n\le \left [ T/d \right ]-1, n\in \mathbb{N}^{+} } (\left \|\mathbb{E}\int_{nd}^{(n+1)d} \eta _{1}^{\varepsilon }(s)\nabla u ^{*}(nd) -\eta _{1}^{* }\nabla u ^{*}(nd) \text{d}s\right \|^{2}_{\mathcal{H} ^{0}})^{\frac{1}{2} }\cdot C(\left \| u_{0}^{\varepsilon} \right \|_{\mathcal{H}^{1} }^{2}+\left \| u_{0} ^{* } \right \|^{2}_{\mathcal{H} ^{1}}+1 ) \nonumber
 \\&~~~+C(\left \| u_{0}^{\varepsilon} \right \|_{\mathcal{H}^{1} }^{2}+\left \| u_{0} ^{* } \right \|^{2}_{\mathcal{H} ^{1}}+1 )d, \nonumber
	\end{align}
where
\begin{align}\label{s12}
	&(\left \|\mathbb{E}\int_{nd}^{(n+1)d} \eta _{1}^{\varepsilon }(s)\nabla u ^{*}(nd) -\eta _{1}^{* }\nabla u ^{*}(nd) \text{d}s\right \|^{2}_{\mathcal{H} ^{0}} )^{\frac{1}{2} }\nonumber
\\&=(\mathbb{E}\left \| u ^{*}(nd)  \right \|^{2}_{\mathcal{H} ^{1}}\left |\int_{nd}^{(n+1)d} \eta _{1}^{\varepsilon }(s) -\eta _{1}^{* }  \text{d}s\right |^{2} )^{\frac{1}{2} }\nonumber
\\&\le C_{T}(\left \| u ^{*}_{0}  \right \|^{2}_{\mathcal{H} ^{1}}+1)(\left |\int_{nd}^{(n+1)d} \eta _{0}(\frac{s }{\varepsilon } ) -\eta _{1}^{* }  \text{d}s\right |^{2} )^{\frac{1}{2} }
\\&\le C_{T}(\left \| u ^{*}_{0}  \right \|^{2}_{\mathcal{H} ^{1}}+1)\varepsilon \left |\int_{\frac{nd}{\varepsilon }}^{\frac{(n+1)d}{\varepsilon }} \eta _{1}(r) -\eta _{1}^{* }\text{d}r\right | \nonumber
\\&\le C_{T}\mathcal{R} _{1}(\frac{d}{\varepsilon } )(\left \| u ^{*}_{0}  \right \|^{2}_{\mathcal{H} ^{1}}+1)d.\nonumber
	\end{align}
Substituting \eqref{s2}-\eqref{s12} into \eqref{s1} implies
\begin{align}\label{a19}
	\Sigma _{1}&=\mathbb{E}\underset{t\in [0,T\wedge \tau _{R}]}{\sup}\int_{0}^{t}\left \langle u^{\varepsilon}(s)-u^{*}(s),\eta _{1}^{\varepsilon } (s)S_{1}(u^{\varepsilon}(s) )-\eta _{1}^{* } S_{1}(u^{*}(s) ) \right \rangle _{\mathcal{H} ^{0}}\text{d}s\nonumber
\\&\le-a_{1}\mathbb{E}\underset{t\in [0,T\wedge \tau _{R}]}{\sup}\int_{0}^{t}\left \| \nabla u^{\varepsilon}(s)-\nabla u^{*}(s) \right \| ^{2}_{\mathcal{H} ^{0}}\text{d}s
\\&~~~+C_{T}(\left \| u_{0}^{\varepsilon} \right \|_{\mathcal{H}^{1} }^{2}+\left \| u_{0} ^{* } \right \|^{2}_{\mathcal{H} ^{1}}+\left \| u_{0}^{\varepsilon} \right \|_{\mathcal{H}^{1} }^{6}+\left \| u_{0} ^{* } \right \|^{6}_{\mathcal{H} ^{1}}+1 )[d^{\frac{1}{4} }+d+\mathcal{R} _{1}(\frac{d}{\varepsilon } )]. \nonumber
	\end{align}
We shall now estimate $\Sigma _{2}$:
\begin{align}\label{s13}
	\Sigma _{2}&=-\mathbb{E}\underset{t\in [0,T\wedge \tau _{R}]}{\sup}\int_{0}^{t}\left \langle u^{\varepsilon}(s)-u^{*}(s),\eta _{2}^{\varepsilon } (s)S_{2}(u^{\varepsilon}(s) )-\eta _{2}^{* } S_{2}(u^{*}(s) ) \right \rangle _{\mathcal{H} ^{0}}\text{d}s\nonumber
\\&= -\mathbb{E}\underset{t\in [0,T\wedge \tau _{R}]}{\sup}\int_{0}^{t}\left \langle u^{\varepsilon}(s)-u^{*}(s),\eta _{2}^{\varepsilon } (s)S_{2}(u^{\varepsilon}(s) )-\eta _{2}^{\varepsilon}(s) S_{2}(u^{*}(s) ) \right \rangle _{\mathcal{H} ^{0}}\text{d}s
\\&~~~-\mathbb{E}\underset{t\in [0,T\wedge \tau _{R}]}{\sup}\int_{0}^{t}\left \langle u^{\varepsilon}(s)-u^{*}(s),\eta _{2}^{\varepsilon } (s)S_{2}(u^{*}(s) )-\eta _{2}^{* } S_{2}(u^{*}(s) ) \right \rangle _{\mathcal{H} ^{0}}\text{d}s\nonumber
\\&:=\Sigma _{2}^{1}+\Sigma _{2}^{2}. \nonumber
	\end{align}
For $\Sigma _{2}^{1}$, by  Young’s inequality, H\"older's inequality and \eqref{k1}, we have
\begin{align}\label{s15}
	\Sigma _{2}^{1}&=-\mathbb{E}\underset{t\in [0,T\wedge \tau _{R}]}{\sup}\int_{0}^{t}\left \langle u^{\varepsilon}(s)-u^{*}(s),\eta _{2}^{\varepsilon } (s)S_{2}(u^{\varepsilon}(s) )-\eta _{2}^{\varepsilon}(s) S_{2}(u^{*}(s) ) \right \rangle _{\mathcal{H} ^{0}}\text{d}s\nonumber
\\&=-\mathbb{E}\underset{t\in [0,T\wedge \tau _{R}]}{\sup}\int_{0}^{t}\left \langle u^{\varepsilon}(s)-u^{*}(s),\eta _{2}^{\varepsilon } (s)(u^{\varepsilon }(s),\nabla)u^{\varepsilon }(s)-\eta _{2}^{\varepsilon}(s) (u^{* }(s),\nabla)u^{* }(s) \right \rangle _{\mathcal{H} ^{0}}\text{d}s\nonumber
\\&~~~-\mathbb{E}\underset{t\in [0,T\wedge \tau _{R}]}{\sup}\int_{0}^{t}\left \langle u^{\varepsilon}(s)-u^{*}(s),\eta _{2}^{\varepsilon } (s)\Psi _{N}(\left | u^{\varepsilon }(s) \right |^{2} )u^{\varepsilon }(s)-\eta _{2}^{\varepsilon } (s)\Psi _{N}(\left | u^{* }(s) \right |^{2} )u^{* }(s) \right \rangle _{\mathcal{H} ^{0}}\text{d}s\nonumber
\\&\le \frac{a_{1}}{4}   \mathbb{E}\int_{0}^{T\wedge \tau _{R}}\left \| \nabla u^{\varepsilon}(s)-\nabla u^{*}(s) \right \| ^{2}_{\mathcal{H} ^{0}}\text{d}s+\frac{a_{4}}{a_{1}} \mathbb{E}\int_{0}^{T\wedge \tau _{R}}\left \| (u^{\varepsilon}(s))^{T} u^{\varepsilon}(s)-(u^{*}(s))^{T} u^{*}(s) \right \| ^{2}_{\mathcal{H} ^{0}}\text{d}s\nonumber
\\&~~~+4a_{4}\mathbb{E}\int_{0}^{T\wedge \tau _{R}}\left \| \left | u^{\varepsilon}(s)-u^{*}(s) \right | (\left | u^{\varepsilon}(s) \right |+\left | u^{*}(s) \right |) \right \| ^{2}_{\mathcal{H} ^{0}}\text{d}s
\\&\le \frac{a_{1}}{4}  \mathbb{E}\int_{0}^{T\wedge \tau _{R}}\left \| \nabla u^{\varepsilon}(s)-\nabla u^{*}(s) \right \| ^{2}_{\mathcal{H} ^{0}}\text{d}s\nonumber
\\&~~~+C_{a_{1},a_{4} }\mathbb{E}\int_{0}^{T\wedge \tau _{R}}\left \| u^{\varepsilon}(s)-u^{*}(s) \right \| ^{2}_{L^{4}}[\left \| u^{\varepsilon}(s)\right \| ^{2}_{L ^{4}}+\left \|u^{*}(s) \right \| ^{2}_{L ^{4}}\text{d}s\nonumber
\\&\le \frac{a_{1} }{4}  \mathbb{E}\int_{0}^{T\wedge \tau _{R}}\left \| \nabla u^{\varepsilon}(s)-\nabla u^{*}(s) \right \| ^{2}_{\mathcal{H} ^{0}}\text{d}s\nonumber
\\&~~~+C_{a_{1},a_{4} }\mathbb{E}\int_{0}^{T\wedge \tau _{R}}\left \| u^{\varepsilon}(s)-u^{*}(s) \right \| ^{\frac{1}{2} }_{\mathcal{H} ^{0}}\left \| u^{\varepsilon}(s)-u^{*}(s) \right \| ^{\frac{3}{2} }_{\mathcal{H} ^{1}}[\left \| u^{\varepsilon}(s)\right \| ^{2}_{\mathcal{H} ^{1}}+\left \|u^{*}(s) \right \| ^{2}_{\mathcal{H} ^{1}}]\text{d}s\nonumber
\\&\le \frac{a_{1} }{2}    \mathbb{E}\int_{0}^{T\wedge \tau _{R}}\left \| \nabla u^{\varepsilon}(s)-\nabla u^{*}(s) \right \| ^{2}_{\mathcal{H} ^{0}}\text{d}s
+C_{a_{1},a_{4},R }\mathbb{E}\int_{0}^{T\wedge \tau _{R}}\left \| u^{\varepsilon}(s)-u^{*}(s) \right \| ^{2 }_{\mathcal{H} ^{0}}\text{d}s.\nonumber
\end{align}
For $\Sigma _{2}^{2}$,
\begin{align}\label{s16}
	\Sigma _{2}^{2}&=-\mathbb{E}\underset{t\in [0,T\wedge \tau _{R}]}{\sup}\int_{0}^{t}\left \langle u^{\varepsilon}(s)-u^{*}(s),\eta _{2}^{\varepsilon } (s)S_{2}(u^{*}(s) )-\eta _{2}^{* } S_{2}(u^{*}(s) ) \right \rangle _{\mathcal{H} ^{0}}\text{d}s,\nonumber
\\&\le -\mathbb{E}\underset{t\in [0,T\wedge \tau _{R}]}{\sup}\int_{0}^{t}\left \langle u^{\varepsilon}(s)-\overline{u }^{\varepsilon}(s),\eta _{2}^{\varepsilon } (s)S_{2}(u^{*}(s) )-\eta _{2}^{* } S_{2}(u^{*}(s) ) \right \rangle _{\mathcal{H} ^{0}}\text{d}s\nonumber
\\&~~~-\mathbb{E}\underset{t\in [0,T\wedge \tau _{R}]}{\sup}\int_{0}^{t}\left \langle \overline{u }^{\varepsilon}(s)-\overline{u}^{*}(s),\eta _{2}^{\varepsilon } (s)S_{2}(u^{*}(s) )-\eta _{2}^{* } S_{2}(u^{*}(s) ) \right \rangle _{\mathcal{H} ^{0}}\text{d}s
\\&~~~-\mathbb{E}\underset{t\in [0,T\wedge \tau _{R}]}{\sup}\int_{0}^{t}\left \langle \overline{u}^{*}(s)-u^{*}(s),\eta _{2}^{\varepsilon } (s)S_{2}(u^{*}(s) )-\eta _{2}^{* } S_{2}(u^{*}(s) ) \right \rangle _{\mathcal{H} ^{0}}\text{d}s\nonumber
\\&:=\Sigma _{2}^{2,1}+\Sigma _{2}^{2,2}+\Sigma _{2}^{2,3}.\nonumber
	\end{align}
Applying H\"older's  inequality, $\Sigma _{2}^{2,1}$ is controlled by
\begin{align*}
	\Sigma _{2}^{2,1}&=-\mathbb{E}\underset{t\in [0,T\wedge \tau _{R}]}{\sup}\int_{0}^{t}\left \langle u^{\varepsilon}(s)-\overline{u }^{\varepsilon}(s),\eta _{2}^{\varepsilon } (s)S_{2}(u^{*}(s) )-\eta _{2}^{* } S_{2}(u^{*}(s) ) \right \rangle _{\mathcal{H} ^{0}}\text{d}s
\\&\le\mathbb{E}\int_{0}^{T\wedge \tau _{R}}\left \| u^{\varepsilon}(s)-\overline{u }^{\varepsilon}(s) \right \|_{\mathcal{H} ^{0}}\left \| \eta _{2}^{\varepsilon } (s)S_{2}(u^{*}(s) )-\eta _{2}^{* } S_{2}(u^{*}(s) ) \right \|_{\mathcal{H} ^{0}}\text{d}s
\\&\le 2a_{4}(\mathbb{E}\int_{0}^{T\wedge \tau _{R}} \left \| S_{2}(u^{*}(s) )\right \|_{\mathcal{H} ^{0}}^{2}\text{d}s)^{\frac{1}{2} }(\mathbb{E}\int_{0}^{T\wedge \tau _{R}} \left \| u^{\varepsilon}(s)-\overline{u }^{\varepsilon}(s) \right \|_{\mathcal{H} ^{0}}^{2}\text{d}s)^{\frac{1}{2} },
	\end{align*}
where
\begin{align}\label{s23}
	 \left \| S_{2}(u^{*}(s) )\right \|_{\mathcal{H} ^{0}}^{2}&\le\left \| (u^{* }(s),\nabla)u^{* }(s)\right \|_{\mathcal{H} ^{0}}^{2}+\left \| \Psi _{N}(\left | u^{* }(s)\right |^{2})u^{* }(s)\right \|_{\mathcal{H} ^{0}}^{2}\nonumber
\\&\le \int _{\mathbb{D} }\sum_{j=1}^{3} (u_{j})^{2}\sum_{i=1}^{3}\sum_{k=1}^{3}(\partial_{k}u_{i} )^{2}\text{d}x+\left \|u^{* }(s)  \right \|^{6}_{L^{6}}\nonumber
\\&\le \sup_{x\in \mathbb{D} } \sum_{j=1}^{3} (u_{j})^{2}\left \|u^{* }(s)  \right \|^{2}_{\mathcal{H} ^{1}}+C\left \|u^{* }(s)  \right \|^{6}_{\mathcal{H} ^{1}}
\\&\le C\left \|u^{* }(s)  \right \|_{\mathcal{H} ^{2}}\left \|u^{* }(s)  \right \|^{3}_{\mathcal{H} ^{1}}+C\left \|u^{* }(s)  \right \|^{6}_{\mathcal{H} ^{1}}.\nonumber
	\end{align}
Consequently,
\begin{align}\label{s17}
	\Sigma _{2}^{2,1}\le C_{T}(\left \| u_{0}^{\varepsilon} \right \|_{\mathcal{H}^{1} }^{2}+\left \| u_{0} ^{* } \right \|^{2}_{\mathcal{H} ^{1}}+\left \| u_{0} ^{* } \right \|^{6}_{\mathcal{H} ^{1}}+1 )d^{\frac{1}{4} }.
	\end{align}
Similarly, for $\Sigma _{2}^{2,3}$,
\begin{align}\label{s18}
	\Sigma _{1}^{2,3}&=-\mathbb{E}\underset{t\in [0,T\wedge \tau _{R}]}{\sup}\int_{0}^{t}\left \langle \overline{u}^{*}(s)-u^{*}(s),\eta _{2}^{\varepsilon } (s)S_{2}(u^{*}(s) )-\eta _{2}^{* } S_{2}(u^{*}(s) ) \right \rangle _{\mathcal{H} ^{0}}\text{d}s
\\&\le C_{T}(\left \| u_{0} ^{* } \right \|^{2}_{\mathcal{H} ^{0}}+\left \| u_{0} ^{* } \right \|^{2}_{\mathcal{H} ^{1}}+\left \| u_{0} ^{* } \right \|^{6}_{\mathcal{H} ^{1}}+1 )d^{\frac{1}{4} }. \nonumber
	\end{align}
Next, for $\Sigma _{2}^{2,2}$, we obtain
\begin{align}\label{s19}
	\Sigma _{2}^{2,2}&=-\mathbb{E}\underset{t\in [0,T\wedge \tau _{R}]}{\sup}\int_{0}^{t}\left \langle \overline{u }^{\varepsilon}(s)-\overline{u}^{*}(s),\eta _{2}^{\varepsilon } (s)S_{2}(u^{*}(s) )-\eta _{2}^{* } S_{2}(u^{*}(s) ) \right \rangle _{\mathcal{H} ^{0}}\text{d}s \nonumber
\\&\le-\mathbb{E}\underset{t\in [0,T\wedge \tau _{R}]}{\sup}\int_{0}^{t}\left \langle \overline{u }^{\varepsilon}(s)-\overline{u}^{*}(s),\eta _{2}^{\varepsilon } (s)S_{2}(u^{*}(s) )-\eta _{2}^{\varepsilon }(s) S_{2}(\overline{u }^{*}(s) ) \right \rangle _{\mathcal{H} ^{0}}\text{d}s \nonumber
\\&~~~-\mathbb{E}\underset{t\in [0,T\wedge \tau _{R}]}{\sup}\int_{0}^{t}\left \langle \overline{u }^{\varepsilon}(s)-\overline{u}^{*}(s),\eta _{2}^{\varepsilon }(s) S_{2}(\overline{u }^{*}(s) )-\eta _{2}^{* } S_{2}(\overline{u }^{*}(s) ) \right \rangle _{\mathcal{H} ^{0}}\text{d}s
\\&~~~-\mathbb{E}\underset{t\in [0,T\wedge \tau _{R}]}{\sup}\int_{0}^{t}\left \langle \overline{u }^{\varepsilon}(s)-\overline{u}^{*}(s),\eta _{2}^{* } S_{2}(\overline{u }^{*}(s) )-\eta _{2}^{* } S_{2}(u^{*}(s) ) \right \rangle _{\mathcal{H} ^{0}}\text{d}s \nonumber
\\&:=\Sigma _{2}^{2,2,1}+\Sigma _{2}^{2,2,2}+\Sigma _{2}^{2,2,3}. \nonumber
	\end{align}
Using H\"older's  inequality and \eqref{k1}, $\Sigma _{2}^{2,2,1}$ is estimated by
\begin{align}\label{s20}
	&\Sigma _{2}^{2,2,1}\nonumber
\\&=-\mathbb{E}\underset{t\in [0,T\wedge \tau _{R}]}{\sup}\int_{0}^{t}\left \langle \overline{u }^{\varepsilon}(s)-\overline{u}^{*}(s),\eta _{2}^{\varepsilon } (s)S_{2}(u^{*}(s) )-\eta _{2}^{\varepsilon }(s) S_{2}(\overline{u }^{*}(s) ) \right \rangle _{\mathcal{H} ^{0}}\text{d}s\nonumber
\\&=-\mathbb{E}\underset{t\in [0,T\wedge \tau _{R}]}{\sup}\int_{0}^{t}\left \langle \overline{u }^{\varepsilon}(s)-\overline{u}^{*}(s),\eta _{2}^{\varepsilon } (s)(u^{* }(s),\nabla)u^{* }(s)-\eta _{2}^{\varepsilon}(s) (\overline{u}^{*}(s),\nabla)\overline{u}^{*}(s) \right \rangle _{\mathcal{H} ^{0}}\text{d}s\nonumber
\\&~~~-\mathbb{E}\underset{t\in [0,T\wedge \tau _{R}]}{\sup}\int_{0}^{t}\left \langle \overline{u }^{\varepsilon}(s)-\overline{u}^{*}(s),\eta _{2}^{\varepsilon } (s)\Psi _{N}(\left | u^{* }(s) \right |^{2} )u^{* }(s)-\eta _{2}^{\varepsilon } (s)\Psi _{N}(\left | \overline{u}^{*}(s) \right |^{2} )\overline{u}^{*}(s) \right \rangle _{\mathcal{H} ^{0}}\text{d}s\nonumber
\\&\le \mathbb{E}\underset{t\in [0,T\wedge \tau _{R}]}{\sup}\int_{0}^{t}\left \langle \nabla\overline{u }^{\varepsilon}(s)-\nabla\overline{u}^{*}(s),\eta _{2}^{\varepsilon } (s)(u^{* }(s))^{T}u^{* }(s)-\eta _{2}^{\varepsilon}(s) (\overline{u}^{*}(s))^{T}\overline{u}^{*}(s) \right \rangle _{\mathcal{H} ^{0}}\text{d}s\nonumber
\\&~~~+a_{4}\mathbb{E}\int_{0}^{T\wedge \tau _{R}}\left \| \overline{u }^{\varepsilon}(s)-\overline{u}^{*}(s)\right \| _{\mathcal{H} ^{0}}\left \|\Psi _{N}(\left | u^{* }(s) \right |^{2} )u^{* }(s)-\Psi _{N}(\left | \overline{u}^{*}(s) \right |^{2} )\overline{u}^{*}(s) \right \| _{\mathcal{H} ^{0}}\text{d}s
\\&\le C (\mathbb{E}\int_{0}^{T\wedge \tau _{R}}\left \| \overline{u }^{\varepsilon}(s)\right \| ^{2}_{\mathcal{H} ^{1}}+\left \| \overline{u}^{*}(s)\right \| ^{2}_{\mathcal{H} ^{1}}\text{d}s)^{\frac{1}{2} }(\mathbb{E}\int_{0}^{T\wedge \tau _{R}}\left \| (u^{*}(s))^{T} u^{*}(s)-(\overline{u}^{*}(s))^{T} \overline{u}^{*}(s) \right \| ^{2}_{\mathcal{H} ^{0}}\text{d}s)^{\frac{1}{2}}\nonumber
\\&~~~+C(\mathbb{E}\int_{0}^{T\wedge \tau _{R}}\left \| \overline{u }^{\varepsilon}(s)\right \| ^{2}_{\mathcal{H} ^{0}}+\left \| \overline{u}^{*}(s)\right \| ^{2}_{\mathcal{H} ^{0}}\text{d}s)^{\frac{1}{2} }(\mathbb{E}\int_{0}^{T\wedge \tau _{R}}\left \| u^{*}(s)- \overline{u}^{*}(s) \right \| ^{2}_{L ^{6}}[\left \| u^{*}(s) \right \| ^{4}_{L ^{6}}+\left \| \overline{u}^{*}(s) \right \| ^{4}_{L ^{6}}]\text{d}s)^{\frac{1}{2}}\nonumber
\\&\le C(\left \| u_{0} ^{\varepsilon } \right \|^{2}_{\mathcal{H} ^{0}}+\left \| u_{0} ^{* } \right \|^{2}_{\mathcal{H} ^{0}}+1 )[(\mathbb{E}\int_{0}^{T\wedge \tau _{R}}\left \|  u^{*}(s)- \overline{u}^{*}(s) \right \| ^{\frac{3}{2} }_{\mathcal{H} ^{1}}\left \| u^{*}(s)- \overline{u}^{*}(s) \right \| ^{\frac{1}{2} }_{\mathcal{H} ^{0}}\text{d}s)^{\frac{1}{2}}\nonumber
\\&~~~+(\mathbb{E}\int_{0}^{T\wedge \tau _{R}}\left \|  u^{*}(s)- \overline{u}^{*}(s) \right \| _{\mathcal{H} ^{2}}\left \| u^{*}(s)- \overline{u}^{*}(s) \right \| _{\mathcal{H} ^{0}}[\left \| u^{*}(s) \right \| ^{4}_{\mathcal{H} ^{1}}+\left \| \overline{u}^{*}(s) \right \| ^{4}_{\mathcal{H} ^{1}}]\text{d}s)^{\frac{1}{2}}],\nonumber
	\end{align}
where
\begin{align*}
&\mathbb{E}\int_{0}^{T\wedge \tau _{R}}\left \|  u^{*}(s)- \overline{u}^{*}(s) \right \| ^{\frac{3}{2} }_{\mathcal{H} ^{1}}\left \| u^{*}(s)- \overline{u}^{*}(s) \right \| ^{\frac{1}{2} }_{\mathcal{H} ^{0}}\text{d}s
\\&\le (\mathbb{E}\int_{0}^{T\wedge \tau _{R}}\left \|  u^{*}(s)- \overline{u}^{*}(s) \right \| ^{2 }_{\mathcal{H} ^{1}}\text{d}s)^{\frac{3}{4} }(\mathbb{E}\int_{0}^{T\wedge \tau _{R}}\left \| u^{*}(s)- \overline{u}^{*}(s) \right \| ^{2}_{\mathcal{H} ^{0}}\text{d}s)^{\frac{1}{4} }
\\&\le C_{T}(\left \| u_{0} ^{* } \right \|^{2}_{\mathcal{H} ^{0}}+1 )(\left \| u^{* }_{0}\right \|_{\mathcal{H} ^{0}}^{2}+\left \| u^{* }_{0}\right \|_{\mathcal{H} ^{1}}^{2}+1)d^{\frac{1}{8} },
	\end{align*}
and
\begin{align*}
&\mathbb{E}\int_{0}^{T\wedge \tau _{R}}\left \|  u^{*}(s)- \overline{u}^{*}(s) \right \| _{\mathcal{H} ^{2}}\left \| u^{*}(s)- \overline{u}^{*}(s) \right \| _{\mathcal{H} ^{0}}\text{d}s
\\&\le (\mathbb{E}\int_{0}^{T\wedge \tau _{R}}\left \|  u^{*}(s)- \overline{u}^{*}(s) \right \| ^{2 }_{\mathcal{H} ^{2}}\text{d}s)^{\frac{1}{2} }(\mathbb{E}\int_{0}^{T\wedge \tau _{R}}\left \| u^{*}(s)- \overline{u}^{*}(s) \right \| ^{2}_{\mathcal{H} ^{0}}\text{d}s)^{\frac{1}{2} }
\\&\le C_{T}(\left \| u_{0} ^{* } \right \|^{2}_{\mathcal{H} ^{0}}+1 )(\left \| u^{* }_{0}\right \|_{\mathcal{H} ^{0}}^{2}+\left \| u^{* }_{0}\right \|_{\mathcal{H} ^{1}}^{2}+1)d^{\frac{1}{4} }.
	\end{align*}
Consequently,
\begin{align}\label{s21}
	\Sigma _{2}^{2,2,1}\le C(\left \| u_{0} ^{\varepsilon } \right \|^{2}_{\mathcal{H} ^{0}}+\left \| u_{0} ^{* } \right \|^{2}_{\mathcal{H} ^{1}}+1 )(d^{\frac{1}{16} }+d^{\frac{1}{4} }).
	\end{align}
Similarly, for $\Sigma _{2}^{2,2,3}$,
\begin{align}\label{s35}
	\Sigma _{2}^{2,2,3}\le C(\left \| u_{0} ^{\varepsilon } \right \|^{2}_{\mathcal{H} ^{0}}+\left \| u_{0} ^{* } \right \|^{2}_{\mathcal{H} ^{1}}+1 )(d^{\frac{1}{16} }+d^{\frac{1}{4} }) .
	\end{align}
By employing the time discretization technique to address $\Sigma _{2}^{2,2,2}$, and applying \eqref{s23}, we have
\begin{align*}
	&\Sigma _{2}^{2,2,2}\nonumber
\\&=-\mathbb{E}\underset{t\in [0,T\wedge \tau _{R}]}{\sup}\int_{0}^{t}\left \langle \overline{u }^{\varepsilon}(s)-\overline{u}^{*}(s),\eta _{2}^{\varepsilon }(s) S_{2}(\overline{u }^{*}(s) )-\eta _{2}^{* } S_{2}(\overline{u }^{*}(s) ) \right \rangle _{\mathcal{H} ^{0}}\text{d}s  \nonumber
\\&=-\mathbb{E}\underset{t\in [0,T\wedge \tau _{R}]}{\sup}\sum_{n=0}^{\left [ t/d  \right ]-1 } \int_{nd}^{(n+1)d}\left \langle u ^{\varepsilon}(nd)-u^{*}(nd),\eta _{2}^{\varepsilon }(s) S_{2}(u ^{*}(nd) )-\eta _{2}^{* } S_{2}(u ^{*}(nd) ) \right \rangle _{\mathcal{H} ^{0}}\text{d}s  \nonumber
\\&~~~-\mathbb{E}\int_{\left [ t/d  \right ]d}^{t}\left \langle u ^{\varepsilon}(\left [ t/d  \right ]d)-u^{*}(\left [ t/d  \right ]d),\eta _{2}^{\varepsilon }(s) S_{2}(u ^{*}(\left [ t/d  \right ]d) )-\eta _{2}^{* } S_{2}(u ^{*}(\left [ t/d  \right ]d) ) \right \rangle _{\mathcal{H} ^{0}}\text{d}s  \nonumber
\\&=\mathbb{E}\underset{t\in [0,T\wedge \tau _{R}]}{\sup}\sum_{n=0}^{\left [ t/d  \right ]-1 } \int_{nd}^{(n+1)d}\left \langle \nabla u ^{\varepsilon}(nd)-\nabla u^{*}(nd),\eta _{2}^{\varepsilon } (s)S^{1}_{2}(u ^{*}(nd) )-\eta _{2}^{*} S^{1}_{2}(u ^{*}(nd) ) \right \rangle _{\mathcal{H} ^{0}}\text{d}s  \nonumber
\\&~~~-\mathbb{E}\underset{t\in [0,T\wedge \tau _{R}]}{\sup}\sum_{n=0}^{\left [ t/d  \right ]-1 } \int_{nd}^{(n+1)d}\left \langle u ^{\varepsilon}(nd)-u^{*}(nd),\eta _{2}^{\varepsilon } (s)S^{2}_{2}(u ^{*}(nd)-\eta _{2}^{*}S^{2}_{2}(u ^{*}(nd) ) \right \rangle _{\mathcal{H} ^{0}}\text{d}s  \nonumber
\\&~~~+\mathbb{E}\int_{\left [ t/d  \right ]d}^{t}\left \langle \nabla u ^{\varepsilon}(\left [ t/d  \right ]d)-\nabla u^{*}(\left [ t/d  \right ]d),\eta _{2}^{\varepsilon } (s)S^{1}_{2}(u ^{*}(\left [ t/d  \right ]d) )-\eta _{2}^{*} S^{1}_{2}(u ^{*}(\left [ t/d  \right ]d) ) \right \rangle _{\mathcal{H} ^{0}}\text{d}s\nonumber
\\&~~~-\mathbb{E}\int_{\left [ t/d  \right ]d}^{t}\left \langle u ^{\varepsilon}(\left [ t/d  \right ]d)-u^{*}(\left [ t/d  \right ]d),\eta _{2}^{\varepsilon } (s)S^{2}_{2}(u ^{*}(\left [ t/d  \right ]d)-\eta _{2}^{*}S^{2}_{2}(u ^{*}(\left [ t/d  \right ]d) )\right \rangle _{\mathcal{H} ^{0}}\text{d}s\nonumber
\\&=\mathbb{E}\underset{t\in [0,T\wedge \tau _{R}]}{\sup}\sum_{n=0}^{\left [ t/d  \right ]-1 }  \left \| \int_{nd}^{(n+1)d} \eta _{2}^{\varepsilon } (s)S^{1}_{2}(u ^{*}(nd) )-\eta _{2}^{*} S^{1}_{2}(u ^{*}(nd) ) \text{d}s \right \|_{\mathcal{H}^{0} }\left \|  u ^{\varepsilon}(nd)- u^{*}(nd) \right \|_{\mathcal{H}^{1} } \nonumber
\\&~~~+\mathbb{E}\underset{t\in [0,T\wedge \tau _{R}]}{\sup}\sum_{n=0}^{\left [ t/d  \right ]-1 } \left \|\int_{nd}^{(n+1)d}   \eta _{2}^{\varepsilon } (s)S^{2}_{2}(u ^{*}(nd) )-\eta _{2}^{*} S^{2}_{2}(u ^{*}(nd) ) \text{d}s  \right \|_{\mathcal{H}^{0} }\left \|  u ^{\varepsilon}(nd)- u^{*}(nd) \right \|_{\mathcal{H}^{0} }
\\&~~~+\mathbb{E}\int_{\left [ t/d  \right ]d}^{t}\left \|  u ^{\varepsilon}(\left [ t/d  \right ]d)- u^{*}(\left [ t/d  \right ]d) \right \|_{\mathcal{H}^{1} } \left \|  \eta _{2}^{\varepsilon } (s)S^{1}_{2}(u ^{*}(\left [ t/d  \right ]d) )-\eta _{2}^{*} S^{1}_{2}(u ^{*}(\left [ t/d  \right ]d) ) \right \|_{\mathcal{H}^{0} }\text{d}s\nonumber
\\&~~~+\mathbb{E}\int_{\left [ t/d  \right ]d}^{t}\left \|  u ^{\varepsilon}(\left [ t/d  \right ]d)- u^{*}(\left [ t/d  \right ]d) \right \|_{\mathcal{H}^{0} } \left \|  \eta _{2}^{\varepsilon } (s)S^{2}_{2}(u ^{*}(\left [ t/d  \right ]d) )-\eta _{2}^{*} S^{2}_{2}(u ^{*}(\left [ t/d  \right ]d) ) \right \|_{\mathcal{H}^{0} }\text{d}s\nonumber
\\&\le \sum_{n=0}^{\left [ T/d  \right ]-1 } (\left \|\mathbb{E}\int_{nd}^{(n+1)d} \eta _{2}^{\varepsilon } (s)S^{1}_{2}(u ^{*}(nd) )-\eta _{2}^{*} S^{1}_{2}(u ^{*}(nd) )  \text{d}s\right \|^{2}_{\mathcal{H} ^{0}})^{\frac{1}{2} } (\mathbb{E} \left \|u ^{\varepsilon}(nd)-u ^{*}(nd) \right \| ^{2}_{\mathcal{H} ^{1}} )^{\frac{1}{2} } \nonumber
\\&~~~+\sum_{n=0}^{\left [ T/d  \right ]-1 } (\left \|\mathbb{E}\int_{nd}^{(n+1)d} \eta _{2}^{\varepsilon } (s)S^{2}_{2}(u ^{*}(nd) )-\eta _{2}^{*} S^{2}_{2}(u ^{*}(nd) )  \text{d}s\right \|^{2}_{\mathcal{H} ^{0}})^{\frac{1}{2} } (\mathbb{E} \left \|u ^{\varepsilon}(nd)-u ^{*}(nd) \right \| ^{2}_{\mathcal{H} ^{0}} )^{\frac{1}{2} } \nonumber
\\&~~~+C(\left \| u_{0}^{\varepsilon} \right \|_{\mathcal{H}^{0} }^{2}+\left \| u_{0}^{\varepsilon} \right \|_{\mathcal{H}^{1} }^{2}+\left \| u_{0} ^{* } \right \|^{2}_{\mathcal{H} ^{1}}+\left \| u_{0} ^{* } \right \|^{2}_{\mathcal{H} ^{0}}+\left \| u_{0} ^{* } \right \|^{6}_{\mathcal{H} ^{1}}+1 )d\nonumber
\\&\le \frac{T}{d} \max_{0\le n\le \left [ T/d \right ]-1, n\in \mathbb{N}^{+} } (\left \|\mathbb{E}\int_{nd}^{(n+1)d} \eta _{2}^{\varepsilon } (s)S^{1}_{2}(u ^{*}(nd) )-\eta _{2}^{*} S^{1}_{2}(u ^{*}(nd) )  \text{d}s\right \|^{2}_{\mathcal{H} ^{0}})^{\frac{1}{2} } C(\left \| u_{0}^{\varepsilon} \right \|_{\mathcal{H}^{1} }^{2}+\left \| u_{0} ^{* } \right \|^{2}_{\mathcal{H} ^{1}}+1 ) \nonumber
\\&~~~+\frac{T}{d} \max_{0\le n\le \left [ T/d \right ]-1, n\in \mathbb{N}^{+} } (\left \|\mathbb{E}\int_{nd}^{(n+1)d} \eta _{2}^{\varepsilon } (s)S^{2}_{2}(u ^{*}(nd) )-\eta _{2}^{*} S^{2}_{2}(u ^{*}(nd) )  \text{d}s\right \|^{2}_{\mathcal{H} ^{0}})^{\frac{1}{2} } C(\left \| u_{0}^{\varepsilon} \right \|_{\mathcal{H}^{0} }^{2}+\left \| u_{0} ^{* } \right \|^{2}_{\mathcal{H} ^{0}}+1 ) \nonumber
\\&~~~+C(\left \| u_{0}^{\varepsilon} \right \|_{\mathcal{H}^{1} }^{2}+\left \| u_{0} ^{* } \right \|^{2}_{\mathcal{H} ^{1}}+\left \| u_{0} ^{* } \right \|^{6}_{\mathcal{H} ^{1}}+1 )d,\nonumber
	\end{align*}
where $S_{2}^{1}(u ^{*}(t) ):=(u ^{*}(t))^{T}u ^{*}(t), S^{2}_{2}(u ^{*}(t)):= \Psi _{N}(\left | u ^{*}(t) \right |^{2} )u ^{*}(t)$. Based on assumption \textbf{(H4)}, we obtain
\begin{align}\label{s32}
	&(\left \|\mathbb{E}\int_{nd}^{(n+1)d} \eta _{2}^{\varepsilon } (s)S^{1}_{2}(u ^{*}(nd) )-\eta _{2}^{*} S^{1}_{2}(u ^{*}(nd) )  \text{d}s\right \|^{2}_{\mathcal{H} ^{0}})^{\frac{1}{2} }\nonumber
\\&=(\left \|\mathbb{E}\int_{nd}^{(n+1)d} \eta _{2} (\frac{s}{\varepsilon } )(u ^{*}(nd))^{T}u ^{*}(nd)-\eta _{2}^{*} (u ^{*}(nd))^{T}u ^{*}(nd)  \text{d}s\right \|^{2}_{\mathcal{H} ^{0}})^{\frac{1}{2} }\nonumber
\\&\le (\mathbb{E}\left \|u ^{*}(nd)\right \|^{4}_{L ^{4}}\left |\int_{nd}^{(n+1)d} \eta _{2} (\frac{s}{\varepsilon } )-\eta _{2}^{*}   \text{d}s\right |^{2})^{\frac{1}{2} }
\\&\le (\mathbb{E}\left \|u ^{*}(nd)\right \|^{4}_{\mathcal{H} ^{1}})^{\frac{1}{2} }\left |\int_{nd}^{(n+1)d} \eta _{2} (\frac{s}{\varepsilon } )-\eta _{2}^{*}   \text{d}s\right |\nonumber
\\&\le \varepsilon C(\left \|u ^{*}_{0}\right \|^{4}_{\mathcal{H} ^{1}}+1)\left |\int_{\frac{nd}{\varepsilon } }^{\frac{(n+1)d}{\varepsilon }} \eta _{2} (r )-\eta _{2}^{*}  \text{d}r\right |\nonumber
\\&\le C_{T}\mathcal{R} _{2}(\frac{d}{\varepsilon } )(\left \| u ^{*}_{0}  \right \|^{4}_{\mathcal{H} ^{1}}+1)d,\nonumber
	\end{align}
and
\begin{align}\label{s31}
	&(\left \|\mathbb{E}\int_{nd}^{(n+1)d} \eta _{2}^{\varepsilon } (s)S^{2}_{2}(u ^{*}(nd) )-\eta _{2}^{*} S^{2}_{2}(u ^{*}(nd) )  \text{d}s\right \|^{2}_{\mathcal{H} ^{0}})^{\frac{1}{2} }\nonumber
\\&=(\left \|\mathbb{E}\int_{nd}^{(n+1)d} \eta _{2} (\frac{s}{\varepsilon } )\Psi _{N}(\left | u ^{*}(nd) \right |^{2} )u ^{*}(nd)-\eta _{2}^{*} \Psi _{N}(\left | u ^{*}(nd) \right |^{2} )u ^{*}(nd)  \text{d}s\right \|^{2}_{\mathcal{H} ^{0}})^{\frac{1}{2} }\nonumber
\\&\le (\mathbb{E}\left \|u ^{*}(nd)\right \|^{6}_{L ^{6}}\left |\int_{nd}^{(n+1)d} \eta _{2} (\frac{s}{\varepsilon } )-\eta _{2}^{*}   \text{d}s\right |^{2})^{\frac{1}{2} }
\\&\le C_{T}\mathcal{R} _{2}(\frac{d}{\varepsilon } )(\left \| u ^{*}_{0}  \right \|^{6}_{\mathcal{H} ^{1}}+1)d.\nonumber
	\end{align}
Substituting \eqref{s15}-\eqref{s31} into \eqref{s13} implies
\begin{align}\label{a33}
	\Sigma _{2}&\le \frac{a_{1} }{2}    \mathbb{E}\int_{0}^{T\wedge \tau _{R}}\left \| \nabla u^{\varepsilon}(s)-\nabla u^{*}(s) \right \| ^{2}_{\mathcal{H} ^{0}}\text{d}s
+C_{a_{1},a_{4},R}\mathbb{E}\int_{0}^{T\wedge \tau _{R}}\left \| u^{\varepsilon}(s)-u^{*}(s) \right \| ^{2 }_{\mathcal{H} ^{0}}\text{d}s
\\&~~~+C(\left \| u_{0}^{\varepsilon} \right \|_{\mathcal{H}^{1} }^{2}+\left \| u_{0} ^{* } \right \|^{2}_{\mathcal{H} ^{1}}+\left \| u_{0} ^{* } \right \|^{6}_{\mathcal{H} ^{1}}+\left \| u_{0} ^{* } \right \|^{4}_{\mathcal{H} ^{1}}+1 )(\mathcal{R} _{2}(\frac{d}{\varepsilon } )+d^{\frac{1}{4} }+d^{\frac{1}{16} }).\nonumber
	\end{align}
Given that $f^{\varepsilon}$  satisfies locally weak
monotonicity conditions, by employing a similar argument as in $\Sigma _{1}$,  along with H\"older's inequality, Jensen's inequality and Lemma \ref{le1} we obtain
\begin{align}\label{a8}
	\Sigma _{3}&=\mathbb{E}\underset{t\in [0,T\wedge \tau _{R}]}{\sup}\int_{0}^{t}\langle u^{\varepsilon}(s)-u^{*}(s),f^{\varepsilon}(s,u^{\varepsilon}(s))-f^{*}(u^{*}(s)) \rangle_{\mathcal{H} ^{0}}\text{d}s\nonumber
\\&\le \mathbb{E}\underset{t\in [0,T\wedge \tau _{R}]}{\sup}\int_{0}^{t}\langle u^{\varepsilon}(s)-u^{*}(s),f^{\varepsilon}(s,u^{\varepsilon}(s))-f^{\varepsilon}(s,u^{*}(s))\rangle_{\mathcal{H} ^{0}}\text{d}s\nonumber
\\&~~~+\mathbb{E}\underset{t\in [0,T\wedge \tau _{R}]}{\sup}\int_{0}^{t}\langle u^{\varepsilon}(s)-u^{*}(s),f^{\varepsilon}(s,u^{*}(s))-f^{*}(u^{*}(s))\rangle_{\mathcal{H} ^{0}}\text{d}s\nonumber
\\&\le c\int_{0}^{T\wedge \tau _{R}}\mathcal{A} (\mathbb{E}\left \| u^{\varepsilon }(s) -u^{* }(s) \right \|^{2 }_{\mathcal{H} ^{0}})\text{d}s+\mathbb{E}\underset{t\in [0,T\wedge \tau _{R}]}{\sup}\int_{0}^{t}\langle u^{\varepsilon}(s)-\overline{u} ^{\varepsilon}(s),f^{\varepsilon}(s,u^{*}(s))-f^{*}(u^{*}(s))\rangle_{\mathcal{H} ^{0}}\text{d}s\nonumber
\\&~~~+\mathbb{E}\underset{t\in [0,T\wedge \tau _{R}]}{\sup}\int_{0}^{t}\langle \overline{u} ^{\varepsilon}(s)-\overline{u} ^{*}(s),f^{\varepsilon}(s,u^{*}(s))-f^{*}(u^{*}(s))\rangle_{\mathcal{H} ^{0}}\text{d}s
\\&~~~+\mathbb{E}\underset{t\in [0,T\wedge \tau _{R}]}{\sup}\int_{0}^{t}\langle \overline{u} ^{*}(s)-u^{*}(s),f^{\varepsilon}(s,u^{*}(s))-f^{*}(u^{*}(s))\rangle_{\mathcal{H} ^{0}}\text{d}s\nonumber
\\&\le c\int_{0}^{T\wedge \tau _{R}}\mathcal{A} (\mathbb{E}\left \| u^{\varepsilon }(s) -u^{* }(s) \right \|^{2 }_{\mathcal{H} ^{0}})\text{d}s+\mathbb{E}\underset{t\in [0,T\wedge \tau _{R}]}{\sup}\int_{0}^{t}\langle \overline{u} ^{\varepsilon}(s)-\overline{u} ^{*}(s),f^{\varepsilon}(s,u^{*}(s))-f^{*}(u^{*}(s))\rangle_{\mathcal{H} ^{0}}\text{d}s\nonumber
\\&~~~+(\mathbb{E}\int_{0}^{T}(\left \| f^{\varepsilon}(s,u^{*}(s))\right \|^{2}_{\mathcal{H} ^{0}} +\left \| f^{*}(u^{*}(s))\right \|^{2}_{\mathcal{H} ^{0}}) \text{d}s)^{\frac{1}{2} }(\mathbb{E}\int_{0}^{T} \left \| u^{\varepsilon }(s)- \overline{u}  ^{\varepsilon}(s) \right \|^{2}_{\mathcal{H} ^{0}} \text{d}s)^{\frac{1}{2} }\nonumber
\\&~~~+(\mathbb{E}\int_{0}^{T}(\left \| f^{\varepsilon}(s,u^{*}(s))\right \|^{2}_{\mathcal{H} ^{0}} +\left \| f^{*}(u^{*}(s))\right \|^{2}_{\mathcal{H} ^{0}}) \text{d}s)^{\frac{1}{2} }(\mathbb{E}\int_{0}^{T} \left \| u^{*}(s)- \overline{u}  ^{*}(s) \right \|^{2}_{\mathcal{H} ^{0}} \text{d}s)^{\frac{1}{2} }\nonumber
\\&\le c\int_{0}^{T\wedge \tau _{R}}\mathcal{A} (\mathbb{E}\left \| u^{\varepsilon }(s) -u^{* }(s) \right \|^{2 }_{\mathcal{H} ^{0}})\text{d}s+\mathbb{E}\underset{t\in [0,T\wedge \tau _{R}]}{\sup}\int_{0}^{t}\langle \overline{u} ^{\varepsilon}(s)-\overline{u} ^{*}(s),f^{\varepsilon}(s,u^{*}(s))-f^{*}(u^{*}(s))\rangle_{\mathcal{H} ^{0}}\text{d}s\nonumber
\\&~~~+C_{T}(\left \| u_{0} ^{\varepsilon } \right \|^{2}_{\mathcal{H} ^{1}}+\left \| u_{0} ^{\varepsilon } \right \|^{6}_{\mathcal{H} ^{1}}+\left \| u_{0} ^{* } \right \|^{2}_{\mathcal{H} ^{1}}+\left \| u_{0} ^{* } \right \|^{6}_{\mathcal{H} ^{1}}+1 )d^{\frac{1}{4} }.\nonumber
	\end{align}
The next critical task is to estimate $\mathbb{E}\underset{t\in [0,T\wedge \tau _{R}]}{\sup}\int_{0}^{t}\langle \overline{u} ^{\varepsilon}(s)-\overline{u} ^{*}(s),f^{\varepsilon}(s,u^{*}(s))-f^{*}(u^{*}(s))\rangle_{\mathcal{H} ^{0}}\text{d}s$. By Lemma \ref{le1} and  Jensen's inequality, we have
\begin{align}\label{a13}
	&\mathbb{E}\underset{t\in [0,T\wedge \tau _{R}]}{\sup}\int_{0}^{t}\langle \overline{u} ^{\varepsilon}(s)-\overline{u} ^{*}(s),f^{\varepsilon}(s,u^{*}(s))-f^{*}(u^{*}(s))\rangle_{\mathcal{H} ^{0}}\text{d}s \nonumber
\\&\le\mathbb{E}\underset{t\in [0,T\wedge \tau _{R}]}{\sup}\int_{0}^{t}\langle \overline{u} ^{\varepsilon}(s)-\overline{u} ^{*}(s),f^{\varepsilon}(s,u^{*}(s))-f^{\varepsilon}(s,\overline{u} ^{*}(s))\rangle_{\mathcal{H} ^{0}}\text{d}s \nonumber
\\&~~~+\mathbb{E}\underset{t\in [0,T\wedge \tau _{R}]}{\sup}\int_{0}^{t}\langle \overline{u} ^{\varepsilon}(s)-\overline{u} ^{*}(s),f^{\varepsilon}(s,\overline{u} ^{*}(s))-f^{*}(\overline{u} ^{*}(s))\rangle_{\mathcal{H} ^{0}}\text{d}s\nonumber
\\&~~~+\mathbb{E}\underset{t\in [0,T\wedge \tau _{R}]}{\sup}\int_{0}^{t}\langle \overline{u} ^{\varepsilon}(s)-\overline{u} ^{*}(s),f^{*}(\overline{u} ^{*}(s))-f^{*}(u^{*}(s))\rangle_{\mathcal{H} ^{0}}\text{d}s \nonumber
\\&\le C\mathbb{E}\int_{0}^{T\wedge \tau _{R}}\left\| \overline{u} ^{\varepsilon}(s)-\overline{u} ^{*}(s)\right\|_{\mathcal{H} ^{0}}\sqrt{\mathcal{A} (\left\|u^{*}(s)-\overline{u} ^{*}(s)\right\|^{2 }_{\mathcal{H} ^{0}})}\text{d}s
\\&~~~+\mathbb{E}\underset{t\in [0,T\wedge \tau _{R}]}{\sup}\int_{0}^{t}\langle \overline{u} ^{\varepsilon}(s)-\overline{u} ^{*}(s),f^{\varepsilon}(s,\overline{u} ^{*}(s))-f^{*}(\overline{u} ^{*}(s))\rangle_{\mathcal{H} ^{0}}\text{d}s \nonumber
\\&\le [\mathcal{A} (\mathbb{E}\int_{0}^{T\wedge \tau _{R}}\left\|u^{*}(s)-\overline{u} ^{*}(s)\right\|^{2}_{\mathcal{H} ^{0}}\text{d}s)]^{\frac{1}{2} }(\mathbb{E}\int_{0}^{T\wedge \tau _{R}}\left\| \overline{u} ^{\varepsilon}(s)-\overline{u} ^{*}(s)\right\|^{2  }_{\mathcal{H} ^{0}}\text{d}s)^{\frac{1}{2 } }\nonumber
\\&~~~+\mathbb{E}\underset{t\in [0,T\wedge \tau _{R}]}{\sup}\int_{0}^{t}\langle \overline{u} ^{\varepsilon}(s)-\overline{u} ^{*}(s),f^{\varepsilon}(s,\overline{u} ^{*}(s))-f^{*}(\overline{u} ^{*}(s))\rangle_{\mathcal{H} ^{0}}\text{d}s \nonumber
\\&\le C\sqrt{\mathcal{A} ((\left \| u_{0} ^{* } \right \|^{2}_{\mathcal{H} ^{1}}+\left \| u_{0} ^{* } \right \|^{6}_{\mathcal{H} ^{1}}+\left \| u_{0} ^{\varepsilon  } \right \|^{2}_{\mathcal{H} ^{0}}+1 )d^{\frac{1 }{2} })}\nonumber
\\&~~~+\mathbb{E}\underset{t\in [0,T\wedge \tau _{R}]}{\sup}\int_{0}^{t}\langle \overline{u} ^{\varepsilon}(s)-\overline{u} ^{*}(s),f^{\varepsilon}(s,\overline{u} ^{*}(s))-f^{*}(\overline{u} ^{*}(s))\rangle_{\mathcal{H} ^{0}}\text{d}s.\nonumber
	\end{align}
In the subsequent step, we will employ the time discretization technique:
\begin{align}\label{a17}
	&\mathbb{E}\underset{t\in [0,T\wedge \tau _{R}]}{\sup}\int_{0}^{t}\langle \overline{u} ^{\varepsilon}(s)-\overline{u} ^{*}(s),f^{\varepsilon}(s,\overline{u} ^{*}(s))-f^{*}(\overline{u} ^{*}(s))\rangle_{\mathcal{H} ^{0}}\text{d}s \nonumber
\\&=\mathbb{E}\underset{t\in [0,T\wedge \tau _{R}]}{\sup}[\sum_{n=0}^{\left [ t/d  \right ]-1 } \int_{nd}^{(n+1)d}\langle u ^{\varepsilon}(nd)-u ^{*}(nd),f^{\varepsilon}(s,u ^{*}(nd))-f^{*}(u ^{*}(nd))\rangle_{\mathcal{H} ^{0}}\text{d}s  \nonumber
\\&~~~+\int_{\left [ t/d  \right ]d}^{t}\langle u ^{\varepsilon}(\left [ t/d  \right ]d)-u ^{*}(\left [ t/d  \right ]d),f^{\varepsilon}(s,u ^{*}(\left [ t/d  \right ]d))-f^{*}(u ^{*}(\left [ t/d  \right ]d))\rangle_{\mathcal{H} ^{0}}\text{d}s]
\\&\le\sum_{n=0}^{\left [ T/d  \right ]-1 } (\left \|\mathbb{E}\int_{nd}^{(n+1)d} f^{\varepsilon}(s,u ^{*}(nd))-f^{*}(u ^{*}(nd)) \text{d}s\right \|^{2}_{\mathcal{H} ^{0}})^{\frac{1}{2} } (\mathbb{E}\left \| u ^{\varepsilon}(nd)-u ^{*}(nd) \right \| ^{2}_{\mathcal{H} ^{0}})^{\frac{1}{2} } \nonumber
\\&~~~+C(\left \| u_{0} ^{* } \right \|^{2}_{\mathcal{H} ^{0}}+\left \| u^{\varepsilon} _{0 } \right \|^{2}_{\mathcal{H} ^{0}}+1 )d \nonumber
\\&\le C(\left \| u_{0} ^{* } \right \|^{2}_{\mathcal{H} ^{0}}+\left \| u^{\varepsilon} _{0 } \right \|^{2}_{\mathcal{H} ^{0}}+1 )[d+\frac{T\varepsilon }{d} \max_{0\le n\le \left [ T/d \right ]-1, n\in \mathbb{N}^{+} } (\left \|\mathbb{E}\int_{\frac{nd}{\varepsilon }}^{\frac{(n+1)d}{\varepsilon }} f(r,u ^{*}(nd))-f^{*}(u ^{*}(nd)) \text{d}r\right \|^{2}_{\mathcal{H} ^{0}})^{\frac{1}{2} }\nonumber
\\&\le C(\left \| u_{0} ^{* } \right \|^{2}_{\mathcal{H} ^{0}}+\left \| u^{\varepsilon} _{0 } \right \|^{2}_{\mathcal{H} ^{0}}+1 )[d+\mathcal{R} _{3}(\frac{d}{\varepsilon } )].\nonumber
	\end{align}
Substituting \eqref{a13} and \eqref{a17} into \eqref{a8} gives
\begin{align}\label{a36}
	\Sigma _{3}&=\mathbb{E}\underset{t\in [0,T\wedge \tau _{R}]}{\sup}\int_{0}^{t}\left \langle u^{\varepsilon}(s)-u^{*}(s),f^{\varepsilon}(t,u^{\varepsilon}(s))-f^{*}(t,u^{*}(s)) \right \rangle _{\mathcal{H} ^{0}}\text{d}s\nonumber
\\&\le c\int_{0}^{T\wedge \tau _{R}}\mathcal{A} (\mathbb{E}\left \| u^{\varepsilon }(s) -u^{* }(s) \right \|^{2 }_{\mathcal{H} ^{0}})\text{d}s+C\sqrt{\mathcal{A} ((\left \| u_{0} ^{* } \right \|^{2}_{\mathcal{H} ^{1}}+\left \| u_{0} ^{* } \right \|^{6}_{\mathcal{H} ^{1}}+\left \| u_{0} ^{\varepsilon  } \right \|^{2}_{\mathcal{H} ^{0}}+1 )d^{\frac{1 }{2} })}
\\&~~~+C_{T}(\left \| u_{0} ^{\varepsilon } \right \|^{2}_{\mathcal{H} ^{1}}+\left \| u_{0} ^{\varepsilon } \right \|^{6}_{\mathcal{H} ^{1}}+\left \| u_{0} ^{* } \right \|^{2}_{\mathcal{H} ^{1}}+\left \| u_{0} ^{* } \right \|^{6}_{\mathcal{H} ^{1}}+1 )(\mathcal{R} _{3}(\frac{d}{\varepsilon } )+d+d^{\frac{1}{4} }).\nonumber
	\end{align}
For $\Sigma _{4}$,  we have
 \begin{align}\label{a21}
		\Sigma _{4}&=\mathbb{E} \int_{0}^{T\wedge \tau _{R}}\left \| S^{\varepsilon}_{3}(s,x,u^{\varepsilon }(s))-S^{*}_{3}(x,u^{*}(s)) \right \|_{\mathscr{L}(K,\mathcal{H} ^{0})} ^{2}\text{d}s\nonumber
\\&\le \mathbb{E} \int_{0}^{T\wedge \tau _{R}}\left \| S^{\varepsilon}_{3}(s,x,u^{\varepsilon }(s))-S^{\varepsilon}_{3}(s,x,u^{*}(s)) \right \|_{\mathscr{L}(K,\mathcal{H} ^{0})} ^{2}\text{d}s\nonumber
\\&~~~+\mathbb{E} \int_{0}^{T\wedge \tau _{R}}\left \| S^{\varepsilon}_{3}(s,x,u^{*}(s))-S^{*}_{3}(x,u^{*}(s)) \right \|_{\mathscr{L}(K,\mathcal{H} ^{0})} ^{2}\text{d}s
\\&\le \sup_{t\in[0,T],x\in \mathbb{D} } \left \| \mathcal{K} (t,x) \right \|^{2}_{\mathscr{L}(K,L^{2})}\mathbb{E} \int_{0}^{T\wedge \tau _{R}}\left \| \nabla u^{\varepsilon }(s)-\nabla u^{*}(s) \right \|_{\mathcal{H} ^{0}} ^{2}\text{d}s+\Sigma _{4}^{1},\nonumber
	\end{align}
where
\begin{align}\label{a22}
	\Sigma _{4}^{1}&=\mathbb{E} \int_{0}^{T\wedge \tau _{R}}\left \| S^{\varepsilon}_{3}(s,x,u^{*}(s))-S^{*}_{3}(x,u^{*}(s)) \right \|_{\mathscr{L}(K,\mathcal{H} ^{0})} ^{2}\text{d}s \nonumber
\\&\le\mathbb{E} \int_{0}^{T\wedge \tau _{R}}\left \| S^{\varepsilon}_{3}(s,x,u^{*}(s))-S^{\varepsilon}_{3}(s,x,\overline{ u}^{*}(s)) \right \|_{\mathscr{L}(K,\mathcal{H} ^{0})} ^{2}\text{d}s
\\&~~~+\mathbb{E} \int_{0}^{T\wedge \tau _{R}}\left \| S^{\varepsilon}_{3}(s,x,\overline{ u}^{*}(s))-S^{*}_{3}(x,\overline{u}^{*}(s)) \right \|_{\mathscr{L}(K,\mathcal{H} ^{0})} ^{2}\text{d}s\nonumber
\\&~~~+\mathbb{E} \int_{0}^{T\wedge \tau _{R}}\left \| S^{*}_{3}(x,\overline{u}^{*}(s)) -S^{*}_{3}(x,u^{*}(s)) \right \|_{\mathscr{L}(K,\mathcal{H} ^{0})} ^{2}\text{d}s \nonumber
\\&:=\Sigma _{4}^{1,1}+\Sigma _{4}^{1,2}+\Sigma _{4}^{1,3}. \nonumber
	\end{align}
Then, by  H\"older's inequality and \eqref{t7}, we have
\begin{align}\label{a23}
	\Sigma _{4}^{1,1}&=\mathbb{E} \int_{0}^{T\wedge \tau _{R}}\left \| S^{\varepsilon}_{3}(s,x,u^{*}(s))-S^{\varepsilon}_{3}(s,x,\overline{ u}^{*}(s)) \right \|_{\mathscr{L}(K,\mathcal{H} ^{0})} ^{2}\text{d}s\nonumber
\\&=\mathbb{E} \int_{0}^{T\wedge \tau _{R}}\left \langle (\mathcal{K} ^{\varepsilon}(s ,x)\cdot \nabla )u^{*}(s)-(\mathcal{K} ^{\varepsilon}(s,x)\cdot \nabla )\overline{ u}^{*}(s),(\mathcal{K} ^{\varepsilon}(s,x)\cdot \nabla )u^{*}(s)-(\mathcal{K} ^{\varepsilon}(s,x)\cdot \nabla )\overline{ u}^{*}(s) \right \rangle_{\mathcal{H} ^{0}} \text{d}s\nonumber
\\&=-\mathbb{E} \int_{0}^{T\wedge \tau _{R}}\left \langle \nabla S^{\varepsilon}_{3}(s,x,u^{*}(s))-\nabla S^{\varepsilon}_{3}(s,x,\overline{ u}^{*}(s)),\mathcal{K} ^{\varepsilon}(s,x)\cdot u^{*}(s)-\mathcal{K} ^{\varepsilon}(s,x)\cdot\overline{ u}^{*}(s) \right \rangle_{\mathcal{H} ^{0}} \text{d}s\nonumber
\\&\le (\mathbb{E} \int_{0}^{T\wedge \tau _{R}}\left \|  S^{\varepsilon}_{3}(s,x,u^{*}(s)) \right \|^{2}_{\mathscr{L}(K,\mathcal{H} ^{1})} +\left \|  S^{\varepsilon}_{3}(s,x,\overline{ u}^{*}(s)) \right \|^{2}_{\mathscr{L}(K,\mathcal{H} ^{1})} \text{d}s)^{\frac{1}{2} }
\\&~~~\times (\mathbb{E} \int_{0}^{T\wedge \tau _{R}}\left \| \mathcal{K} ^{\varepsilon}(s,x)\cdot u^{*}(s)-\mathcal{K} ^{\varepsilon}(s,x)\cdot\overline{ u}^{*}(s) \right \|^{2}_{\mathcal{H} ^{0}} \text{d}s)^{\frac{1}{2} }\nonumber
\\&\le C(\mathbb{E} \int_{0}^{T\wedge \tau _{R}}[\left \|  u^{*}(s) \right \|^{2}_{\mathcal{H} ^{1}} +\left \| \overline{ u}^{*}(s) \right \|^{2}_{\mathcal{H} ^{1}}+\left \|  u^{*}(s) \right \|^{2}_{\mathcal{H} ^{2}} +\left \| \overline{ u}^{*}(s) \right \|^{2}_{\mathcal{H} ^{2}} +1] \text{d}s)^{\frac{1}{2} }\nonumber
\\&~~~\times (\mathbb{E} \int_{0}^{T\wedge \tau _{R}}\left \|u^{*}(s)-\overline{ u}^{*}(s) \right \|^{2}_{\mathcal{H} ^{0}} \text{d}s)^{\frac{1}{2} }\nonumber
\\&\le C(\left \| u_{0} ^{* } \right \|^{2}_{\mathcal{H} ^{1}}+\left \| u_{0} ^{* } \right \|^{2}_{\mathcal{H} ^{0}}+1 )d^{\frac{1}{4} }.\nonumber
	\end{align}
Similarly, for $\Sigma _{4}^{1,3}$,
\begin{align}\label{a25}
	\Sigma _{4}^{1,3}&=\mathbb{E} \int_{0}^{T\wedge \tau _{R}}\left \| S^{*}_{3}(x,\overline{u}^{*}(s)) -S^{*}_{3}(x,u^{*}(s)) \right \|_{\mathscr{L}(K,\mathcal{H} ^{0})} ^{2}\text{d}s\nonumber
\\&\le C(\left \| u_{0} ^{* } \right \|^{2}_{\mathcal{H} ^{1}}+\left \| u_{0} ^{* } \right \|^{2}_{\mathcal{H} ^{0}}+1 )d^{\frac{1}{4} }.
	\end{align}
Subsequently, we will employ the time discretization technique to address $\Sigma _{4}^{1,2}$:
\begin{align}\label{a26}
	\Sigma _{4}^{1,2}&=\mathbb{E} \int_{0}^{T\wedge \tau _{R}}\left \| S^{\varepsilon}_{3}(s,x,\overline{ u}^{*}(s))-S^{*}_{3}(x,\overline{u}^{*}(s)) \right \|_{\mathscr{L}(K,\mathcal{H} ^{0})} ^{2}\text{d}s   \nonumber
\\&\le \mathbb{E}\sum_{n=0}^{\left [ T/d  \right ]-1 } \int_{nd}^{(n+1)d}\left \| (\mathcal{K} ^{\varepsilon }(s,x)\cdot \nabla )u^{*}(nd)-(\mathcal{K} ^{* }(x)\cdot \nabla )u^{*}(nd) \right \|_{\mathscr{L}(K,\mathcal{H} ^{0})} ^{2}\text{d}s  \nonumber
\\&~~~+\mathbb{E}\int_{\left [ T/d  \right ]d}^{T}\left \| S^{\varepsilon}_{3}(s,x, u^{*}(\left [ T/d  \right ]d))-S^{*}_{3}(x,u^{*}(nd)) \right \|_{\mathscr{L}(K,\mathcal{H} ^{0})} ^{2}\text{d}s   \nonumber
\\& \le \varepsilon\sum_{n=0}^{\left [ T/d \right ]-1 } \mathbb{E}\int_{\frac{nd}{\varepsilon }}^{\frac{(n+1)d}{\varepsilon }}\left \| \mathcal{K} (r,x)-\mathcal{K} ^{* }(x) \right \|_{\mathscr{L}(K,\mathcal{H} ^{0})} ^{2}\text{d}r\cdot \mathbb{E}\left \|  u^{*}(nd) \right \|_{\mathscr{L}(K,\mathcal{H} ^{1})} ^{2}
\\&~~~+C(\left \| u_{0} ^{* } \right \|^{2}_{\mathcal{H} ^{0}}+1 )d \nonumber
\\&\le C(\left \| u_{0} ^{* } \right \|^{2}_{\mathcal{H} ^{1}}+\left \| u_{0} ^{* } \right \|^{2}_{\mathcal{H} ^{0}}+1 )[d^{\frac{1}{4} }+d+\mathcal{R}_{5}(\frac{d}{\varepsilon } )]
. \nonumber
	\end{align}
In conclusion, by \eqref{a21}-\eqref{a26}, we derive
\begin{align}\label{s51}
	\Sigma _{4}&\le \frac{a_{1}}{73} \mathbb{E} \int_{0}^{T\wedge \tau _{R}}\left \| \nabla u^{\varepsilon }(s)-\nabla u^{*}(s) \right \|_{\mathcal{H} ^{0}} ^{2}\text{d}s+C_{T}(\left \| u_{0} ^{* } \right \|^{2}_{\mathcal{H} ^{1}}+\left \| u_{0} ^{* } \right \|^{2}_{\mathcal{H} ^{0}}+1 )[d+\mathcal{R}_{5}(\frac{d}{\varepsilon } )].
	\end{align}
Similarly to the derivation of $\Sigma _{4}$, for $\Sigma _{5}$, by \textbf{(H2)}, H\"older's inequality, Jensen's inequality and Lemma \ref{le1} we obtain
\begin{align}\label{s52}
	\Sigma _{5}&=\mathbb{E}\int_{0}^{T\wedge \tau _{R}}\left \| g^{\varepsilon}(s,u^{\varepsilon }_{s})- g^{*}(u^{* }_{s})\right \|^{2}_{\mathscr{L}(K,\mathcal{H} ^{0})}\text{d}s   \nonumber
\\&\le \mathbb{E}\int_{0}^{T\wedge \tau _{R}}\left \| g^{\varepsilon}(s,u^{\varepsilon }_{s})- g^{\varepsilon}(s,u^{* }_{s})\right \|^{2}_{\mathscr{L}(K,\mathcal{H} ^{0})}\text{d}s
+\mathbb{E}\int_{0}^{T}\left \|  g^{\varepsilon}(s,u^{* }_{s})-g^{*}(u^{* }_{s})\right \|^{2}_{\mathscr{L}(K,\mathcal{H} ^{0})}\text{d}s\nonumber
\\&\le c\int_{0}^{T\wedge \tau _{R}}\mathcal{A} (\mathbb{E}\left \| u^{\varepsilon }(s) -u^{* }(s) \right \|^{2   }_{\mathcal{H} ^{0}})\text{d}s+\mathbb{E}\int_{0}^{T}\left \|  g^{\varepsilon}(s,u^{* }_{s})-g^{\varepsilon}(s,\overline{u}^{* }_{s})\right \|^{2}_{\mathscr{L}(K,\mathcal{H} ^{0})}\text{d}s\nonumber
\\&~~~+\mathbb{E}\int_{0}^{T}\left \|  g^{\varepsilon}(s,\overline{u}^{* }_{s})-g^{*}(\overline{u}^{* }_{s})\right \|^{2}_{\mathscr{L}(K,\mathcal{H} ^{0})}\text{d}s
+\mathbb{E}\int_{0}^{T}\left \|  g^{*}(\overline{u}^{* }_{s})-g^{*}(u^{* }_{s})\right \|^{2}_{\mathscr{L}(K,\mathcal{H} ^{0})}\text{d}s.   \nonumber
\\&\le c[\int_{0}^{T\wedge \tau _{R}}\mathcal{A} (\mathbb{E}\left \| u^{\varepsilon }(s) -u^{* }(s) \right \|^{2  }_{\mathcal{H} ^{0}})\text{d}s+\int_{0}^{T\wedge \tau _{R}}\mathcal{A} (\mathbb{E}\left \| u^{* }(s) -\overline{u}^{* }(s) \right \|^{2 }_{\mathcal{H} ^{0}})\text{d}s]
\\&~~~+\mathbb{E}\int_{0}^{T}\left \|  g^{\varepsilon}(s,\overline{u}^{* }_{s})-g^{*}(\overline{u}^{* }_{s})\right \|^{2}_{\mathscr{L}(K,\mathcal{H} ^{0})}\text{d}s\nonumber
\\&\le c\int_{0}^{T\wedge \tau _{R}}\mathcal{A} (\mathbb{E}\left \| u^{\varepsilon }(s) -u^{* }(s) \right \|^{2   }_{\mathcal{H} ^{0}})\text{d}s+C_{T}\mathcal{A} ((\left \| u_{0} ^{* } \right \|^{2}_{\mathcal{H} ^{1}}+\left \| u_{0} ^{* } \right \|^{6}_{\mathcal{H} ^{1}}+1 )d^{\frac{1 }{2} })\nonumber
\\&~~~+\mathbb{E}\int_{0}^{T}\left \|  g^{\varepsilon}(s,\overline{u}^{* }_{s})-g^{*}(\overline{u}^{* }_{s})\right \|^{2}_{\mathscr{L}(K,\mathcal{H} ^{0})}\text{d}s,\nonumber
	\end{align}
where
\begin{align}\label{s58}
	&\mathbb{E}\int_{0}^{T}\left \|  g^{\varepsilon}(s,\overline{u}^{* }_{s})-g^{*}(\overline{u}^{* }_{s})\right \|^{2}_{\mathscr{L}(K,\mathcal{H} ^{0})}\text{d}s \nonumber \\&=\mathbb{E}\sum_{n=0}^{\left [ T/d  \right ]-1 } \int_{nd}^{(n+1)d}\left \|  g^{\varepsilon}(s,u^{* }(nd))-g^{*}(u^{* }(nd))\right \|^{2}_{\mathscr{L}(K,\mathcal{H} ^{0})}\text{d}s  \nonumber
\\&~~~+\mathbb{E}\int_{\left [ T/d  \right ]d}^{T}\left \|  g^{\varepsilon}(s,u^{* }_{\left [ T/d  \right ]d})-g^{*}(u^{* }_{\left [ T/d  \right ]d})\right \|^{2}_{\mathscr{L}(K,\mathcal{H} ^{0})}\text{d}s  \\& \le \varepsilon\sum_{n=0}^{\left [T/d  \right ]-1 } \mathbb{E}\int_{\frac{nd}{\varepsilon }}^{\frac{(n+1)d}{\varepsilon }}\left \|  g(r,u^{* }_{nd})-g^{*}(u^{* }_{nd})\right \|^{2}_{\mathscr{L}(K,\mathcal{H} ^{0})}\text{d}r+C(1+\mathbb{E}\left \| u^{* }(\left [ T/d  \right ]d)\right \|^{2}_{h})d   \nonumber
\\&\le C_{T}(\left \|u^{* }_{0} \right \| _{\mathcal{H} ^{0}}^{2}+1)[d+\mathcal{R}  _{5}(\frac{d}{\varepsilon } )]
. \nonumber
	\end{align}
Then, by \eqref{s52} and \eqref{s58}, we have
\begin{align}\label{s60}
	\Sigma _{5}&\le c\int_{0}^{T\wedge \tau _{R}}\mathcal{A} (\mathbb{E}\left \| u^{\varepsilon }(s) -u^{* }(s) \right \|^{2  }_{\mathcal{H} ^{0}})\text{d}s+C_{T}[\mathcal{A} ((\left \| u_{0} ^{* } \right \|^{2}_{\mathcal{H} ^{1}}+\left \| u_{0} ^{* } \right \|^{6}_{\mathcal{H} ^{1}}+1 )d^{\frac{1 }{2} })\nonumber
\\&~~~+(\left \| u_{0} ^{* } \right \|^{2}_{\mathcal{H} ^{1}}+\left \| u_{0} ^{* } \right \|^{6}_{\mathcal{H} ^{1}}+1 )(d+\mathcal{R}  _{5}(\frac{d}{\varepsilon } ))].
	\end{align}
Let $$\mathcal{W}(\varepsilon ,*)=(\left \| u_{0}^{\varepsilon} \right \|_{\mathcal{H}^{1} }^{2}+\left \| u_{0}^{\varepsilon} \right \|_{\mathcal{H}^{1} }^{6}+\left \| u_{0} ^{* } \right \|^{2}_{\mathcal{H} ^{1}}+\left \| u_{0} ^{* } \right \|^{4}_{\mathcal{H} ^{1}}+\left \| u_{0} ^{* } \right \|^{6}_{\mathcal{H} ^{1}}+1 ),$$ and substituting \eqref{a19}, \eqref{a33}, \eqref{a36}, \eqref{s51} and \eqref{s60} into \eqref{t335} implies
\begin{align}\label{c222}
		 & \mathbb{E} (\underset{r\in[0,t]}{\sup}\left \| u^{\varepsilon}(r\wedge \tau _{R};u^{\varepsilon}_{0} )-u^{*}(r\wedge \tau _{R};u^{*}_{0}) \right \| _{\mathcal{H} ^{0}}^{2}) \nonumber
 \\&\le 2\left \| u^{\varepsilon}_{0}-u^{*}_{0} \right \| _{\mathcal{H} ^{0}}^{2}+C_{a_{1},a_{4},R,c}\int_{0}^{t}[\mathbb{E}\left \| u^{\varepsilon}(s\wedge \tau _{R})-u^{*}(s\wedge \tau _{R}) \right \| ^{2 }_{\mathcal{H} ^{0}}+\mathcal{A} (\mathbb{E}\left \| u^{\varepsilon}(s\wedge \tau _{R})-u^{*}(s\wedge \tau _{R}) \right \| ^{2 }_{\mathcal{H} ^{0}})]\text{d}s\nonumber
 \\&~~~+C_{T}[\mathcal{W}(\varepsilon ,*)(\sum_{i=1}^{5} \mathcal{R} _{i}(\frac{d}{\varepsilon } )+d+d^{\frac{1}{4} }+d^{\frac{1}{16} })
 +\mathcal{A} (\mathcal{W}(\varepsilon ,*)d^{\frac{1 }{2} })
+\sqrt{\mathcal{A} (\mathcal{W}(\varepsilon ,*)d^{\frac{1 }{2} })}]
 \\&:=\phi(t),\nonumber
	\end{align}
for any $t\in [0,T]$. For simplicity of description, let $d=\varepsilon ^{\frac{1}{2} }$,
then,  we have
\begin{align*}
		 \phi(t)&=2\left \| u^{\varepsilon}_{0}-u^{*}_{0} \right \| _{\mathcal{H} ^{0}}^{2}+C\int_{0}^{t}[\mathbb{E}\left \| u^{\varepsilon}(s\wedge \tau _{R})-u^{*}(s\wedge \tau _{R}) \right \| ^{2 }_{\mathcal{H} ^{0}}+\mathcal{A} (\mathbb{E}\left \| u^{\varepsilon}(s\wedge \tau _{R})-u^{*}(s\wedge \tau _{R}) \right \| ^{2 }_{\mathcal{H} ^{0}})]\text{d}s
\\&~~~+C_{T}[\mathcal{W}(\varepsilon ,*)(\sum_{i=1}^{5} \mathcal{R} _{i}(\varepsilon ^{-\frac{1}{2} } )+\varepsilon^{\frac{1}{32} })+\mathcal{A} (\mathcal{W}(\varepsilon ,*)\varepsilon ^{\frac{1}{4} })
+\sqrt{\mathcal{A} (\mathcal{W}(\varepsilon ,*)\varepsilon ^{\frac{1}{4} })}].
	\end{align*}
 Indeed, we assert that when $\lim_{\varepsilon  \to 0} \left \|u_{0}^{\varepsilon }- u_{0} ^{*} \right \|^{2}_{\mathcal{H} ^{0}}=0$,
 \begin{align}\label{c121}
\lim_{\varepsilon  \to 0} \mathbb{E}\underset{t\in [0,T]}{\sup}   \left \|u^{\varepsilon }(t\wedge \tau _{R};u_{0}^{\varepsilon }) -u^{* }(t\wedge \tau _{R};u_{0} ^{*}) \right \|^{2}_{\mathcal{H} ^{0}}=0.
\end{align}
For any fixed $T>0$,
\begin{align}\label{c2}
		\Gamma (\mathbb{E}\underset{t\in [0,T]}{\sup}   \left \|u^{\varepsilon }(t\wedge \tau _{R};u_{0} ^{\varepsilon }) -u^{* }(t\wedge \tau _{R};u_{0}^{*}) \right \|^{2}_{\mathcal{H} ^{0}})\le \Gamma ( \phi  (T)),
	\end{align}
further,
\begin{align}\label{c3}
		&\Gamma( \phi  (T))\nonumber\\&=\Gamma(\phi  (0))+\int_{0}^{T}\Gamma'( \phi  (s))\text{d} \phi (s) \nonumber
\\&\le \Gamma(2\left \| u^{\varepsilon}_{0}-u^{*}_{0} \right \| _{\mathcal{H} ^{0}}^{2} +C_{T}\mathcal{W}(\varepsilon ,*)[(\sum_{i=1}^{5} \mathcal{R} _{i}(\varepsilon ^{-\frac{1}{2} } )+\varepsilon^{\frac{1}{32} })+\mathcal{A} (\mathcal{W}(\varepsilon ,*)\varepsilon ^{\frac{1}{4} })
+\sqrt{\mathcal{A} (\mathcal{W}(\varepsilon ,*)\varepsilon ^{\frac{1}{4} })}])\nonumber
\\&~~~+C\int_{0}^{T} \frac{ \mathbb{E}\underset{r\in [0,s]}{\sup}\left \| u^{\varepsilon }(r\wedge \tau _{R};u^{\varepsilon }_{0}) -u^{* }(r\wedge \tau _{R};u ^{*}_{0})\right \|^{2}_{\mathcal{H} ^{0}} }{ \mathcal{A}(\phi(s))+\phi(s)}\text{d}s
\\&~~~+C\int_{0}^{T} \frac{ \mathcal{A} (\mathbb{E}\left \| u^{\varepsilon}(s\wedge \tau _{R})-u^{*}(s\wedge \tau _{R}) \right \| ^{2 }_{\mathcal{H} ^{0}}) }{ \mathcal{A}(\phi(s))+\phi(s)}\text{d}s  \nonumber
\\&\le \Gamma(2\left \| u^{\varepsilon}_{0}-u^{*}_{0} \right \| _{\mathcal{H} ^{0}}^{2} +C_{T}\mathcal{W}(\varepsilon ,*)[(\sum_{i=1}^{5} \mathcal{R} _{i}(\varepsilon ^{-\frac{1}{2} } )+\varepsilon^{\frac{1}{32} })+\mathcal{A} (\mathcal{W}(\varepsilon ,*)\varepsilon ^{\frac{1}{4} })
+\sqrt{\mathcal{A} (\mathcal{W}(\varepsilon ,*)\varepsilon ^{\frac{1}{4} })}])+C_{T}.\nonumber
	\end{align}
By \eqref{c2} and \eqref{c3}, when $\varepsilon\to 0$, we obtain
\begin{align}\label{c999}
\Gamma (\mathbb{E}\underset{t\in [0,T]}{\sup}   \left \|u^{\varepsilon }(t\wedge \tau _{R};u_{0} ^{\varepsilon }) -u^{* }(t\wedge \tau _{R};u_{0}^{*}) \right \|^{2}_{\mathcal{H} ^{0}})\to -\infty,
	\end{align}
 which implies
 \begin{align*}
\lim_{\varepsilon  \to 0} \mathbb{E}\underset{t\in [0,T]}{\sup}   \left \|u^{\varepsilon }(t\wedge \tau _{R};u_{0}^{\varepsilon }) -u^{* }(t\wedge \tau _{R};u_{0} ^{*}) \right \|^{2}_{\mathcal{H} ^{0}}=0.
\end{align*}
Furthermore, by applying Remark \ref{re1} and Fatou's lemma, and letting $R\to \infty$, we get
\begin{align*}
\lim_{\varepsilon  \to 0} \mathbb{E}\underset{t\in [0,T]}{\sup}   \left \|u^{\varepsilon }(t;u_{0}^{\varepsilon }) -u^{* }(t;u_{0} ^{*}) \right \|^{2}_{\mathcal{H} ^{0}}=0.
\end{align*}
This completes the proof.
\end{proof}

If $\eta _{1}^{\varepsilon}(t )=\eta _{1}^{*}=\eta _{2}^{\varepsilon}(t )=\eta _{2}^{*}=1$ and $\mathcal{K} (t,x)=0$, it follows that we can establish the averaging principle within the framework of first-order Sobolev spaces $\mathcal{H} ^{1}$. Firstly, it is essential to introduce the following assumptions:
~\\
\\\textbf{(H5')} There exist functions $\mathcal{R}_{3} $, $\mathcal{R}_{4} $, $\mathcal{R}_{5} $  and $f^{*}\in C(\mathcal{H} ^{1}(\mathbb{D} ),\mathcal{H} ^{1}(\mathbb{D} ))$, $g^{*}\in C(\mathcal{H} ^{1}(\mathbb{D} ),\mathscr{L}(K,\mathcal{H} ^{1}(\mathbb{D} )) )$ such that for any $T>0$, $t\in [0,T]$, $u \in \mathcal{H} ^{1}(\mathbb{D} )$ and $x\in \mathbb{D}$,
\begin{align*}
		\frac{1}{T} \left \| \int_{t}^{t+T}[f(s,u ) -f^{*}(u )] \text{d}s\right \| _{\mathcal{H} ^{1}}\le \mathcal{R}_{3}(T)(\left \| u \right \|_{\mathcal{H} ^{1}}+M ),
	\end{align*}
\begin{align*}
		\frac{1}{T}\int_{t}^{t+T} \left \| g(s,u ) -g^{*}(u )\right \|^{2} _{\mathscr{L}(K,\mathcal{H} ^{1}) } \text{d}s\le \mathcal{R}_{4}(T)(\left \| u \right \|^{2}_{\mathcal{H} ^{1}}+M ).
	\end{align*}
Then the second result on the averaging principle is as follows:
\begin{theorem}\label{th3} Consider \eqref{a2} and \eqref{a222}. Suppose that the assumptions \textbf{(H1)}, \textbf{(H2)}, \textbf{(H6)} and \textbf{(H5')} hold and $\eta _{1}^{\varepsilon}(t )=\eta _{1}^{*}=\eta _{2}^{\varepsilon}(t )=\eta _{2}^{*}=1$, $\mathcal{K} (t,x)=0$.  For any initial values $u_{0}^{\varepsilon }, u_{0} ^{*}\in \mathcal{H} ^{1}$ and $T > 0$, assume further that $\lim_{\varepsilon  \to 0} \left \|u_{0}^{\varepsilon }- u_{0} ^{*} \right \|^{2}_{\mathcal{H} ^{1}}=0$. Then, we have
\begin{align}\label{c12121}
		\lim_{\varepsilon  \to 0} \mathbb{E}\underset{t\in [0,T]}{\sup}   \left \|u^{\varepsilon }(t;u_{0}^{\varepsilon }) -u^{* }(t;u_{0} ^{*}) \right \|^{2}_{\mathcal{H} ^{1}}=0.
	\end{align}
\end{theorem}
\begin{remark}\label{re2} When $\eta _{1}^{\varepsilon}(t )=\eta _{1}^{*}=\eta _{2}^{\varepsilon}(t )=\eta _{2}^{*}=1$ and $\mathcal{K} (t,x)=0$, it follows that the strong convergence \eqref{c12121} is established within the first-order Sobolev space $\mathcal{H} ^{1}$, and this result is evidently stronger than \eqref{c1}. Indeed, under the conditions  \textbf{(H1)} and \textbf{(H2)}, the derivation of the coefficients $f$, $g$ is analogous to that in Theorem \ref{th2}, primarily employing the  Khasminskii time discretization. The critical aspect involves employing stopping techniques and prior estimation to analyze $S_{1}(u )$ and $S_{2}(u )$, ultimately leading to \eqref{c12121}.
\end{remark}
\begin{proof}
Let $$\tau _{M}=\underset{t\in[0,T]}{\inf} \{\left \| u^{\varepsilon}(t;u^{\varepsilon}_{0} ) \right \| _{\mathcal{H} ^{1}}^{2}\vee \left \| u^{*}(t;u^{*}_{0}) \right \| _{\mathcal{H} ^{1}}^{2}\vee\int_{0}^{t} \left \| u^{\varepsilon }(s)\right \| ^{2}_{\mathcal{H}^{2}}\text{d}s\vee \int_{0}^{t} \left \| u^{* }(s)\right \| ^{2}_{\mathcal{H}^{2}}\text{d}s>M \}.$$ By applying It$\hat{\text{o}} $'s formula formula to $\left \| u^{\varepsilon}(t;u^{\varepsilon}_{0} )-u^{*}(t;u^{*}_{0}) \right \| _{\mathcal{H} ^{1}}^{2}$, and incorporating \textbf{(H3)}, Young's inequality and the B-D-G inequality, we derive
\begin{align}\label{111}
		 & \left \| u^{\varepsilon}(t;u^{\varepsilon}_{0} )-u^{*}(t;u^{*}_{0}) \right \| _{\mathcal{H} ^{1}}^{2} \nonumber
\\&=\left \| u^{\varepsilon}_{0}-u^{*}_{0} \right \| _{\mathcal{H} ^{1}}^{2}+\int_{0}^{t} [2 \left \langle u^{\varepsilon}(s)-u^{*}(s),S_{1}(u^{\varepsilon}(s) )- S_{1}(u^{*}(s) ) \right \rangle _{\mathcal{H} ^{1}}
\\&~~~-2  \left \langle u^{\varepsilon}(s)-u^{*}(s),S_{2}(u^{\varepsilon}(s) )- S_{2}(u^{*}(s) ) \right \rangle _{\mathcal{H} ^{1}}+2 \left \langle u^{\varepsilon}(s)-u^{*}(s),f^{\varepsilon}(s,u^{\varepsilon}(s))-f^{*}(u^{*}(s)) \right \rangle _{\mathcal{H} ^{1}}\nonumber
\\&~~~+\left \| g^{\varepsilon}(s,u^{\varepsilon}(s))-g^{*}(u^{*}(s)) \right \|_{\mathscr{L}(K,\mathcal{H} ^{1})} ^{2}]\text{d}s+2\int_{0}^{t}\left \langle u^{\varepsilon}(s)-u^{*}(s),[g^{\varepsilon}(t,u^{\varepsilon}(s))-g^{*}(u^{*}(s))]\text{d}W(s) \right \rangle _{\mathcal{H} ^{1} }, \nonumber
	\end{align}
where
\begin{align}\label{112}
		 &\int_{0}^{t}\left \langle u^{\varepsilon}(s)-u^{*}(s),S_{1}(u^{\varepsilon}(s) )- S_{1}(u^{*}(s) ) \right \rangle _{\mathcal{H} ^{1}}\text{d}s\nonumber
\\&=\int_{0}^{t}\left \langle (I-\Delta )(u^{\varepsilon}(s)-u^{*}(s)),S_{1}(u^{\varepsilon}(s) )- S_{1}(u^{*}(s) ) \right \rangle _{\mathcal{H} ^{0}}\text{d}s\nonumber
\\&=\int_{0}^{t}\left \langle (I-\Delta )(u^{\varepsilon}(s)-u^{*}(s)),(u^{\varepsilon}(s) - u^{*}(s) ) \right \rangle _{\mathcal{H} ^{0}}\text{d}s
\\&~~~-\int_{0}^{t}\left \langle (I-\Delta )(u^{\varepsilon}(s)-u^{*}(s)),(I-\Delta )(u^{\varepsilon}(s) - u^{*}(s) ) \right \rangle _{\mathcal{H} ^{0}}\text{d}s\nonumber
\\&\le 2\int_{0}^{t}\left \|u^{\varepsilon}(s) - u^{*}(s) ) \right \| ^{2}_{\mathcal{H} ^{1}}\text{d}s-\int_{0}^{t}\left \|u^{\varepsilon}(s) - u^{*}(s) ) \right \| ^{2}_{\mathcal{H} ^{2}}\text{d}s,\nonumber
	\end{align}
and according to (\cite{ref27}, (3.31)–(3.32)) and Young's inequality,
\begin{align}\label{113}
		 &-\int_{0}^{t}   \left \langle u^{\varepsilon}(s)-u^{*}(s),S_{2}(u^{\varepsilon}(s) )- S_{2}(u^{*}(s) ) \right \rangle _{\mathcal{H} ^{1}}\text{d}s\nonumber
\\&\le \int_{0}^{t}   \left \|  u^{\varepsilon}(s)-u^{*}(s)\right \| _{\mathcal{H} ^{2}} \left \| S^{1}_{2}(u^{\varepsilon}(s) )- S^{1}_{2}(u^{*}(s) )\right \| _{\mathcal{H} ^{0}}\text{d}s\nonumber
\\&~~~+\int_{0}^{t}   \left \|  u^{\varepsilon}(s)-u^{*}(s)\right \| _{\mathcal{H} ^{2}} \left \| \Psi _{N}(\left | u^{\varepsilon}(s) \right |^{2} ) u^{\varepsilon}(s)- \Psi _{N}(\left | u^{*}(s) \right |^{2} ) u^{*}(s)\right \| _{\mathcal{H} ^{0}}\text{d}s\nonumber
\\&\le C\int_{0}^{t}   \left \|  u^{\varepsilon}(s)-u^{*}(s)\right \| _{\mathcal{H} ^{2}} [\left \|  u^{\varepsilon}(s)-u^{*}(s)\right \|^{\frac{1}{2} } _{\mathcal{H} ^{2}}\left \|  u^{\varepsilon}(s)-u^{*}(s)\right \|^{\frac{1}{2} } _{\mathcal{H} ^{1}}\left \|  u^{\varepsilon}(s)\right \| _{\mathcal{H} ^{1}}
\\&~~~+\left \|  u^{\varepsilon}(s)-u^{*}(s)\right \| _{\mathcal{H} ^{1}}\left \|  u^{*}(s)\right \|^{\frac{1}{2} } _{\mathcal{H} ^{2}}\left \|  u^{*}(s)\right \|^{\frac{1}{2} } _{\mathcal{H} ^{1}}]\text{d}s\nonumber
\\&~~~+C\int_{0}^{t}   \left \|  u^{\varepsilon}(s)-u^{*}(s)\right \| _{\mathcal{H} ^{2}}  \left \|  u^{\varepsilon}(s)-u^{*}(s)\right \| _{\mathcal{H} ^{1}} ( \left \|  u^{\varepsilon}(s)\right \|^{2} _{\mathcal{H} ^{1}} + \left \|  u^{*}(s)\right \|^{2} _{\mathcal{H} ^{1}} )\text{d}s\nonumber
\\&\le \int_{0}^{t}\left \|u^{\varepsilon}(s) - u^{*}(s) ) \right \| ^{2}_{\mathcal{H} ^{2}}\text{d}s+ C\int_{0}^{t}\left \|u^{\varepsilon}(s) - u^{*}(s) ) \right \| ^{2}_{\mathcal{H} ^{1}}[\left \|u^{\varepsilon}(s) \right \| ^{4}_{\mathcal{H} ^{1}}+\left \|u^{*}(s) \right \| ^{2}_{\mathcal{H} ^{1}}+\left \|u^{*}(s) \right \| ^{4}_{\mathcal{H} ^{1}}\nonumber
\\&~~~+\left \|u^{*}(s) \right \| ^{2}_{\mathcal{H} ^{2}}]\text{d}s.\nonumber
	\end{align}
By applying Gronwall's lemma and taking the expectation, we derive
\begin{align}\label{115}
		 & \mathbb{E} \sup_{t\in [0,T\wedge \tau _{M}]} \left \| u^{\varepsilon}(t;u^{\varepsilon}_{0} )-u^{*}(t;u^{*}_{0}) \right \| _{\mathcal{H} ^{1}}^{2} \nonumber
\\&\le  \mathbb{E} \{\text{exp}[\int_{0}^{T}(1+ \left \|u^{\varepsilon}(s) \right \| ^{4}_{\mathcal{H} ^{1}}+\left \|u^{*}(s) \right \| ^{2}_{\mathcal{H} ^{1}}+\left \|u^{*}(s) \right \| ^{4}_{\mathcal{H} ^{1}}
+\left \|u^{*}(s) \right \| ^{2}_{\mathcal{H} ^{2}})\text{d}s]\cdot [\left \| u^{\varepsilon}_{0}-u^{*}_{0} \right \| _{\mathcal{H} ^{1}}^{2}\nonumber
\\&~~~+\int_{0}^{T} \left \| g^{\varepsilon}(s,u^{\varepsilon}(s))-g^{*}(u^{*}(s) \right \|_{\mathscr{L}(K,\mathcal{H} ^{1})} ^{2})\text{d}s\nonumber
\\&~~~+2\sup_{t\in [0,T\wedge \tau _{M}]}\left |\int_{0}^{t}  \left \langle u^{\varepsilon}(s)-u^{*}(s),f^{\varepsilon}(s,u^{\varepsilon}(s))-f^{*}(u^{*}(s)) \right \rangle _{\mathcal{H} ^{1}}\text{d}s\right | \nonumber
\\&~~~+2\sup_{t\in [0,T\wedge \tau _{M}]}\left |\int_{0}^{t}\left \langle u^{\varepsilon}(s)-u^{*}(s),[g^{\varepsilon}(t,u^{\varepsilon}(s))-g^{*}(u^{*}(s))]\text{d}W(s) \right \rangle _{\mathcal{H} ^{1} }\right |]\}
\\&\le C_{M,T}\{\left \| u^{\varepsilon}_{0}-u^{*}_{0} \right \| _{\mathcal{H} ^{1}}^{2}+\mathbb{E}\int_{0}^{T} +\left \| g^{\varepsilon}(s,u^{\varepsilon}(s))-g^{*}(u^{*}(s) \right \|_{\mathscr{L}(K,\mathcal{H} ^{1})} ^{2})\text{d}s\nonumber
\\&~~~+2\mathbb{E}\sup_{t\in [0,T\wedge \tau _{M}]}\left |\int_{0}^{t}  \left \langle u^{\varepsilon}(s)-u^{*}(s),f^{\varepsilon}(s,u^{\varepsilon}(s))-f^{*}(u^{*}(s)) \right \rangle _{\mathcal{H} ^{1}}\text{d}s\right | \nonumber
\\&~~~+2\mathbb{E}\sup_{t\in [0,T\wedge \tau _{M}]}\left |\int_{0}^{t}\left \langle u^{\varepsilon}(s)-u^{*}(s),[g^{\varepsilon}(t,u^{\varepsilon}(s))-g^{*}(u^{*}(s))]\text{d}W(s) \right \rangle _{\mathcal{H} ^{1} }\right |\}.\nonumber
	\end{align}
Following the derivation of \eqref{t21} and tilizing the B-D-G inequality, Young's inequality, we obtain
\begin{align*}
		&C_{M,T}\mathbb{E}\sup_{t\in [0,T\wedge \tau _{M}]}\left |\int_{0}^{t}\left \langle u^{\varepsilon}(s)-u^{*}(s),[g^{\varepsilon}(t,u^{\varepsilon}(s))-g^{*}(u^{*}(s))]\text{d}W(s) \right \rangle _{\mathcal{H} ^{1} }\right |
\\&\le C_{M,T}\mathbb{E}\left (  \int_{0}^{T\wedge \tau _{M}}\left \| u^{\varepsilon}(s)-u^{*}(s) \right \|^{2} _{\mathcal{H} ^{1} }\left \|g^{\varepsilon}(t,u^{\varepsilon}(s))-g^{*}(u^{*}(s)) \right \|^{2} _{\mathscr{L}(K,\mathcal{H} ^{1})}\text{d}s\right ) ^{\frac{1}{2} }
\\&\le \mathbb{E}[( \sup_{t\in [0,T\wedge \tau _{M}]} \left \| u^{\varepsilon}(t )-u^{*}(t) \right \| _{\mathcal{H} ^{1}}^{2})^{\frac{1}{2} }\left (C^{2}_{M,T}  \int_{0}^{T\wedge \tau _{M}}\left \|g^{\varepsilon}(t,u^{\varepsilon}(s))-g^{*}(u^{*}(s)) \right \|^{2} _{\mathscr{L}(K,\mathcal{H} ^{1})}\text{d}s\right ) ^{\frac{1}{2} }]
\\&\le \frac{1}{2} \mathbb{E}\sup_{t\in [0,T\wedge \tau _{M}]} \left \| u^{\varepsilon}(t )-u^{*}(t) \right \| _{\mathcal{H} ^{1}}^{2} +\frac{C^{2}_{M,T}}{2}\mathbb{E}  \int_{0}^{T}\left \|g^{\varepsilon}(t,u^{\varepsilon}(s))-g^{*}(u^{*}(s)) \right \|^{2} _{\mathscr{L}(K,\mathcal{H} ^{1})}\text{d}s,
	\end{align*}
which implies that
\begin{align}\label{116}
		 & \mathbb{E} \sup_{t\in [0,T\wedge \tau _{M}]} \left \| u^{\varepsilon}(t )-u^{*}(t) \right \| _{\mathcal{H} ^{1}}^{2} \nonumber
\\&\le C_{M,T}\{\left \| u^{\varepsilon}_{0}-u^{*}_{0} \right \| _{\mathcal{H} ^{1}}^{2}+\mathbb{E}\sup_{t\in [0,T\wedge \tau _{M}]}\left |\int_{0}^{t}  \left \langle u^{\varepsilon}(s)-u^{*}(s),f^{\varepsilon}(s,u^{\varepsilon}(s))-f^{*}(u^{*}(s)) \right \rangle _{\mathcal{H} ^{1}}\text{d}s\right |
\\&~~~+\mathbb{E}\int_{0}^{T} \left \| g^{\varepsilon}(s,u^{\varepsilon}(s))-g^{*}(u^{*}(s)) \right \|_{\mathscr{L}(K,\mathcal{H} ^{1})} ^{2}\text{d}s\}\nonumber
\\&:=C_{M,T}[\left \| u^{\varepsilon}_{0}-u^{*}_{0} \right \| _{\mathcal{H} ^{1}}^{2}+\mathcal{Z}_{1} +\mathcal{Z}_{2}]. \nonumber
	\end{align}
For $\mathcal{Z}_{1}$, analogous to the derivation presented in \eqref{a8}-\eqref{a17}, we can derive the following from \textbf{(H5')} and Lemma \ref{le1}:
\begin{align*}
	\mathcal{Z}_{1}&=\mathbb{E}\underset{t\in [0,T\wedge \tau _{M}]}{\sup}\left | \int_{0}^{t}\left \langle u^{\varepsilon}(s)-u^{*}(s),f^{\varepsilon}(t,u^{\varepsilon}(s))-f^{*}(t,u^{*}(s)) \right \rangle _{\mathcal{H} ^{1}}\text{d}s\right | \nonumber
\\&\le c\int_{0}^{T\wedge \tau _{M}}\mathcal{A} (\mathbb{E}\left \| u^{\varepsilon }(s) -u^{* }(s) \right \|^{2 }_{\mathcal{H} ^{1}})\text{d}s+C\sqrt{\mathcal{A} ((\left \| u_{0} ^{* } \right \|^{2}_{\mathcal{H} ^{1}}+\left \| u_{0} ^{* } \right \|^{6}_{\mathcal{H} ^{1}}+\left \| u_{0} ^{\varepsilon  } \right \|^{2}_{\mathcal{H} ^{1}}+1 )d^{\frac{1 }{2} })}
\\&~~~+C_{T}(\left \| u_{0} ^{\varepsilon } \right \|^{2}_{\mathcal{H} ^{1}}+\left \| u_{0} ^{\varepsilon } \right \|^{6}_{\mathcal{H} ^{1}}+\left \| u_{0} ^{* } \right \|^{2}_{\mathcal{H} ^{1}}+\left \| u_{0} ^{* } \right \|^{6}_{\mathcal{H} ^{1}}+1 )(\mathcal{R} _{3}(\frac{d}{\varepsilon } )+d+d^{\frac{1}{4} }).\nonumber
	\end{align*}
For $\mathcal{Z}_{2}$, analogous to the derivation presented in \eqref{s52}-\eqref{s58}, we can obtain the following from \textbf{(H5')} and  Lemma \ref{le1}:
\begin{align*}
	\mathcal{Z}_{2}&\le c\int_{0}^{T\wedge \tau _{M}}\mathcal{A} (\mathbb{E}\left \| u^{\varepsilon }(s) -u^{* }(s) \right \|^{2   }_{\mathcal{H} ^{1}})\text{d}s+C_{T}[\mathcal{A} ((\left \| u_{0} ^{* } \right \|^{2}_{\mathcal{H} ^{1}}+\left \| u_{0} ^{* } \right \|^{6}_{\mathcal{H} ^{1}}+1 )d^{\frac{1 }{2} })
\\&~~~+(\left \| u_{0} ^{* } \right \|^{2}_{\mathcal{H} ^{1}}+\left \| u_{0} ^{* } \right \|^{6}_{\mathcal{H} ^{1}}+1 )(d+\mathcal{R}  _{6}(\frac{d}{\varepsilon } ))].
	\end{align*}
Consequently, we derive
\begin{align*}
		 & \mathbb{E} (\underset{t\in[0,T]}{\sup}\left \| u^{\varepsilon}(t\wedge \tau _{M};u^{\varepsilon}_{0} )-u^{*}(t\wedge \tau _{M};u^{*}_{0}) \right \| _{\mathcal{H} ^{1}}^{2}) \nonumber
 \\&\le C_{M,T}[\left \| u^{\varepsilon}_{0}-u^{*}_{0} \right \| _{\mathcal{H} ^{1}}^{2}+C_{a_{1},a_{4},R,c}\int_{0}^{T}\mathcal{A} (\mathbb{E}\left \| u^{\varepsilon}(s\wedge \tau _{M})-u^{*}(s\wedge \tau _{M}) \right \| ^{2}_{\mathcal{H} ^{1}})\text{d}s
 \\&~~~+C_{T}[\sqrt{\mathcal{A} ((\left \| u_{0} ^{* } \right \|^{2}_{\mathcal{H} ^{1}}+\left \| u_{0} ^{* } \right \|^{6}_{\mathcal{H} ^{1}}+\left \| u_{0} ^{\varepsilon  } \right \|^{2}_{\mathcal{H} ^{1}}+1 )d^{\frac{1 }{2} })}+\mathcal{A} ((\left \| u_{0} ^{* } \right \|^{2}_{\mathcal{H} ^{1}}+\left \| u_{0} ^{* } \right \|^{6}_{\mathcal{H} ^{1}}+1 )d^{\frac{1 }{2} })]
 \\&~~~+C_{T}(\left \| u_{0}^{\varepsilon} \right \|_{\mathcal{H}^{1} }^{2}+\left \| u_{0}^{\varepsilon} \right \|_{\mathcal{H}^{1} }^{6}+\left \| u_{0} ^{* } \right \|^{2}_{\mathcal{H} ^{1}}+\left \| u_{0} ^{* } \right \|^{4}_{\mathcal{H} ^{1}}+\left \| u_{0} ^{* } \right \|^{6}_{\mathcal{H} ^{1}}+1 )(\sum_{i=1}^{5} \mathcal{R} _{i}(\frac{d}{\varepsilon } )+d+d^{\frac{1}{4} })].\nonumber
	\end{align*}
By employing a similar argument as presented in \eqref{c222}-\eqref{c999}, we obtain
\begin{align*}
\lim_{\varepsilon  \to 0} \mathbb{E}\underset{t\in [0,T]}{\sup}   \left \|u^{\varepsilon }(t;u_{0}^{\varepsilon }) -u^{* }(t;u_{0} ^{*}) \right \|^{2}_{\mathcal{H} ^{1}}=0.
\end{align*}
This completes the proof.
\end{proof}
\section*{Appendix \uppercase\expandafter{\romannumeral1}: The specific proof of  weak solutions(the step 1 of Theorem 3.1):}
\textbf{\emph{Proof:}} Let $\chi_{m}$ be a strong solution of the following equation:
\begin{align*}
\begin{cases}
 \text{d}\chi(t)=\mathcal{F}_{m}^{\mathcal{M} }(t,\chi)\text{d}t+\mathcal{G}_{m}^{\mathcal{M} }(t,\chi)\text{d}W^{k}(t), \\
\chi(0)=u^{k}_{0}.
\end{cases}
\end{align*}
For any fixed \( p\ge 2 \), by applying It$\hat{\text{o}} $ formula to $\left \| \chi_{m}(t) \right \| _{\mathcal{H}^{1}_{k}}^{p}$, we derive
\begin{align*}
\begin{split}
\left \| \chi_{m}(t) \right \| _{\mathcal{H}^{1}_{k}}^{p}&=\left \| u^{k}_{0} \right \| _{\mathcal{H}^{1}_{k}}^{p}+\int_{0}^{t} [ p\left \| \chi_{m}(s) \right \| _{\mathcal{H}^{1}_{k}}^{p-2}\left \langle \mathcal{F}^{\mathcal{M} }_{m} (s,\chi_{m}(s)),\chi_{m}(s)  \right \rangle_{\mathcal{H}^{1}_{k}} \\&~~~+\frac{p(p-1)}{2} \left \| \chi_{m}(s) \right \| _{\mathcal{H}^{1}_{k}}^{p-2} \left \| \mathcal{G}^{\mathcal{M} }_{m} (s,\chi_{m}(s)) \right \|_{\mathscr{L}(K^{k},\mathcal{H}^{1}_{k})}^{2} ]\text{d}s
\\&~~~+\int_{0}^{t}p\left \| \chi_{m}(s) \right \| _{\mathcal{H}^{1}_{k}}^{p-2} ( \chi_{m}(s))^{T} \mathcal{G}^{\mathcal{M} }_{m} (s,\chi_{m}(s))\text{d}W^{k}(s).
\end{split}
\end{align*}
By applying Young’s inequality in conjunction with \eqref{t6}, we obtain
\begin{align*}
\begin{split}
\left \| \chi_{m}(t) \right \| _{\mathcal{H}^{1}_{k}}^{p}&\le\left \| u^{k}_{0} \right \| _{\mathcal{H}^{1}_{k}}^{p}+C_{k,T,\nu,M,p}\int_{0}^{t} (\left \| \chi_{m}(s) \right \| _{\mathcal{H}^{1}_{k}}^{p}+1)\text{d}s
\\&~~~+C_{k,T,\nu,M,p}\int_{0}^{t}(\left \| \chi_{m}(s) \right \| _{\mathcal{H}^{1}_{k}}^{p}+1)\text{d}W^{k}(s).
\end{split}
\end{align*}
By applying Cauchy's inequality and the Burkholder-Davis-Gundy (B-D-G) inequality, we have
\begin{align*}
\begin{split}
&\mathbb{E} \underset{r\in[0,t]}{\sup}  \left \| \chi_{m}(r) \right \| _{\mathcal{H}^{1}_{k}}^{2p}\\&\le C_{p} \left \| u^{k}_{0} \right \| _{\mathcal{H}^{1}_{k}}^{2p}+C_{p,M}\mathbb{E}[\int_{0}^{t} (\underset{r\in[0,s]}{\sup}\left \| \chi_{m}(r) \right \| _{\mathcal{H}^{1}_{k}}^{p}+1)\text{d}s]^{2}
+C_{p,M}\mathbb{E}(\underset{r\in[0,t]}{\sup}\int_{0}^{r}(\left \| \chi_{m}(s) \right \| _{\mathcal{H}^{1}_{k}}^{p}+1)\text{d}W^{k}(s)) ^{2}
\\&\le C_{p} \left \| u^{k}_{0} \right \| _{\mathcal{H}^{1}_{k}}^{2p}+C_{p,M}\mathbb{E}\int_{0}^{t} (\underset{r\in[0,s]}{\sup}\left \| \chi_{m}(r) \right \| _{\mathcal{H}^{1}_{k}}^{2p}+1)\text{d}s
+C_{p,M}\int_{0}^{t} [\mathbb{E}\left \| \chi_{m}(r) \right \| _{\mathcal{H}^{1}_{k}}^{2p}+1]\text{d}s
\\&\le C_{p} \left \| u^{k}_{0} \right \| _{\mathcal{H}^{1}_{k}}^{2p}+C_{p,M}[\int_{0}^{t} \mathbb{E}\underset{r\in[0,s]}{\sup}  \left \| \chi_{m}(r) \right \| _{\mathcal{H}^{1}_{k}}^{2p}\text{d}s+t].
\end{split}
\end{align*}
The application of the Gronwall inequality subsequently yields
\begin{align}\label{r6}
\begin{split}
\mathbb{E} \underset{r\in[0,t]}{\sup}  \left \| \chi_{m}(r) \right \| _{\mathcal{H}^{1}_{k}}^{2p}\le C_{q,M}(\left \| u^{k}_{0} \right \| _{\mathcal{H}^{1}_{k}}^{2p}+t+e^{t})<\infty,
\end{split}
	\end{align}
for any \( t \in [0,T] \) and \( T > 0 \). Furthermore, we deduce from \eqref{t6} that $\mathcal{F}^{N}_{m} (s,\chi_{m}(s))$ and $\mathcal{G}^{N}_{m} (s,\chi_{m}(s))$ are bounded on $\mathcal{H}^{1}_{k}$. That is, there exists a constant \( M_{N,T} \), independent of \( m \), such that
\begin{align*}
\left \| \mathcal{F}^{\mathcal{M} }_{m} (t,\chi) \right \|_{\mathcal{H}^{1}_{k}}\vee   \left \| \mathcal{G}^{\mathcal{M} }_{m} (t,\chi)\right \|_{\mathscr{L}(K^{k},\mathcal{H}^{1}_{k})}\le M_{T,\mathcal{M} }.
\end{align*}
Consequently, for any $0\le r, t \le T < \infty  $, we have
\begin{align}\label{r7}
\begin{split}
&\underset{m\ge1}{\sup}\mathbb{E} \underset{r\in[0,t]}{\sup}  \left \| \chi_{m}(t)-\chi_{m}(r) \right \| _{\mathcal{H}^{1}_{k}}^{2p}\\&\le C_{p}\underset{m\ge1}{\sup}\mathbb{E} \left | \int_{r}^{t} \mathcal{F}^{\mathcal{M} }_{m} (s,\chi_{m}(s))\text{d}s \right | ^{2p}
+C_{p}\mathbb{E} \underset{m\ge1}{\sup}\left | \int_{r}^{t} \mathcal{G}^{\mathcal{M} }_{m} (s,\chi_{m}(s))\text{d}W^{k}(s) \right | ^{2p}
\\&\le C_{\mathcal{M} ,q,T}\left | t-r \right | ^{p}.
\end{split}
	\end{align}
This implies that the family of laws of $\chi_{m}(t)$ is weakly compact. Thus there exists a stochastic process $\chi(t)$ such that the law of $\chi_{m}(t)$ weakly converges to the law of $\chi(t)$. Note that $\chi_{m}(t)$ with the initial condition $\chi_{m}(0) =u_{0}^{k}$ is the unique strong solution of  \eqref{aaaaaa1}, which implies that
\begin{align*}
\mathfrak{M} _{m}(t)=\chi_{m}(t)-u_{0}^{k}-\int_{0}^{t} \mathcal{F}^{\mathcal{M} }_{m} (s,\chi_{m}(s))\text{d}s
\end{align*}
is a martingale with the covariance  given by $
\sum_{i=1}^{\mathcal{M}} \int_{0}^{t} [\mathcal{G}^{\mathcal{M} }_{m} (s,\chi_{m}(s))]_{ik}[\mathcal{G}^{\mathcal{M} }_{m} (s,\chi_{m}(s))]_{jk}\text{d}s$, for any $1\le i,j\le n$.
Then we define the following  coordinate process:
\begin{align*}
\chi^{*}(t)\omega =\omega (t),
\end{align*}
where $\omega\in C(\mathbb{R}^{+} ,\mathbb{R}^{n} )$, and let $\mathcal{D} _{t}=\sigma \{\omega (s); 0\le s \le t\}$, hence $x^{*}(t)\omega$ is $\mathcal{D} _{t}$-adapted. Next let
\begin{align*}
\mathfrak{M} _{*,m}(t)=\chi^{*}(t)-u_{0}^{k}-\int_{0}^{t} \mathcal{F}^{\mathcal{M} }_{m} (s,\chi^{*}(s))\text{d}s,
\end{align*}
then $\mathfrak{M} _{*,m}(t)$ is a martingale relative to $(\mathcal{L}_{\chi_{m}(t)},\mathcal{D} _{t})$ with the covariance
\begin{align*}
\left \langle \mathfrak{M} _{*,m}^{i},\mathfrak{M} _{*,m}^{j} \right \rangle(t) =\sum_{i=1}^{\mathcal{M}} \int_{0}^{t} [\mathcal{G}^{\mathcal{M} }_{m} (s,\chi^{*}(s))]_{ik}[\mathcal{G}^{\mathcal{M} }_{m} (s,\chi^{*}(s))]_{jk}\text{d}s.
\end{align*}
Further by  \eqref{r6}, let $m\to \infty $ and we obtain
\begin{align*}
\mathfrak{M} _{*}(t)=\chi^{*}(t)-u_{0}^{k}-\int_{0}^{t} \mathcal{F}^{\mathcal{M} } (s,\chi^{*}(s))\text{d}s.
\end{align*}
Then for any $t>s$ and $\Gamma \in \mathcal{D}_{t} $, by  Problem 2.4.12 of \cite{ref66}, we have
\begin{align*}
\int _{\mathbb{R}^{n} }\mathcal{X}_{\Gamma }\mathfrak{M} _{*}(t)\text{d}\mathcal{L} _{\chi^{*}(t)}&=\lim_{n \to \infty} \int _{\mathbb{R}^{n} }\mathcal{X}_{\Gamma }\mathfrak{M} _{*,m}(t)\text{d}\mathcal{L} _{\chi_{m}(t)}\\&=\lim_{n \to \infty} \int _{\mathbb{R}^{n} }\mathcal{X}_{\Gamma }\mathfrak{M} _{*,m}(s)\text{d}\mathcal{L} _{\chi_{m}(s)}\\&=\int _{\mathbb{R}^{n} }\mathcal{X}_{\Gamma }\mathfrak{M} _{*}(s)\text{d}\mathcal{L} _{\chi^{*}(s)},
\end{align*}
where $\mathcal{X}_{\Gamma }$ represents the indicator function of $\Gamma$. This implies that $\mathfrak{M} _{*}(t)$ is a $\mathcal{L} _{\chi^{*}(s)}$-martingale. In addition, by  \eqref{r6}, we get
\begin{align*}
\left \langle \mathfrak{M} _{*}^{i},\mathfrak{M} _{*}^{j} \right \rangle(t) =\sum_{i=1}^{\mathcal{M}} \int_{0}^{t} [\mathcal{G}^{\mathcal{M} } (s,\chi^{*}(s))]_{ik}[\mathcal{G}^{\mathcal{M} } (s,\chi^{*}(s))]_{jk}\text{d}s.
\end{align*}
for $1\le i,j\le n$. Based on Theorem II.7.1' of \cite{ref67}, it can be deduced that there exists an $m$-dimensional Brownian motion $B^{*}(t)$ on an extended probability space $(C(\mathbb{R}^{+} ,\mathbb{R}^{n} ),\mathcal{D} _{t},\mathcal{L}_{\chi^{*}(s)}))$ such that $
\mathfrak{M} _{*}(t)=\int_{0}^{t} \mathcal{G}^{\mathcal{M} } (s,\chi^{*}(s))\text{d}B^{*}(s)$,
i.e.,
\begin{align*}
\chi^{*}(t)=u^{k}_{0}+\int_{0}^{t} \mathcal{F}^{\mathcal{M} } (s,\chi^{*}(s))\text{d}s+\int_{0}^{t} \mathcal{G}^{\mathcal{M} } (s,\chi^{*}(s))\text{d}B^{*}(s),
\end{align*}
hence $\chi^{*}(t)$ is  a weak solution to \eqref{t10}. The proof is complete.\quad $\Box$


\section*{Acknowledgments}
 The first author (S. Lu)  supported by Graduate Innovation Fund of Jilin University. The second author (X. Yang) was supported by  National Natural Science Foundation of China (12071175, 12371191). The third author (Y. Li) was supported by  National Natural Science Foundation of China (12071175 and 12471183).
\section*{Data availability}
No data was used for the research described in the article.

\section*{References}
\bibliographystyle{plain}
\bibliography{ref}

\end{document}